\documentclass[10pt]{article}
\usepackage{amsmath}
\usepackage{amssymb}
\usepackage{amsthm}
\usepackage{mathrsfs}
\numberwithin{equation}{section}

\begin{document}

\title{Multilinear $\theta$-type Calder\'on--Zygmund operators and their commutators on products of weighted amalgam spaces}
\author{Xia Han and Hua Wang\footnote{E-mail address: 1044381894@qq.com, wanghua@pku.edu.cn.} \\
\textbf{Dedicated to the memory of Li Xue}.\\
\footnotesize{School of Mathematics and System Sciences, Xinjiang University, Urumqi 830046, P. R. China}}
\date{}
\maketitle

\begin{abstract}
In this paper, we first introduce several new classes of weighted amalgam spaces. Then we discuss both strong type and weak type estimates for certain multilinear $\theta$-type Calder\'on--Zygmund operators $T_\theta$ recently introduced in the literature on products of these spaces with multiple weights. Furthermore, the strong type and weak end-point estimates for both multilinear commutators and iterated commutators of $T_\theta$ and pointwise multiplication with BMO functions are established as well.\\
MSC(2020): 42B20; 42B25; 47B38; 47G10\\
Keywords: Multilinear $\theta$-type Calder\'on--Zygmund operators; multilinear commutators; iterated commutators; weighted amalgam spaces; multiple weights; Orlicz spaces
\end{abstract}
\tableofcontents

\section{Introduction}
\label{sec1}
In this paper, the symbols $\mathbb R$ and $\mathbb N$ stand for the sets of all real numbers and natural numbers, respectively. Let $\mathbb R^n$ be the $n$-dimensional Euclidean space of points $x=(x_1,x_2,\dots,x_n)$ with norm $|x|=(\sum_{i=1}^n x_i^2)^{1/2}$. Let $m\in\mathbb N$ and $(\mathbb R^n)^m=\overbrace{\mathbb R^n\times\cdots\times\mathbb R^n}^m$ be the $m$-fold product space. We use the standard notation $\mathcal{S}(\mathbb R^n)$ for the Schwartz space of test functions on $\mathbb R^n$ and $\mathcal{S}'(\mathbb R^n)$ its dual, the space of tempered distributions on $\mathbb R^n$. Calder\'on--Zygmund singular integral operators and their generalizations on the Euclidean space $\mathbb R^n$ have been extensively studied in the literature, see for example \cite{duoand,grafakos2,stein1,stein2,yabuta} for the standard theory. In particular, Yabuta \cite{yabuta} introduced certain $\theta$-type Calder\'on--Zygmund operators to facilitate his study of certain classes of pseudo-differential operators. Following the terminology of Yabuta \cite{yabuta}, we introduce the so-called $\theta$-type Calder\'on--Zygmund operators as follows.

\newtheorem{defn}{Definition}[section]

\begin{defn}
Let $\theta$ be a nonnegative, nondecreasing function on $\mathbb R^+:=(0,+\infty)$ with $0<\theta(1)<\infty$ and
\begin{equation*}
\int_0^1\frac{\theta(t)}{t}\,dt<\infty.
\end{equation*}
A measurable function $K(x,y)$ on $\mathbb R^n\times\mathbb R^n\setminus\{(x,y):x=y\}$ is said to be a $\theta$-type Calder\'on--Zygmund kernel, if there exists a constant $A>0$ such that
\begin{itemize}
  \item $\big|K(x,y)\big|\leq \frac{A}{|x-y|^{n}}$,\quad for any $x\neq y;$
  \item $\big|K(x,y)-K(z,y)\big|+\big|K(y,x)-K(y,z)\big|\leq \frac{A}{|x-y|^{n}}\cdot\theta\Big(\frac{|x-z|}{|x-y|}\Big)$, \quad for $|x-z|<\frac{|x-y|}{2}$.
\end{itemize}
\end{defn}

\begin{defn}
Let ${\mathcal T}_\theta$ be a linear operator from $\mathcal S(\mathbb R^n)$ into its dual ${\mathcal S}'(\mathbb R^n)$. We say that ${\mathcal T}_\theta$ is a $\theta$-type Calder\'on--Zygmund operator with associated kernel $K$, if
\begin{itemize}
  \item ${\mathcal T}_\theta$ can be extended to be a bounded linear operator on $L^2(\mathbb R^n);$
  \item for any $f\in C^\infty_0(\mathbb R^n)$ and for all $x\notin\mathrm{supp\,}f$, there is a $\theta$-type Calder\'on--Zygmund kernel $K(x,y)$ such that
\begin{equation*}
{\mathcal T}_\theta f(x):=\int_{\mathbb R^n}K(x,y)f(y)\,dy,
\end{equation*}
where $C^\infty_0(\mathbb R^n)$ is the space consisting of all infinitely differentiable functions on $\mathbb R^n$ with compact support.
\end{itemize}
\end{defn}
Note that the classical Calder\'on--Zygmund operator with standard kernel (see \cite{duoand,stein1}) is a special case of $\theta$-type operator ${\mathcal T}_{\theta}$ when $\theta(t)=t^{\delta}$ with $0<\delta\leq1$.

In 2009, Maldonado and Naibo \cite{ma} considered the bilinear $\theta$-type Calder\'on--Zygmund operators which are natural generalizations of the linear case, and established weighted norm inequalities for bilinear $\theta$-type Calder\'on--Zygmund operators on products of weighted Lebesgue spaces with Muckenhoupt weights. Moreover, they applied these operators to the study of certain paraproducts and bilinear pseudo-differential operators with mild regularity. Later, in 2014, Lu and Zhang \cite{lu} introduced the general $m$-linear $\theta$-type Calder\'on--Zygmund operators and their commutators for $m\geq2$, and established boundedness properties of these multilinear operators and multilinear commutators on products of weighted Lebesgue spaces with multiple weights. In addition, as applications, Lu and Zhang \cite{lu} obtained multiple-weighted norm inequalities for the paraproducts and bilinear pseudo-differential operators with mild regularity and their commutators as well. Following \cite{lu}, we now give the definition of the multilinear $\theta$-type Calder\'on--Zygmund operators.
\begin{defn}
Let $\theta$ be a nonnegative, nondecreasing function on $\mathbb R^+$ with $0<\theta(1)<\infty$ and
\begin{equation}\label{theta1}
\int_0^1\frac{\theta(t)}{t}\,dt<\infty.
\end{equation}
A measurable function $K(x,y_1,\ldots,y_m)$, defined away from the diagonal $x=y_1=\cdots=y_m$ in $(\mathbb R^n)^{m+1}$, is called an $m$-linear $\theta$-type Calder\'on--Zygmund kernel, if there exists a constant $A>0$ such that
\begin{itemize}
  \item for all $(x,y_1,\dots,y_m)\in(\mathbb R^n)^{m+1}$ with $x\neq y_k$ for some $k\in\{1,2,\dots,m\}$,
\begin{equation}\label{size}
\big|K(x,y_1,\ldots,y_m)\big|\leq\frac{A}{(|x-y_1|+\cdots+|x-y_m|)^{mn}}
\end{equation}
and
  \item for all $x,x'\in \mathbb R^n$,
\begin{equation}\label{regular1}
\begin{split}
&\big|K(x,y_1,\dots,y_m)-K(x',y_1,\dots,y_m)\big|\\
\leq& \frac{A}{(|x-y_1|+\cdots+|x-y_m|)^{mn}}\cdot\theta\bigg(\frac{|x-x'|}{|x-y_1|+\cdots+|x-y_m|}\bigg)
\end{split}
\end{equation}
whenever $|x-x'|\leq\frac{\,1\,}{2}\max_{1\leq i\leq m}|x-y_i|$, and
  \item for each fixed $k$ with $1\leq k\leq m$,
\begin{equation}\label{regular2}
\begin{split}
&\big|K(x,y_1,\dots,y_k,\dots,y_m)-K(x,y_1,\dots,y'_k,\dots,y_m)\big|\\
\leq& \frac{A}{(|x-y_1|+\cdots+|x-y_m|)^{mn}}\cdot\theta\bigg(\frac{|y_k-y'_k|}{|x-y_1|+\cdots+|x-y_m|}\bigg)
\end{split}
\end{equation}
whenever $|y_k-y'_k|\leq\frac{\,1\,}{2}\max_{1\leq i\leq m}|x-y_i|$.
\end{itemize}
\end{defn}

\begin{defn}
Let $m\in\mathbb N$ and $T_{\theta}$ be an $m$-linear operator initially defined on the $m$-fold product of Schwartz spaces
and taking values into the space of tempered distributions, i.e.,
\begin{equation*}
T_{\theta}:\overbrace{\mathcal S(\mathbb R^n)\times\cdots\times\mathcal S(\mathbb R^n)}^m\longrightarrow{\mathcal S}'(\mathbb R^n).
\end{equation*}
We say that $T_\theta$ is an $m$-linear $\theta$-type Calder\'on--Zygmund operator, if
\begin{itemize}
  \item $T_\theta$ can be extended to be a bounded multilinear operator from $L^{q_1}(\mathbb R^n)\times\cdots\times L^{q_m}(\mathbb R^n)$ into $L^q(\mathbb R^n)$ for some $q_1,\ldots,q_m\in[1,\infty)$ and $q\in[1/m,\infty)$ with $1/q=\sum_{k=1}^m 1/{q_k};$
  \item for any given $m$-tuples $\vec{f}=(f_1,\ldots,f_m)$, there is an $m$-linear $\theta$-type Calder\'on--Zygmund kernel $K(x,y_1,\dots,y_m)$(the conditions \eqref{size}, \eqref{regular1} and \eqref{regular2} are satisfied) such that
\begin{equation*}
\begin{split}
T_\theta(\vec{f})(x)&=T_\theta(f_1,\dots,f_m)(x)\\
&:=\int_{(\mathbb R^n)^m}K(x,y_1,\dots,y_m)f_1(y_1)\cdots f_m(y_m)\,dy_1\cdots dy_m
\end{split}
\end{equation*}
whenever $x\notin\bigcap_{k=1}^m\mathrm{supp\,}f_k$ and each $f_k\in C^\infty_0(\mathbb R^n)$ for $k=1,2,\dots,m$.
\end{itemize}
\end{defn}

\newtheorem{theorem}{Theorem}[section]

\newtheorem{corollary}{Corollary}[section]

\newtheorem{rek}{Remark}[section]

\newtheorem{lemma}{Lemma}[section]

We note that, if we simply take $\theta(t)=t^{\varepsilon}$ for some $0<\varepsilon\leq1$, then the multilinear $\theta$-type operator $T_{\theta}$ is exactly the multilinear Calder\'on--Zygmund singular integral operator, which was systematically studied by many authors. There is a vast literature of results of this nature, originated from the work of Grafakos and Torres \cite{grafakos}. The theory of multilinear Calder\'on--Zygmund singular integral operators plays an important role in harmonic analysis and other fields, this direction of research has been attracting a lot of attention in the last two decades, we refer the reader to \cite{grafakos,lerner,perez3} and the references therein for the standard theory of multilinear Calder\'on--Zygmund singular integrals. In 2014, Lu and Zhang \cite{lu} investigated weighted norm inequalities of multilinear $\theta$-type Calder\'on--Zygmund operators and their commutators with BMO functions. The following weighted strong-type and weak-type estimates of multilinear $\theta$-type Calder\'on--Zygmund operators on products of weighted Lebesgue spaces were given by Lu and Zhang in \cite[Theorem 1.2]{lu}.

\begin{theorem}[\cite{lu}]\label{strong}
Let $m\in\mathbb N$ and $T_{\theta}$ be an $m$-linear $\theta$-type Calder\'on--Zygmund operator with $\theta$ satisfying the condition \eqref{theta1}. If $p_1,\ldots,p_m\in(1,\infty)$ and $p\in(1/m,\infty)$ with $1/p=\sum_{k=1}^m 1/{p_k}$, and $\vec{w}=(w_1,\ldots,w_m)$ satisfies the multilinear $A_{\vec{P}}$ condition, then there exists a constant $C>0$ independent of $\vec{f}=(f_1,\ldots,f_m)$ such that
\begin{equation*}
\big\|T_{\theta}(\vec{f})\big\|_{L^p(\nu_{\vec{w}})}\leq C\prod_{k=1}^m\big\|f_k\big\|_{L^{p_k}(w_k)},
\end{equation*}
where $\nu_{\vec{w}}=\prod_{k=1}^m w_k^{p/{p_k}}$.
\end{theorem}

\begin{theorem}[\cite{lu}]\label{weak}
Let $m\in\mathbb N$ and $T_{\theta}$ be an $m$-linear $\theta$-type Calder\'on--Zygmund operator with $\theta$ satisfying the condition \eqref{theta1}. If $p_1,\ldots,p_m\in[1,\infty)$, $\min\{p_1,\ldots,p_m\}=1$ and $p\in[1/m,\infty)$ with $1/p=\sum_{k=1}^m 1/{p_k}$, and $\vec{w}=(w_1,\ldots,w_m)$ satisfies the multilinear $A_{\vec{P}}$ condition, then there exists a constant $C>0$ independent of $\vec{f}=(f_1,\ldots,f_m)$ such that
\begin{equation*}
\big\|T_{\theta}(\vec{f})\big\|_{WL^p(\nu_{\vec{w}})}\leq C\prod_{k=1}^m\big\|f_k\big\|_{L^{p_k}(w_k)},
\end{equation*}
where $\nu_{\vec{w}}=\prod_{k=1}^m w_k^{p/{p_k}}$. In particular, if $\vec{w}=(w_1,\ldots,w_m)$ satisfies the multilinear $A_{(1,\dots,1)}$ condition, then
\begin{equation*}
\big\|T_{\theta}(\vec{f})\big\|_{WL^{1/m}(\nu_{\vec{w}})}\leq C\prod_{k=1}^m\big\|f_k\big\|_{L^{1}(w_k)},
\end{equation*}
where $\nu_{\vec{w}}=\prod_{k=1}^m w_k^{1/{m}}$.
\end{theorem}

\begin{rek}
We remark that in the linear case $m=1$, the above weighted results were given by Quek and Yang in \cite{quek}. For the bilinear case $m=2$, Theorems \ref{strong} and \ref{weak} were proved by Maldonado and Naibo in \cite{ma} when some additional conditions imposed on $\theta$. And when $\theta(t)=t^{\varepsilon}$ for some $0<\varepsilon\leq1$, Theorems \ref{strong} and \ref{weak} were obtained by Lerner et al. \cite{lerner}.
\end{rek}

Next, we give the definition of the commutators for the multilinear $\theta$-type Calder\'on--Zygmund operator. Given a collection of locally integrable functions $\vec{b}=(b_1,\ldots,b_m)$, the $m$-linear commutator of $T_{\theta}$ with $\vec{b}$ is defined by
\begin{equation}\label{multicomm}
\begin{split}
\big[\Sigma\vec{b},T_\theta\big](\vec{f})(x)&=\big[\Sigma\vec{b},T_\theta\big](f_1,\ldots,f_m)(x)\\
&:=\sum_{k=1}^m\big[b_k,T_{\theta}\big]_{k}(f_1,\ldots,f_m)(x),
\end{split}
\end{equation}
where each term is the commutator of $b_k$ and $T_{\theta}$ in the $k$-th entry of $T_{\theta}$; that is
\begin{equation*}
\begin{split}
&\big[b_k,T_{\theta}\big]_{k}(f_1,\ldots,f_m)(x)\\
&=b_k(x)\cdot T_{\theta}(f_1,\ldots,f_k,\dots,f_m)(x)-T_{\theta}(f_1,\dots,b_kf_k,\dots,f_m)(x).
\end{split}
\end{equation*}
Then, at a formal level, we can see that
\begin{equation*}
\begin{split}
&\big[\Sigma\vec{b},T_\theta\big](\vec{f})(x)=\big[\Sigma\vec{b},T_\theta\big](f_1,\ldots,f_m)(x)\\
&=\int_{(\mathbb R^n)^m}\sum_{k=1}^m\big[b_k(x)-b_k(y_k)\big]K(x,y_1,\dots,y_m)f_1(y_1)\cdots f_m(y_m)\,dy_1\cdots dy_m.
\end{split}
\end{equation*}
Obviously, when $m=1$ in the above definition, this operator coincides with the linear commutator $[b,{\mathcal T}_{\theta}]$~(see \cite{liu,zhang2}), which is defined as
\begin{equation*}
[b,{\mathcal T}_{\theta}](f)=b{\mathcal T}_{\theta}(f)-{\mathcal T}_{\theta}(bf).
\end{equation*}
Let us now recall the definition of the space of $\mathrm{BMO}(\mathbb R^n)$(see \cite{duoand,john}). A locally integrable function $b(x)$ on $\mathbb R^n$ is said to be in $\mathrm{BMO}(\mathbb R^n)$ if it satisfies
\begin{equation*}
\sup_{B}\frac{1}{|B|}\int_B|b(x)-b_B|\,dx<\infty,
\end{equation*}
where the supremum is taken over all balls $B$ in $\mathbb R^n$, and $b_B$ stands for the average of $b$ over $B$, i.e.,
$b_B:=\frac{1}{|B|}\int_B b(y)\,dy$. The BMO norm of $b(x)$ is defined by
\begin{equation*}
\|b\|_*:=\sup_{B}\frac{1}{|B|}\int_B|b(x)-b_B|\,dx.
\end{equation*}
 In the multilinear setting, we say that $\vec{b}=(b_1,\ldots,b_m)\in \mathrm{BMO}^m$, if $b_k\in \mathrm{BMO}(\mathbb R^n)$ for all $k=1,2,\dots,m$. For convenience, we will use the following notation
\begin{equation*}
\big\|\vec{b}\big\|_{\mathrm{BMO}^m}:=\max_{1\leq k\leq m}\big\|b_k\big\|_{\ast},\quad\mbox{for} \; \vec{b}=(b_1,\ldots,b_m)\in \mathrm{BMO}^m.
\end{equation*}
In 2014, Lu and Zhang \cite[Theorems 1.3 and 1.4]{lu} also proved some weighted estimate and the $L\log L$-type estimate for multilinear commutators $\big[\Sigma\vec{b},T_\theta\big]$ defined in \eqref{multicomm} under a stronger condition \eqref{theta2} assumed on $\theta$, if $\vec{b}\in \mathrm{BMO}^m$.

\begin{theorem}[\cite{lu}]\label{comm}
Let $m\in\mathbb N$ and $\big[\Sigma\vec{b},T_\theta\big]$ be the $m$-linear commutator generated by $\theta$-type Calder\'on--Zygmund operator $T_{\theta}$ with $\vec{b}=(b_1,\ldots,b_m)\in \mathrm{BMO}^m;$ let $\theta$ satisfy
\begin{equation}\label{theta2}
\int_0^1\frac{\theta(t)\cdot(1+|\log t|)}{t}\,dt<\infty.
\end{equation}
If $p_1,\ldots,p_m\in(1,\infty)$ and $p\in(1/m,\infty)$ with $1/p=\sum_{k=1}^m 1/{p_k}$, and $\vec{w}=(w_1,\ldots,w_m)\in A_{\vec{P}}$, then there exists a constant $C>0$ independent of $\vec{b}$ and $\vec{f}=(f_1,\ldots,f_m)$ such that
\begin{equation*}
\big\|\big[\Sigma\vec{b},T_\theta\big](\vec{f})\big\|_{L^p(\nu_{\vec{w}})}\le C\cdot\big\|\vec{b}\big\|_{\mathrm{BMO}^m}\prod_{k=1}^m\big\|f_k\big\|_{L^{p_k}(w_k)},
\end{equation*}
where $\nu_{\vec{w}}=\prod_{k=1}^m w_k^{p/{p_k}}$.
\end{theorem}

\begin{theorem}[\cite{lu}]\label{Wcomm}
Let $m\in\mathbb N$ and $\big[\Sigma\vec{b},T_\theta\big]$ be the $m$-linear commutator generated by $\theta$-type Calder\'on--Zygmund operator $T_{\theta}$ with $\vec{b}=(b_1,\ldots,b_m)\in\mathrm{BMO}^m;$ let $\theta$ satisfy the condition \eqref{theta2}. If $p_k=1$, $k=1,2,\dots,m$ and $\vec{w}=(w_1,\ldots,w_m)\in A_{(1,\dots,1)}$, then for any given $\lambda>0$, there exists a constant $C>0$ independent of $\vec{f}=(f_1,\ldots,f_m)$ and $\lambda$ such that
\begin{equation*}
\begin{split}
&\nu_{\vec{w}}\Big(\Big\{x\in\mathbb R^n:\big|\big[\Sigma\vec{b},T_\theta\big](\vec{f})(x)\big|>\lambda^m\Big\}\Big)\\
&\leq C\cdot\Phi\big(\big\|\vec{b}\big\|_{\mathrm{BMO}^m}\big)^{1/m}
\prod_{k=1}^m\bigg(\int_{\mathbb R^n}\Phi\bigg(\frac{|f_k(x)|}{\lambda}\bigg)w_k(x)\,dx\bigg)^{1/m},
\end{split}
\end{equation*}
where $\nu_{\vec{w}}=\prod_{k=1}^m w_k^{1/{m}}$, $\Phi(t):=t\cdot(1+\log^+t)$ and $\log^+t:=\max\{\log t,0\}$.
\end{theorem}

\begin{rek}
As is well known, (multilinear) commutator has a greater degree of singularity than the underlying (multilinear) $\theta$-type Calder\'on--Zygmund operator, so more regular condition imposed on $\theta(t)$ is reasonable. Obviously, our condition \eqref{theta2} is slightly stronger than the condition \eqref{theta1}. For such type of commutators, the condition that $\theta(t)$ satisfying \eqref{theta2} is needed in the linear case (see \cite{liu,zhang2} for more details), so does in the multilinear case. Moreover, it is straightforward to check that when $\theta(t)=t^{\varepsilon}$ for some $\varepsilon>0$,
\begin{equation*}
\int_0^1\frac{t^{\varepsilon} \cdot(1+|\log t|)}{t}dt=\int_0^1 t^{\varepsilon-1}\cdot\left(1+\log\frac{\,1\,}{t}\right)dt<\infty.
\end{equation*}
Thus, the multilinear Calder\'on--Zygmund operator is also the multilinear $\theta$-type operator $T_{\theta}$ with $\theta(t)$ satisfying the condition \eqref{theta2}.
\end{rek}

\begin{rek}
When $m=1$, the above weighted endpoint estimate for the linear commutator $[b,\mathcal{T}_{\theta}]$ was given by Zhang and Xu in \cite{zhang2} (for the unweighted case, see \cite{liu}). Since $\mathcal{T}_{\theta}$ is bounded on $L^p(w)$ for $1<p<\infty$ and $w\in A_p$ as mentioned earlier, then by the well-known boundedness criterion for commutators of linear operators, which was obtained by Alvarez et al. in \cite{alvarez}, we know that $[b,\mathcal{T}_{\theta}]$ is also bounded on $L^p(w)$ for all $1<p<\infty$ and $w\in A_p$, whenever $b\in \mathrm{BMO}(\mathbb R^n)$. When $m\geq2$, $w_1=\cdots=w_m\equiv1$ and $\theta(t)=t^{\varepsilon}$ for some $\varepsilon>0$, P\'erez and Torres \cite{perez3} proved that if $\vec{b}=(b_1,\ldots,b_m)\in\mathrm{BMO}^m$, then
\begin{equation*}
\big[\Sigma\vec{b},T_\theta\big]:L^{p_1}(\mathbb R^n)\times\cdots\times L^{p_m}(\mathbb R^n)\longrightarrow L^p(\mathbb R^n)
\end{equation*}
for $1<p_k<\infty$ and $1<p<\infty$ with $1/p=1/{p_1}+\cdots+1/{p_m}$, where $k=1,2,\dots,m$. And when $m\geq2$ and $\theta(t)=t^{\varepsilon}$ for some $\varepsilon>0$, Theorems \ref{comm} and \ref{Wcomm} were obtained by Lerner et al. in \cite{lerner} (the range of $p$ is enlarged). Namely, Lerner et al.\cite{lerner} proved that if $\vec{b}=(b_1,\ldots,b_m)\in\mathrm{BMO}^m$ and $\vec{w}=(w_1,\ldots,w_m)\in A_{\vec{P}}$, then
\begin{equation*}
\big[\Sigma\vec{b},T_\theta\big]:L^{p_1}(w_1)\times\cdots\times L^{p_m}(w_m)\longrightarrow L^p(\nu_{\vec{w}})
\end{equation*}
for $1<p_k<\infty$ and $1/m<p<\infty$ with $1/p=1/{p_1}+\cdots+1/{p_m}$(the full range of $p$ is achieved), where $k=1,2,\dots,m$.
\end{rek}

\begin{rek}
Note that when $m=1$ and $\theta$ satisfies the condition \eqref{theta1}, the linear commutator $[b,\mathcal{T}_{\theta}]$ is bounded on $L^p(w)$ for all $1<p<\infty$ and $w\in A_p$. Thus, it is natural to ask whether the condition on $\theta(t)$ in Theorem $\ref{comm}$ can be weakened when $m\ge2$. Below, we will show that the conclusion of Theorem $\ref{comm}$ still holds provided that $\theta(t)$ only fulfills \eqref{theta1}. It should be pointed out that in the multilinear case $m\geq2$, the method used in this paper is different from the one in \cite{lu}. The idea of the proof is essentially that of \cite{alvarez,ding,perez3}.
\end{rek}
Motivated by \cite{perez} and \cite{lu}, we will consider another type of commutators on $\mathbb R^n$. Assume that $\vec{b}=(b_1,\dots,b_m)$ is a collection of locally integrable functions, we define the iterated commutator $\big[\Pi\vec{b},T_\theta\big]$ as
\begin{equation}\label{iteratedcomm}
\begin{split}
\big[\Pi\vec{b},T_\theta\big](\vec{f})(x)&=\big[\Pi\vec{b},T_\theta\big](f_1,\ldots,f_m)(x)\\
&:=[b_1,[b_2,\dots[b_{m-1},[b_m,T_{\theta}]_m]_{m-1}\dots]_2]_1(f_1,\ldots,f_m)(x),
\end{split}
\end{equation}
where
\begin{equation*}
\begin{split}
&\big[b_k,T_{\theta}\big]_{k}(f_1,\ldots,f_m)(x)\\
&=b_k(x)\cdot T_{\theta}(f_1,\ldots,f_k,\dots,f_m)(x)-T_{\theta}(f_1,\dots,b_kf_k,\dots,f_m)(x).
\end{split}
\end{equation*}
Then $\big[\Pi\vec{b},T_\theta\big]$ can be expressed in the following way:
\begin{align}\label{iteratedc}
&\big[\Pi\vec{b},T_\theta\big](\vec{f})(x)=\big[\Pi\vec{b},T_\theta\big](f_1,\ldots,f_m)(x)\\
&=\int_{(\mathbb R^n)^m}\prod_{k=1}^m\big[b_k(x)-b_k(y_k)\big]K(x,y_1,\dots,y_m)f_1(y_1)\cdots f_m(y_m)\,dy_1\cdots dy_m\notag.
\end{align}
Following the arguments used in \cite{perez} and \cite{lu} with some minor modifications, we can also establish the corresponding results (strong type and weak endpoint estimates) for iterated commutators $\big[\Pi\vec{b},T_\theta\big]$ defined in \eqref{iteratedcomm} under a stronger condition \eqref{theta3} imposed on $\theta$, if $\vec{b}\in \mathrm{BMO}^m$.

\begin{theorem}\label{commwh}
Let $m\in\mathbb N$ and $\big[\Pi\vec{b},T_\theta\big]$ be the iterated commutator generated by $\theta$-type Calder\'on--Zygmund operator $T_{\theta}$ and $\vec{b}=(b_1,\ldots,b_m)\in\mathrm{BMO}^m;$ let $\theta$ satisfy
\begin{equation}\label{theta3}
\int_0^1\frac{\theta(t)\cdot(1+|\log t|^m)}{t}\,dt<\infty.
\end{equation}
If $p_1,\ldots,p_m\in(1,\infty)$ and $p\in(1/m,\infty)$ with $1/p=\sum_{k=1}^m 1/{p_k}$, and $\vec{w}=(w_1,\ldots,w_m)\in A_{\vec{P}}$, then there exists a constant $C>0$ independent of $\vec{b}$ and $\vec{f}=(f_1,\ldots,f_m)$ such that
\begin{equation*}
\big\|\big[\Pi\vec{b},T_\theta\big](\vec{f})\big\|_{L^p(\nu_{\vec{w}})}\leq C\cdot\prod_{k=1}^m\big\|b_k\big\|_{*}\prod_{k=1}^m\big\|f_k\big\|_{L^{p_k}(w_k)},
\end{equation*}
where $\nu_{\vec{w}}=\prod_{k=1}^m w_k^{p/{p_k}}$.
\end{theorem}
\begin{rek}
Below, we will also show that the conclusion of Theorem $\ref{commwh}$ still holds provided that $\theta(t)$ only fulfills \eqref{theta1}.
\end{rek}

\begin{theorem}\label{Wcommwh}
Let $m\in\mathbb N$ and $\big[\Pi\vec{b},T_\theta\big]$ be the iterated commutator generated by $\theta$-type Calder\'on--Zygmund operator $T_{\theta}$ and $\vec{b}=(b_1,\ldots,b_m)\in\mathrm{BMO}^m;$ let $\theta$ satisfy the condition \eqref{theta3}.
If $p_k=1$, $k=1,2,\dots,m$ and $\vec{w}=(w_1,\ldots,w_m)\in A_{(1,\dots,1)}$, then for any given $\lambda>0$, there exists a constant $C>0$ independent of $\vec{f}=(f_1,\ldots,f_m)$ and $\lambda$ such that
\begin{equation*}
\begin{split}
&\nu_{\vec{w}}\Big(\Big\{x\in\mathbb R^n:\big|\big[\Pi\vec{b},T_\theta\big](\vec{f})(x)\big|>\lambda^m\Big\}\Big)\\
&\leq C\cdot\prod_{k=1}^m\bigg(\int_{\mathbb R^n}\Phi^{(m)}\bigg(\frac{|f_k(x)|}{\lambda}\bigg)w_k(x)\,dx\bigg)^{1/m},
\end{split}
\end{equation*}
where $\nu_{\vec{w}}=\prod_{k=1}^m w_k^{1/{m}}$, $\Phi(t):=t\cdot(1+\log^+t)$ and $\Phi^{(m)}:=\overbrace{\Phi\circ\cdots\circ\Phi}^m$.
\end{theorem}
\begin{rek}
It was proved in \cite{perez} that when $\theta(t)=t^{\varepsilon}$ for some $\varepsilon>0$, the estimate in Theorem \ref{Wcommwh} is sharp in the sense that $\Phi^{(m)}$ cannot be replaced by $\Phi^{(k)}$ for any $k<m$.
\end{rek}
The main purpose of this paper is to investigate the behaviour of multilinear $\theta$-type Calder\'on--Zygmund operators and their commutators
with BMO functions in the context of weighted amalgam spaces.

\section{Notations and preliminaries}
\label{sec2}
\subsection{Multiple weights}
We equip the $n$-dimensional Euclidean space $\mathbb R^n$ with the Euclidean norm $|\cdot|$ and the Lebesgue measure $dx$. For any $r>0$ and $y\in\mathbb R^n$, let $B(y,r)=\big\{x\in\mathbb R^n:|x-y|<r\big\}$ denote the open ball centered at $y$ with radius $r$, $B(y,r)^{\complement}=\mathbb R^n\backslash B(y,r)$ denote its complement and $|B(y,r)|$ be the Lebesgue measure of the ball $B(y,r)$. We also use the notation $\chi_{B(y,r)}$ to denote the characteristic function of $B(y,r)$. For some $\lambda>0$, $\lambda B$ stands for the ball concentric with $B$ and having radius $\lambda$ times as large.

A weight $\omega$ is a nonnegative locally integrable function on $\mathbb R^n$ that takes values in $(0,+\infty)$ almost everywhere. First we recall the Muckenhoupt $A_p$ weight classes. A weight $\omega$ is said to be in the class $A_p$ for $1<p<\infty$, if there exists a constant $C>0$ such that
\begin{equation*}
\bigg(\frac{1}{|B|}\int_B \omega(x)\,dx\bigg)^{1/p}\bigg(\frac{1}{|B|}\int_B \omega(x)^{-p'/p}\,dx\bigg)^{1/{p'}}\leq C
\end{equation*}
for every ball $B$ in $\mathbb R^n$. Here, and in what follows, $p'$ is the conjugate exponent of $p$ such that $1/p+1/{p'}=1$. For $p=1$, we say that $\omega$ is in the class $A_1$, if there exists a constant $C>0$ such that
\begin{equation*}
\frac{1}{|B|}\int_B\omega(x)\,dx\leq C\cdot\underset{x\in B}{\mbox{ess\,inf}}\;\omega(x)
\end{equation*}
for every ball $B$ in $\mathbb R^n$. Since the $A_p$ classes are increasing with respect to $p$, the $A_\infty$ class of weights is defined in a natural way by
\begin{equation*}
A_\infty:=\bigcup_{1\leq p<\infty}A_p.
\end{equation*}
Moreover, for $\omega\in A_{\infty}$, the following characterization is often used in applications: there are positive constants $C$ and $\delta$ such that for any ball $B$ and any measurable set $E$ contained in $B$,
\begin{equation}\label{compare}
\frac{\omega(E)}{\omega(B)}\leq C\left(\frac{|E|}{|B|}\right)^\delta,
\end{equation}
where $\omega(E):=\int_{E}\omega(x)\,dx$ for any given Lebesgue measurable set $E\subset\mathbb R^n$. We say that a weight $\omega$ satisfies the doubling condition, simply denoted by $\omega\in\Delta_2$, if there is an absolute constant $C>0$ such that
\begin{equation}\label{weights}
\omega(2B)\leq C\,\omega(B)
\end{equation}
holds for any ball $B$ in $\mathbb R^n$. If $\omega\in A_p$ with $1\leq p<\infty$ (or $\omega\in A_\infty$), then we know that $\omega\in\Delta_2$.

Recently, the theory of multiple weights adapted to multilinear Calder\'on--Zygmund singular integral operators was established by Lerner et al.in \cite{lerner}. New more refined multilinear maximal function was defined and used in \cite{lerner} to characterize the class of multiple $A_{\vec{P}}$ weights, and to obtain some weighted norm inequalities for multilinear Calder\'on--Zygmund operators and their commutators with BMO functions. Now let us recall the definition of multiple weights. For given $m$ exponents $p_1,\ldots,p_m\in[1,\infty)$, we will often write $\vec{P}$ for the vector $\vec{P}=(p_1,\ldots,p_m)$, and $p$ for the number given by $1/p:=\sum_{k=1}^m 1/{p_k}$ with $p\in[1/m,\infty)$. Given $\vec{w}=(w_1,\ldots,w_m)$, set
\begin{equation*}
\nu_{\vec{w}}:=\prod_{k=1}^m w_k^{p/{p_k}}.
\end{equation*}
We say that $\vec{w}$ satisfies the multilinear $A_{\vec{P}}$ condition if it satisfies
\begin{equation}\label{multiweight}
\sup_B\bigg(\frac{1}{|B|}\int_B \nu_{\vec{w}}(x)\,dx\bigg)^{1/p}\prod_{k=1}^m\bigg(\frac{1}{|B|}\int_B w_k(x)^{-p'_k/{p_k}}\,dx\bigg)^{1/{p'_k}}<\infty.
\end{equation}
When $p_k=1$, we denote $p'_k=\infty$, and the condition $\big(\frac{1}{|B|}\int_B w_k(x)^{-p'_k/{p_k}}\,dx\big)^{1/{p'_k}}$ in \eqref{multiweight} is understood as $\big(\inf_{x\in B}w_k(x)\big)^{-1}$. In particular, when each $p_k=1$, $k=1,2,\dots,m$, we denote $A_{\vec{1}}=A_{(1,\dots,1)}$. One can easily check that $A_{(1,\dots,1)}$ is contained in $A_{\vec{P}}$ for each $\vec{P}=(p_1,\ldots,p_m)$ with $1\leq p_k<\infty$, however, the classes $A_{\vec{P}}$ are not increasing with the natural partial order. It means that for two vectors $\vec{P}=(p_1,\ldots,p_m)$ and $\vec{Q}=(q_1,\ldots,q_m)$ with $p_k\leq q_k$, $k=1,2,\dots,m$, the following relation may not be true $A_{\vec{P}}\subset A_{\vec{Q}}$, see \cite[Remark 7.3]{lerner}. The multilinear Riesz transforms are typical examples of multilinear Calder\'on--Zygmund singular integral operators. It was shown in \cite[Theorem 3.11]{lerner} that the classes $A_{\vec{P}}$ are also characterized by the boundedness of all multilinear Riesz transforms on weighted Lebesgue spaces. Moreover, in general, the condition $\vec{w}\in A_{\vec{P}}$ does not imply $w_k\in L^1_{\mathrm{loc}}(\mathbb R^n)$ for any $1\leq k\leq m$ (see \cite[Remark 7.2]{lerner}), but instead we have the following characterization of the classes $A_{\vec{P}}$.
\begin{lemma}[\cite{lerner}]\label{multi}
Let $p_1,\ldots,p_m\in[1,\infty)$ and $1/p=\sum_{k=1}^m 1/{p_k}$. Then $\vec{w}=(w_1,\ldots,w_m)\in A_{\vec{P}}$ if and only if
\begin{equation}\label{multi2}
\left\{
\begin{aligned}
&\nu_{\vec{w}}\in A_{mp},\\
&w_k^{1-p'_k}\in A_{mp'_k},\quad k=1,2,\ldots,m,
\end{aligned}\right.
\end{equation}
where $\nu_{\vec{w}}=\prod_{k=1}^m w_k^{p/{p_k}}$ and the condition $w_k^{1-p'_k}\in A_{mp'_k}$ in the case $p_k=1$ is understood as $w_k^{1/m}\in A_1$.
\end{lemma}
\begin{rek}
Obviously, when $m=1$, $A_{\vec{P}}$ reduces to the classical $A_p$ class of Muckenhoupt.
Observe that in the linear case $m=1$ both conditions included in \eqref{multi2} represent the same $A_p$ condition. However, in the multilinear case $m\geq2$ neither of the conditions in \eqref{multi2} implies the other, see \cite[Remark 7.1]{lerner} for more details.
\end{rek}
Given a weight $\omega$ defined on $\mathbb R^n$, as usual, the weighted Lebesgue space $L^p(\omega)$ for $0<p<\infty$ is defined to be the set of all functions $f$ defined on $\mathbb R^n$ such that
\begin{equation*}
\big\|f\big\|_{L^p(\omega)}:=\bigg(\int_{\mathbb R^n}|f(x)|^p\omega(x)\,dx\bigg)^{1/p}<\infty.
\end{equation*}
We also denote by $WL^p(\omega)$~($0<p<\infty$), the weighted weak Lebesgue space consisting of all measurable functions $f$ defined on $\mathbb R^n$ such that
\begin{equation*}
\big\|f\big\|_{WL^p(\omega)}:=
\sup_{\lambda>0}\lambda\cdot\Big[\omega\big(\big\{x\in\mathbb R^n:|f(x)|>\lambda\big\}\big)\Big]^{1/p}<\infty.
\end{equation*}

\subsection{Orlicz spaces and Luxemburg norms}
Next we recall some basic definitions and facts from the theory of Orlicz spaces. For more information about these spaces, we refer the reader to the book \cite{rao}. A function $\mathcal A(t):[0,+\infty)\rightarrow[0,+\infty)$ is called a Young function if it is continuous, convex, strictly increasing and satisfies $\mathcal A(0)=0$ and $\mathcal A(t)\to+\infty$ as $t\to +\infty$. Given a Young function $\mathcal A$ and a ball $B$ in $\mathbb R^n$, we consider the $\mathcal A$-average of a function $f$ over a ball $B$ given by the following Luxemburg norm:
\begin{equation*}
\big\|f\big\|_{\mathcal A,B}
:=\inf\left\{\lambda>0:\frac{1}{|B|}\int_B\mathcal A\bigg(\frac{|f(x)|}{\lambda}\bigg)dx\leq1\right\}.
\end{equation*}
When $\mathcal A(t)=t^p$, $1\leq p<\infty$, it is easy to see that
\begin{equation*}
\big\|f\big\|_{\mathcal A,B}=\bigg(\frac{1}{|B|}\int_B\big|f(x)\big|^p\,dx\bigg)^{1/p};
\end{equation*}
that is, the Luxemburg norm coincides with the normalized $L^p$ norm. Associated to each Young function $\mathcal A$, one can define its complementary function $\bar{\mathcal A}$ by
\begin{equation*}
\bar{\mathcal A}(s):=\sup_{0\leq t<\infty}\big[st-\mathcal A(t)\big], \quad 0\leq s<\infty.
\end{equation*}
It can be proved that such $\bar{\mathcal A}$ is also a Young function. A direct computation shows that for all $t>0$,
\begin{equation*}
t\leq \mathcal{A}^{-1}(t)\bar{\mathcal A}^{-1}(t)\leq 2t.
\end{equation*}
Then the following generalized H\"older's inequality in Orlicz spaces holds for any given ball $B$ in $\mathbb R^n$:
\begin{equation*}
\frac{1}{|B|}\int_B\big|f(x)\cdot g(x)\big|\,dx\leq 2\big\|f\big\|_{\mathcal A,B}\big\|g\big\|_{\bar{\mathcal A},B}.
\end{equation*}
A particular case of interest, and especially in this paper, is the Young function $\Phi(t)=t\cdot(1+\log^+t)$, and we know that its complementary Young function is given by $\bar{\Phi}(t)\approx\exp(t)-1$. The corresponding averages will be denoted by
\begin{equation*}
\big\|f\big\|_{\Phi,B}=\big\|f\big\|_{L\log L,B}, \qquad\mbox{and}\qquad
\big\|g\big\|_{\bar{\Phi},B}=\big\|g\big\|_{\exp L,B}.
\end{equation*}
Thus, from the above generalized H\"older's inequality in Orlicz spaces, it follows that
\begin{equation}\label{holder}
\frac{1}{|B|}\int_B\big|f(x)\cdot g(x)\big|\,dx\leq 2\big\|f\big\|_{L\log L,B}\big\|g\big\|_{\exp L,B}.
\end{equation}
To obtain endpoint weak-type estimates for the multilinear commutators and iterated commutators on the product of weighted amalgam spaces, we need to define the weighted $\mathcal A$-average of a function $f$ over a ball $B$ by means of the mean Luxemburg norm; that is, given a Young function $\mathcal A$ and $\omega\in A_\infty$, we define (see \cite{rao,zhang})
\begin{equation*}
\big\|f\big\|_{\mathcal A(\omega),B}:=\inf\left\{\sigma>0:\frac{1}{\omega(B)}
\int_B\mathcal A\bigg(\frac{|f(x)|}{\sigma}\bigg)\cdot\omega(x)\,dx\leq1\right\}.
\end{equation*}
When $\mathcal A(t)=t$, we denote the mean Luxemburg norm with respect to $\omega$ by $\|\cdot\|_{L(\omega),B}$, and when $\Phi(t)=t\cdot(1+\log^+t)$, we denote this norm by $\|\cdot\|_{L\log L(\omega),B}$. The complementary Young function of $\Phi(t)$ is given by $\bar{\Phi}(t)\approx\exp(t)-1$ with the corresponding mean Luxemburg norm denoted by $\|\cdot\|_{\exp L(\omega),B}$. An analysis of the proof of the inequality \eqref{holder} reveals that the same conclusion still holds for the weighted case. For $\omega\in A_\infty$ and for every ball $B$ in $\mathbb R^n$, we can also show that the following generalized H\"older's inequality holds in the weighted setting (see \cite{zhang} for instance).
\begin{equation}\label{Wholder}
\frac{1}{\omega(B)}\int_B\big|f(x)\cdot g(x)\big|\omega(x)\,dx\leq C\big\|f\big\|_{L\log L(\omega),B}\big\|g\big\|_{\exp L(\omega),B}.
\end{equation}
This estimate will be used in the proofs of our main results.
\subsection{Weighted amalgam spaces}
Let $1\leq p,q\leq\infty$. A measurable function $f\in L^p_{\mathrm{loc}}(\mathbb R^n)$ is said to be in the Wiener amalgam space $(L^p,L^q)(\mathbb R^n)$ of $L^p(\mathbb R^n)$ and $L^q(\mathbb R^n)$(in the continuous case), if the function $y\mapsto\|f(\cdot)\cdot\chi_{B(y,1)}\|_{L^p}$ belongs to $L^q(\mathbb R^n)$. Define
\begin{equation*}
(L^p,L^q)(\mathbb R^n):=\left\{f:\big\|f\big\|_{(L^p,L^q)}
=\bigg(\int_{\mathbb R^n}\Big[\big\|f\cdot\chi_{B(y,1)}\big\|_{L^p}\Big]^qdy\bigg)^{1/q}<\infty\right\}.
\end{equation*}
It is easy to see that this space $(L^p,L^q)(\mathbb R^n)$ coincides with the usual Lebesgue space $L^{p}(\mathbb R^n)$ whenever $p=q$. The reader is also referred to the survey papers of Fournier--Stewart \cite{F} and Holland \cite{holland} for more information. In general, let $1\leq p,q,\alpha\leq\infty$. We define the amalgam space $(L^p,L^q)^{\alpha}(\mathbb R^n)$ of $L^p(\mathbb R^n)$ and $L^q(\mathbb R^n)$ as the set of all measurable functions $f$ satisfying $f\in L^p_{\mathrm{loc}}(\mathbb R^n)$ and $\big\|f\big\|_{(L^p,L^q)^{\alpha}}<\infty$, where
\begin{equation*}
\begin{split}
\big\|f\big\|_{(L^p,L^q)^{\alpha}}
:=&\sup_{r>0}\left\{\int_{\mathbb R^n}\Big[\big|B(y,r)\big|^{1/{\alpha}-1/p-1/q}\big\|f\cdot\chi_{B(y,r)}\big\|_{L^p}\Big]^qdy\right\}^{1/q}\\
=&\sup_{r>0}\Big\|\big|B(y,r)\big|^{1/{\alpha}-1/p-1/q}\big\|f\cdot\chi_{B(y,r)}\big\|_{L^p}\Big\|_{L^q},
\end{split}
\end{equation*}
with the usual modification when $p=\infty$ or $q=\infty$. This amalgam space(in the continuous case) arises naturally in harmonic analysis and was originally introduced by Fofana in \cite{fofana}. It turns out that the space $(L^p,L^q)^{\alpha}(\mathbb R^n)$ is closely related to Lebesgue and Morrey spaces. Many useful results in harmonic analysis (such as Fourier multipliers and boundedness properties of maximal operators and integral operators), well-known in the Lebesgue space $L^{p}(\mathbb R^n)$, have been extended within the framework of this space. As proved in \cite{fofana} the space $(L^p,L^q)^{\alpha}(\mathbb R^n)$ is nontrivial if and only if $p\leq\alpha\leq q$; thus in the remaining of the paper we will always assume that this condition $p\leq\alpha\leq q$ is fulfilled.

Note that
\begin{itemize}
  \item for $1\leq p\leq\alpha\leq q\leq\infty$, it is readily to see that $(L^p,L^q)^{\alpha}(\mathbb R^n)\subseteq(L^p,L^q)(\mathbb R^n)$ if we put $r=1$;
  \item if $1\leq p<\alpha$ and $q=\infty$, then $(L^p,L^q)^{\alpha}(\mathbb R^n)$ is just the classical Morrey space $\mathcal L^{p,\kappa}(\mathbb R^n)$, which is defined by (with $\kappa=1-p/{\alpha}$, see \cite{morrey,adams1})
\begin{equation*}
\mathcal L^{p,\kappa}(\mathbb R^n):=\Big\{f:\big\|f\big\|_{\mathcal L^{p,\kappa}}<\infty\Big\},
\end{equation*}
where
\begin{equation*}
\big\|f\big\|_{\mathcal L^{p,\kappa}}
:=\sup_{y\in\mathbb R^n,r>0}\bigg(\frac{1}{|B(y,r)|^\kappa}\int_{B(y,r)}|f(x)|^p\,dx\bigg)^{1/p};
\end{equation*}
  \item if $p=\alpha$ and $q=\infty$, then $(L^p,L^q)^{\alpha}(\mathbb R^n)$ reduces to the usual Lebesgue space $L^{p}(\mathbb R^n)$.
\end{itemize}
In \cite{wang}(see also \cite{wang2,wang3}), we considered a weighted version of the amalgam space $(L^p,L^q)^{\alpha}(\omega;\mu)$, and introduced some new classes of weighted amalgam spaces. Let us begin with the definitions of the weighted amalgam spaces. Moreover, in order to deal with the multilinear case $m\geq2$, we shall define these weighted spaces for all $0<p<\infty$.
\begin{defn}\label{amalgam}
Let $0<p\leq\alpha\leq q\leq\infty$, and let $\omega,\mu$ be two weights on $\mathbb R^n$. The weighted amalgam space $(L^p,L^q)^{\alpha}(\omega;\mu)$ is defined to be the set of all locally integrable functions $f$ such that
\begin{equation*}
\begin{split}
\big\|f\big\|_{(L^p,L^q)^{\alpha}(\omega;\mu)}
:=&\sup_{r>0}\left\{\int_{\mathbb R^n}\Big[\omega(B(y,r))^{1/{\alpha}-1/p-1/q}\big\|f\cdot\chi_{B(y,r)}\big\|_{L^p(\omega)}\Big]^q\mu(y)\,dy\right\}^{1/q}\\
=&\sup_{r>0}\Big\|\omega(B(y,r))^{1/{\alpha}-1/p-1/q}\big\|f\cdot\chi_{B(y,r)}\big\|_{L^p(\omega)}\Big\|_{L^q(\mu)}<\infty,
\end{split}
\end{equation*}
with the usual modification when $q=\infty$, where $\omega(B(y,r))=\int_{B(y,r)}\omega(x)\,dx$ is the weighted measure of $B(y,r)$.
\end{defn}
There is a routine procedure to define the weak space.
\begin{defn}\label{Wamalgam}
Let $0<p\leq\alpha\leq q\leq\infty$, and let $\omega,\mu$ be two weights on $\mathbb R^n$. The weighted weak amalgam space $(WL^p,L^q)^{\alpha}(\omega;\mu)$ is defined to be the set of all measurable functions $f$ for which
\begin{equation*}
\begin{split}
\big\|f\big\|_{(WL^p,L^q)^{\alpha}(\omega;\mu)}
:=&\sup_{r>0}\left\{\int_{\mathbb R^n}\Big[\omega(B(y,r))^{1/{\alpha}-1/p-1/q}\big\|f\cdot\chi_{B(y,r)}\big\|_{WL^p(\omega)}\Big]^q\mu(y)\,dy\right\}^{1/q}\\
=&\sup_{r>0}\Big\|\omega(B(y,r))^{1/{\alpha}-1/p-1/q}\big\|f\cdot\chi_{B(y,r)}\big\|_{WL^p(\omega)}\Big\|_{L^q(\mu)}<\infty,
\end{split}
\end{equation*}
with the usual modification when $q=\infty$.
\end{defn}
Note that
\begin{itemize}
 \item for $1\leq p\leq\alpha\leq q\leq\infty$, we can see that $(L^p,L^q)^{\alpha}(\omega;\mu)$ becomes a Banach function space with respect to the norm $\|\cdot\|_{(L^p,L^q)^{\alpha}(\omega;\mu)}$, and $(WL^p,L^q)^{\alpha}(\omega;\mu)$ becomes a quasi-Banach space with respect to the quasi-norm $\|\cdot\|_{(WL^p,L^q)^{\alpha}(\omega;\mu)}$;
 \item when $\mu\equiv1$, the weighted space $(L^p,L^q)^{\alpha}(\omega;\mu)$ was introduced by Feuto in \cite{feuto2} (see also \cite{feuto1,feuto3}), and when $\mu\equiv1$ and $p=1$, the weighted weak space $(WL^p,L^q)^{\alpha}(\omega;\mu)$ was also defined by Feuto in \cite{feuto2};
 \item if $1\leq p<\alpha$ and $q=\infty$, then $(L^p,L^q)^{\alpha}(\omega;\mu)$ is exactly the weighted Morrey space $\mathcal L^{p,\kappa}(\omega)$, which was first defined and studied by Komori and Shirai in \cite{komori}(with $\kappa=1-p/{\alpha}$).
\begin{equation*}
\begin{split}
&\mathcal L^{p,\kappa}(\omega):=\Big\{f :\big\|f\big\|_{\mathcal L^{p,\kappa}(\omega)}<\infty\Big\},
\end{split}
\end{equation*}
where
\begin{equation*}
\big\|f\big\|_{\mathcal L^{p,\kappa}(\omega)}
:=\sup_{y\in\mathbb R^n,r>0}\bigg(\frac{1}{\omega(B(y,r))^{\kappa}}\int_{B(y,r)}|f(x)|^p\omega(x)\,dx\bigg)^{1/p},
\end{equation*}
and $(WL^p,L^q)^{\alpha}(\omega;\mu)$ is exactly the weighted weak Morrey space $W\mathcal L^{p,\kappa}(\omega)$ defined by (with $\kappa=1-p/{\alpha}$, see \cite{wang1})
\begin{equation*}
\begin{split}
W\mathcal L^{p,\kappa}(\omega):=\Big\{f:\big\|f\big\|_{W\mathcal L^{p,\kappa}(\omega)}<\infty\Big\},
\end{split}
\end{equation*}
where
\begin{equation*}
\big\|f\big\|_{W\mathcal L^{p,\kappa}(\omega)}:=\sup_{y\in\mathbb R^n,r>0}\sup_{\lambda>0}\frac{1}{\omega(B(y,r))^{\kappa/p}}\lambda\cdot\Big[\omega\big(\big\{x\in B(y,r):|f(x)|>\lambda\big\}\big)\Big]^{1/p};
\end{equation*}
  \item if $p=\alpha$ and $q=\infty$, then $(L^p,L^q)^{\alpha}(\omega;\mu)$ reduces to the weighted Lebesgue space $L^{p}(\omega)$, and $(WL^p,L^q)^{\alpha}(\omega;\mu)$ reduces to the weighted weak Lebesgue space $WL^{p}(\omega)$.
\end{itemize}
In order to deal with the endpoint case of the commutators, we have to consider the following $L\log L$-type space.
Following \cite{wang}, we now introduce new weighted amalgam spaces of $L\log L$ type as follows.
\begin{defn}
Let $p=1$, $1\leq\alpha\leq q\leq\infty$, and let $\omega,\mu$ be two weights on $\mathbb R^n$. We denote by $(L\log L,L^q)^{\alpha}(\omega;\mu)$ the weighted amalgam space of $L\log L$ type, the space of all locally integrable functions $f$ defined on $\mathbb R^n$ with finite norm
$\big\|f\big\|_{(L\log L,L^q)^{\alpha}(\omega;\mu)}$.
\begin{equation*}
(L\log L,L^q)^{\alpha}(\omega;\mu):=\left\{f:\big\|f\big\|_{(L\log L,L^q)^{\alpha}(\omega;\mu)}<\infty\right\},
\end{equation*}
where
\begin{equation*}
\begin{split}
\big\|f\big\|_{(L\log L,L^q)^{\alpha}(\omega;\mu)}
:=&\sup_{r>0}\left\{\int_{\mathbb R^n}\Big[\omega(B(y,r))^{1/{\alpha}-1/q}\big\|f\big\|_{L\log L(\omega),B(y,r)}\Big]^q\mu(y)\,dy\right\}^{1/q}\\
=&\sup_{r>0}\Big\|\omega(B(y,r))^{1/{\alpha}-1/q}\big\|f\big\|_{L\log L(\omega),B(y,r)}\Big\|_{L^q(\mu)}.
\end{split}
\end{equation*}
\end{defn}
Observe that $t\leq t\cdot(1+\log^+t)$ for all $t>0$. Then for any given ball $B(y,r)\subset\mathbb R^n$ and $\omega\in A_\infty$, we have
\begin{equation*}
\big\|f\big\|_{L(\omega),B(y,r)}\leq \big\|f\big\|_{L\log L(\omega),B(y,r)}
\end{equation*}
by definition. This means that the following inequality (it can be seen as a generalized Jensen's inequality)
\begin{equation}\label{main esti1}
\big\|f\big\|_{L(\omega),B(y,r)}=\frac{1}{\omega(B(y,r))}\int_{B(y,r)}|f(x)|\cdot\omega(x)\,dx\leq\big\|f\big\|_{L\log L(\omega),B(y,r)}
\end{equation}
holds true for any ball $B(y,r)\subset\mathbb R^n$. Hence, for $1\leq\alpha\leq q\leq\infty$ and $\omega\in A_\infty$, we can further see the following inclusion relation from \eqref{main esti1}:
\begin{equation*}
(L\log L,L^q)^{\alpha}(\omega;\mu)\subset (L^1,L^q)^{\alpha}(\omega;\mu).
\end{equation*}

Recently, many works in classical harmonic analysis have been devoted to norm inequalities involving several integral operators in the setting of weighted amalgam spaces, see, for example \cite{feuto4,feuto1,feuto2,feuto3,wei}. These results obtained are extensions of well-known analogues in the weighted Lebesgue spaces. Inspired by the works mentioned above, it is therefore interesting to know the behavior of multilinear $\theta$-type Calder\'on--Zygmund operators and the corresponding commutators on products of weighted amalgam spaces, the aim of this paper is to give a positive answer to this problem. We are going to prove that multilinear $\theta$-type Calder\'on--Zygmund operators $T_{\theta}$ which are known to be bounded on products of weighted Lebesgue spaces with multiple weights, are also bounded from $(L^{p_1},L^{q_1})^{\alpha_1}(w_1;\mu)\times(L^{p_2},L^{q_2})^{\alpha_2}(w_2;\mu)\times\cdots
\times(L^{p_m},L^{q_m})^{\alpha_m}(w_m;\mu)$ into $(L^p,L^q)^{\alpha}(\nu_{\vec{w}};\mu)$ when $1<p_k<\infty$ for $k=1,2,\dots,m$, and bounded from $(L^{p_1},L^{q_1})^{\alpha_1}(w_1;\mu)\times(L^{p_2},L^{q_2})^{\alpha_2}(w_2;\mu)\times\cdots
\times(L^{p_m},L^{q_m})^{\alpha_m}(w_m;\mu)$ into $(WL^p,L^q)^{\alpha}(\nu_{\vec{w}};\mu)$ when $1\leq p_k<\infty$ for $k=1,2,\dots,m$ and at least one of the $p_k=1$. Moreover, the weighted strong-type and weak-type endpoint estimates for both types of multilinear commutators and iterated commutators in the context of weighted amalgam spaces are also obtained.

\section{Main results}
\label{sec3}
We will extend the results obtained in \cite{lu} for the $m$-linear $\theta$-type Calder\'on--Zygmund operators to the product of weighted amalgam spaces with multiple weights.
Our first two results on the boundedness properties of multilinear $\theta$-type Calder\'on--Zygmund operators are presented as follows.

\begin{theorem}\label{mainthm:1}
Let $m\geq 2$ and $T_{\theta}$ be an $m$-linear $\theta$-type Calder\'on--Zygmund operator with $\theta$ satisfying the condition \eqref{theta1}. Suppose that $1<p_k\leq\alpha_k<q_k<\infty$, $k=1,2,\ldots,m$ and $p\in(1/m,\infty)$ with $1/p=\sum_{k=1}^m 1/{p_k}$, $q\in[1,\infty)$ with $1/q=\sum_{k=1}^m 1/{q_k}$ and $1/{\alpha}=\sum_{k=1}^m 1/{\alpha_k};$ $\vec{w}=(w_1,\ldots,w_m)\in A_{\vec{P}}$ with $w_1,\ldots,w_m\in A_\infty$ and $\mu\in\Delta_2$.
In addition, suppose that
\begin{equation}\label{supp}
p_1\bigg(\frac{1}{\alpha_1}-\frac{1}{q_1}\bigg)=p_2\bigg(\frac{1}{\alpha_2}-\frac{1}{q_2}\bigg)
=\cdots=p_m\bigg(\frac{1}{\alpha_m}-\frac{1}{q_m}\bigg).
\end{equation}
Then there exists a constant $C>0$ such that for all $\vec{f}=(f_1,\ldots,f_m)\in(L^{p_1},L^{q_1})^{\alpha_1}(w_1;\mu)\times\cdots
\times(L^{p_m},L^{q_m})^{\alpha_m}(w_m;\mu)$,
\begin{equation*}
\big\|T_\theta(\vec{f})\big\|_{(L^p,L^q)^{\alpha}(\nu_{\vec{w}};\mu)}\leq C\prod_{k=1}^m\big\|f_k\big\|_{(L^{p_k},L^{q_k})^{\alpha_k}(w_k;\mu)}
\end{equation*}
with $\nu_{\vec{w}}=\prod_{k=1}^m w_k^{p/{p_k}}$.
\end{theorem}

\begin{theorem}\label{mainthm:2}
Let $m\geq 2$ and $T_{\theta}$ be an $m$-linear $\theta$-type Calder\'on--Zygmund operator with $\theta$ satisfying the condition \eqref{theta1}. Suppose that $1\leq p_k\leq\alpha_k<q_k<\infty$, $k=1,2,\ldots,m$, $\min\{p_1,\ldots,p_m\}=1$ and $p\in[1/m,\infty)$ with $1/p=\sum_{k=1}^m 1/{p_k}$, $q\in[1,\infty)$ with $1/q=\sum_{k=1}^m 1/{q_k}$ and $1/{\alpha}=\sum_{k=1}^m 1/{\alpha_k};$ $\vec{w}=(w_1,\ldots,w_m)\in A_{\vec{P}}$ with $w_1,\ldots,w_m\in A_\infty$ and $\mu\in\Delta_2$. In addition, suppose that
\begin{equation*}
p_1\bigg(\frac{1}{\alpha_1}-\frac{1}{q_1}\bigg)=p_2\bigg(\frac{1}{\alpha_2}-\frac{1}{q_2}\bigg)
=\cdots=p_m\bigg(\frac{1}{\alpha_m}-\frac{1}{q_m}\bigg).
\end{equation*}
Then there exists a constant $C>0$ such that for all $\vec{f}=(f_1,\ldots,f_m)\in(L^{p_1},L^{q_1})^{\alpha_1}(w_1;\mu)\times\cdots
\times(L^{p_m},L^{q_m})^{\alpha_m}(w_m;\mu)$,
\begin{equation*}
\big\|T_\theta(\vec{f})\big\|_{(WL^p,L^q)^{\alpha}(\nu_{\vec{w}};\mu)}\leq C\prod_{k=1}^m\big\|f_k\big\|_{(L^{p_k},L^{q_k})^{\alpha_k}(w_k;\mu)}
\end{equation*}
with $\nu_{\vec{w}}=\prod_{k=1}^m w_k^{p/{p_k}}$.
\end{theorem}

Our next theorem concerns norm inequalities for the multilinear commutator $\big[\Sigma\vec{b},T_\theta\big]$ with $\vec{b}\in \mathrm{BMO}^m$.
\begin{theorem}\label{mainthm:3}
Let $m\geq2$ and $\big[\Sigma\vec{b},T_\theta\big]$ be the $m$-linear commutator of $\theta$-type Calder\'on--Zygmund operator $T_{\theta}$ with $\theta$ satisfying the condition \eqref{theta1} and $\vec{b}\in\mathrm{BMO}^m$. Assume that $1<p_k\leq\alpha_k<q_k<\infty$, $k=1,2,\ldots,m$ and $p\in(1/m,\infty)$ with $1/p=\sum_{k=1}^m 1/{p_k}$, $q\in[1,\infty)$ with $1/q=\sum_{k=1}^m 1/{q_k}$ and $1/{\alpha}=\sum_{k=1}^m 1/{\alpha_k};$ $\vec{w}=(w_1,\ldots,w_m)\in A_{\vec{P}}$ with $w_1,\ldots,w_m\in A_\infty$ and $\mu\in\Delta_2$. Assume further that \eqref{supp} holds. Then there exists a constant $C>0$ such that for all $\vec{f}=(f_1,\ldots,f_m)\in(L^{p_1},L^{q_1})^{\alpha_1}(w_1;\mu)\times\cdots
\times(L^{p_m},L^{q_m})^{\alpha_m}(w_m;\mu)$,
\begin{equation*}
\big\|\big[\Sigma\vec{b},T_\theta\big](\vec{f})\big\|_{(L^p,L^q)^{\alpha}(\nu_{\vec{w}};\mu)}\leq C\prod_{k=1}^m\big\|f_k\big\|_{(L^{p_k},L^{q_k})^{\alpha_k}(w_k;\mu)}
\end{equation*}
with $\nu_{\vec{w}}=\prod_{k=1}^m w_k^{p/{p_k}}$.
\end{theorem}

For the endpoint case $p_1=\cdots=p_m=1$, we will also prove the following weak-type $L\log L$ estimate for the multilinear commutator $\big[\Sigma\vec{b},T_\theta\big]$ on the weighted amalgam spaces with multiple weights.

\begin{theorem}\label{mainthm:4}
Let $m\geq2$ and $\big[\Sigma\vec{b},T_\theta\big]$ be the $m$-linear commutator of $\theta$-type Calder\'on--Zygmund operator $T_{\theta}$ with $\theta$ satisfying the condition \eqref{theta2} and $\vec{b}\in \mathrm{BMO}^m$. Assume that $p_k=1$, $1\leq\alpha_k<q_k<\infty$, $k=1,2,\ldots,m$ and $p=1/m$, $q\in[1,\infty)$ with $1/q=\sum_{k=1}^m 1/{q_k}$ and $1/{\alpha}=\sum_{k=1}^m 1/{\alpha_k};$ $\vec{w}=(w_1,\ldots,w_m)\in A_{(1,\dots,1)}$ with $w_1,\ldots,w_m\in A_\infty$ and $\mu\in\Delta_2$. Assume further that \eqref{supp} holds. Then for any given $\lambda>0$ and any ball $B(y,r)\subset\mathbb R^n$ with $(y,r)\in\mathbb R^n\times(0,+\infty)$, there exists a constant $C>0$ independent of $\vec{f}=(f_1,\ldots,f_m)$, $B(y,r)$ and $\lambda$ such that
\begin{equation*}
\begin{split}
&\Big\|\nu_{\vec{w}}(B(y,r))^{1/{\alpha}-m-1/q}\cdot\Big[\nu_{\vec{w}}\Big(\Big\{x\in B(y,r):\big|\big[\Sigma\vec{b},T_\theta\big](\vec{f})(x)\big|>\lambda^m\Big\}\Big)\Big]^m\Big\|_{L^q(\mu)}\\
&\leq C\cdot
\prod_{k=1}^m\bigg\|\Phi\bigg(\frac{|f_k|}{\lambda}\bigg)\bigg\|_{(L\log L,L^{q_k})^{\alpha_k}(w_k;\mu)},
\end{split}
\end{equation*}
where $\nu_{\vec{w}}=\prod_{k=1}^m w_k^{1/{m}}$ and $\Phi(t):=t\cdot(1+\log^+t)$. Here the norm $\|\cdot\|_{L^q(\mu)}$ is taken with respect to the variable $y$, i.e.,
\begin{equation*}
\begin{split}
&\Big\|\nu_{\vec{w}}(B(y,r))^{1/{\alpha}-m-1/q}\cdot\Big[\nu_{\vec{w}}\Big(\Big\{x\in B(y,r):\big|\big[\Sigma\vec{b},T_\theta\big](\vec{f})(x)\big|>\lambda^m\Big\}\Big)\Big]^m\Big\|_{L^q(\mu)}\\
=&\left\{\int_{\mathbb R^n}\bigg[\nu_{\vec{w}}(B(y,r))^{1/{\alpha}-m-1/q}\cdot\Big[\nu_{\vec{w}}\Big(\Big\{x\in B(y,r):\big|\big[\Sigma\vec{b},T_\theta\big](\vec{f})(x)\big|>\lambda^m\Big\}\Big)\Big]^m\bigg]^q\mu(y)\,dy\right\}^{1/q}.
\end{split}
\end{equation*}
\end{theorem}

\begin{rek}
From the above definitions and Theorem \ref{mainthm:4}, we can roughly say that the multilinear commutator $\big[\Sigma\vec{b},T_\theta\big]$ is bounded from $(L\log L,L^{q_1})^{\alpha_1}(w_1;\mu)\times\cdots
\times(L\log L,L^{q_m})^{\alpha_m}(w_m;\mu)$ into $(WL^p,L^q)^{\alpha}(\nu_{\vec{w}};\mu)$ with $p=1/m$.
\end{rek}
\begin{rek}
It is worth pointing out that the results mentioned above were obtained by the second author in \cite{wang} when $m=1$. Moreover, in the discrete case, the boundedness of linear $\theta$-type Calder\'on--Zygmund operators on weighted Wiener amalgam spaces $(L^p,\ell^q_w)(\mathbb R^n)$ for $1<p,q<\infty$ was proved by Nakai et al.in \cite{ki}. When $\theta(t)=t^{\varepsilon}$ for some $\varepsilon>0$, the conclusions of Theorems \ref{mainthm:1}, \ref{mainthm:2} and \ref{mainthm:3} (in some special cases) have been obtained by Wang and Liu in \cite{wangpan} and \cite{wangpan1}.
\end{rek}
In what follows, the letter $C$ always stands for a positive constant independent of the main parameters and not necessarily the same at each occurrence. When a constant depends on some important parameters $\gamma_1,\gamma_2,\dots$, we denote it by $C(\gamma_1,\gamma_2,\dots)$. The symbol $\mathbf{X}\lesssim \mathbf{Y}$ means that there is a constant $C>0$ such that $\mathbf{X}\leq C\mathbf{Y}$. We use the symbol $\mathbf{X}\approx \mathbf{Y}$ to denote the equivalence of $\mathbf{X}$ and $\mathbf{Y}$; that is, there exist two positive constants $C_1$, $C_2$ independent of $\mathbf{X}$, $\mathbf{Y}$ such that $C_1 \mathbf{Y}\leq \mathbf{X}\leq C_2 \mathbf{Y}$.

\section{Proofs of Theorems \ref{mainthm:1} and \ref{mainthm:2}}
\label{sec4}
This section is concerned with the proofs of Theorems \ref{mainthm:1} and \ref{mainthm:2}.
Before proving the main theorems of this section, we first present the following important results without proof (see \cite{grafakos2} and \cite{duoand}).

\begin{lemma}[\cite{grafakos2}]\label{Min}
Let $\big\{f_k\big\}_{k=1}^N$ be a sequence of $L^p(\nu)$ functions with $0<p<\infty$ and $\nu\in A_\infty$. Then we have
\begin{equation*}
\Big\|\sum_{k=1}^N f_k\Big\|_{L^p(\nu)}\leq\mathcal{C}(p,N)\sum_{k=1}^N\big\|f_k\big\|_{L^p(\nu)},
\end{equation*}
where $\mathcal{C}(p,N)=\max\big\{1,N^{\frac{1-p}{p}}\big\}$. More specifically, $\mathcal{C}(p,N)=1$ for $1\leq p<\infty$, and $\mathcal{C}(p,N)=N^{\frac{1-p}{p}}$ for $0<p<1$.
\end{lemma}

\begin{lemma}[\cite{grafakos2}]\label{WMin}
Let $\big\{f_k\big\}_{k=1}^N$ be a sequence of $WL^p(\nu)$ functions with $0<p<\infty$ and $\nu\in A_\infty$. Then we have
\begin{equation*}
\Big\|\sum_{k=1}^N f_k\Big\|_{WL^p(\nu)}\leq\mathcal{C'}(p,N)\sum_{k=1}^N\big\|f_k\big\|_{WL^p(\nu)},
\end{equation*}
where $\mathcal{C'}(p,N)=\max\big\{N,N^{\frac{\,1\,}{p}}\big\}$. More specifically, $\mathcal{C'}(p,N)=N$ for $1\leq p<\infty$, and $\mathcal{C'}(p,N)=N^{\frac{\,1\,}{p}}$ for $0<p<1$.
\end{lemma}

\begin{lemma}[\cite{duoand}]\label{Ainfty}
Let $\omega\in A_\infty$. Then for any ball $\mathcal{B}$ in $\mathbb R^n$, the following reverse Jensen's formula holds.
\begin{equation*}
\int_{\mathcal{B}}\omega(x)\,dx\leq C|\mathcal{B}|\cdot\exp\bigg(\frac{1}{|\mathcal{B}|}\int_{\mathcal{B}}\log \omega(x)\,dx\bigg).
\end{equation*}
\end{lemma}
We are now in a position to prove Theorems $\ref{mainthm:1}$ and $\ref{mainthm:2}$.

\begin{proof}[Proof of Theorem $\ref{mainthm:1}$]
Let $1<p_k\leq\alpha_k<q_k<\infty$ and $\vec{f}=(f_1,\dots,f_m)$ be in $(L^{p_1},L^{q_1})^{\alpha_1}(w_1;\mu)\times\cdots
\times(L^{p_m},L^{q_m})^{\alpha_m}(w_m;\mu)$ with $(w_1,\dots,w_m)\in A_{\vec{P}}$ and $\mu\in\Delta_2$. We fix $y\in\mathbb R^n$ and $r>0$, and set $B=B(y,r)$ for the ball centered at $y$ and of radius $r$, $2B=B(y,2r)$. For any $1\leq k\leq m$, we decompose $f_k$ as
\begin{equation*}
f_k=f_k\cdot\chi_{2B}+f_k\cdot\chi_{(2B)^{\complement}}:=f^0_k+f^{\infty}_k;
\end{equation*}
then we write
\begin{equation*}
\begin{split}
\prod_{k=1}^m f_k(z_k)&=\prod_{k=1}^m\Big(f^0_k(z_k)+f^{\infty}_k(z_k)\Big)\\
&=\sum_{\beta_1,\ldots,\beta_m\in\{0,\infty\}}f^{\beta_1}_1(z_1)\cdots f^{\beta_m}_m(z_m)\\
&=\prod_{k=1}^m f^0_k(z_k)+\sum_{(\beta_1,\dots,\beta_m)\in\mathfrak{L}}f^{\beta_1}_1(z_1)\cdots f^{\beta_m}_m(z_m),
\end{split}
\end{equation*}
where
\begin{equation*}
\mathfrak{L}:=\big\{(\beta_1,\dots,\beta_m):\beta_k\in\{0,\infty\},\mbox{there is at least one $\beta_k\neq0$},1\leq k\leq m\big\};
\end{equation*}
that is, each term of $\sum$ contains at least one $\beta_k\neq0$. Since $T_{\theta}$ is an $m$-linear operator, then $T_\theta(\vec{f})$ can be written as
\begin{equation*}
T_\theta(\vec{f})=T_\theta(f^0_1,\dots,f^0_m)(x)
+\sum_{(\beta_1,\dots,\beta_m)\in\mathfrak{L}}T_\theta(f^{\beta_1}_1,\ldots,f^{\beta_m}_m)(x),
\end{equation*}
for an arbitrary fixed $x\in B(y,r)$. By using Lemma \ref{Min}($N=2^m$), we have
\begin{align}\label{I}
&\nu_{\vec{w}}(B(y,r))^{1/{\alpha}-1/p-1/q}\big\|T_\theta(\vec{f})\cdot\chi_{B(y,r)}\big\|_{L^p(\nu_{\vec{w}})}\notag\\
&=\nu_{\vec{w}}(B(y,r))^{1/{\alpha}-1/p-1/q}
\bigg(\int_{B(y,r)}\big|T_\theta(f_1,\dots,f_m)(x)\big|^p\nu_{\vec{w}}(x)\,dx\bigg)^{1/p}\notag\\
&\leq C\nu_{\vec{w}}(B(y,r))^{1/{\alpha}-1/p-1/q}
\bigg(\int_{B(y,r)}\big|T_\theta(f^0_1,\dots,f^0_m)(x)\big|^p\nu_{\vec{w}}(x)\,dx\bigg)^{1/p}\notag\\
&+\sum_{(\beta_1,\dots,\beta_m)\in\mathfrak{L}}C\nu_{\vec{w}}(B(y,r))^{1/{\alpha}-1/p-1/q}
\bigg(\int_{B(y,r)}\big|T_\theta(f^{\beta_1}_1,\ldots,f^{\beta_m}_m)(x)\big|^p\nu_{\vec{w}}(x)\,dx\bigg)^{1/p}\notag\\
&:=I^{0,\dots,0}(y,r)+\sum_{(\beta_1,\dots,\beta_m)\in\mathfrak{L}} I^{\beta_1,\dots,\beta_m}(y,r).
\end{align}
By the weighted strong-type estimate of $T_{\theta}$ (see Theorem \ref{strong}), we have
\begin{align}\label{I0}
I^{0,\dots,0}(y,r)&\leq C\cdot\nu_{\vec{w}}(B(y,r))^{1/{\alpha}-1/p-1/q}
\prod_{k=1}^m\bigg(\int_{B(y,2r)}|f_k(x)|^{p_k}w_k(x)\,dx\bigg)^{1/{p_k}}\notag\\
&=C\cdot\nu_{\vec{w}}(B(y,r))^{1/{\alpha}-1/p-1/q}\prod_{k=1}^m w_k(B(y,2r))^{1/{p_k}+1/{q_k}-1/{\alpha_k}}\notag\\
&\times\prod_{k=1}^m\bigg[w_k(B(y,2r))^{1/{\alpha_k}-1/{p_k}-1/{q_k}}
\big\|f_k\cdot\chi_{B(y,2r)}\big\|_{L^{p_k}(w_k)}\bigg].
\end{align}
Let $p_1,\ldots,p_m\in[1,\infty)$ and $p\in[1/m,\infty)$ with $1/p=\sum_{k=1}^m 1/{p_k}$. We first claim that under the assumptions of Theorem \ref{mainthm:1}(or Theorem \ref{mainthm:2}), the following result holds for any ball $\mathcal{B}=\mathcal{B}(y,r)$ with $y\in\mathbb R^n$ and $r>0$:
\begin{equation}\label{wanghua1}
\prod_{k=1}^m\bigg(\int_{\mathcal{B}} w_k(x)\,dx\bigg)^{p/{p_k}}
\lesssim \int_{\mathcal{B}}\nu_{\vec{w}}(x)\,dx,
\end{equation}
provided that $w_1,\ldots,w_m\in A_\infty$ and $\nu_{\vec{w}}=\prod_{k=1}^m w_k^{p/{p_k}}$. In fact, since $w_1,\ldots,w_m\in A_\infty$, by using Lemma \ref{Ainfty}, then we have
\begin{equation*}
\begin{split}
\prod_{k=1}^m\bigg(\int_{\mathcal{B}}w_k(x)\,dx\bigg)^{p/{p_k}}
&\leq C\prod_{k=1}^m\bigg[|\mathcal{B}|\cdot\exp\bigg(\frac{1}{|\mathcal{B}|}\int_{\mathcal{B}}\log w_k(x)\,dx\bigg)\bigg]^{p/{p_k}}\\
&=C\prod_{k=1}^m\bigg[|\mathcal{B}|^{p/{p_k}}\cdot\exp\bigg(\frac{1}{|\mathcal{B}|}\int_{\mathcal{B}}\log w_k(x)^{p/{p_k}}\,dx\bigg)\bigg]\\
&= C\cdot\big(|\mathcal{B}|\big)^{\sum_{k=1}^m p/{p_k}}\cdot
\exp\bigg(\sum_{k=1}^m\frac{1}{|\mathcal{B}|}\int_{\mathcal{B}}\log w_k(x)^{p/{p_k}}\,dx\bigg).
\end{split}
\end{equation*}
Note that
\begin{equation*}
\sum_{k=1}^m p/{p_k}=p\cdot\sum_{k=1}^m 1/{p_k}=1 \quad \mathrm{and} \quad \nu_{\vec{w}}(x)=\prod_{k=1}^m w_k(x)^{p/{p_k}}.
\end{equation*}
Thus, by Jensen's inequality, we obtain
\begin{equation*}
\begin{split}
\prod_{k=1}^m\bigg(\int_{\mathcal{B}}w_k(x)\,dx\bigg)^{p/{p_k}}
&\leq C\cdot|\mathcal{B}|\cdot\exp\bigg(\frac{1}{|\mathcal{B}|}\int_{\mathcal{B}}\log\nu_{\vec{w}}(x)\,dx\bigg)\\
&\leq C\int_{\mathcal{B}}\nu_{\vec{w}}(x)\,dx.
\end{split}
\end{equation*}
This gives \eqref{wanghua1}. Based on \eqref{wanghua1}, we can further show that under the assumption \eqref{supp}, the following result
\begin{equation}\label{C1}
\prod_{k=1}^m \bigg(\int_{\mathcal{B}(y,r)}w_k(x)\,dx\bigg)^{1/{p_k}+1/{q_k}-1/{\alpha_k}}
\lesssim\bigg(\int_{\mathcal{B}(y,r)}\nu_{\vec{w}}(x)\,dx\bigg)^{1/p+1/q-1/{\alpha}}
\end{equation}
holds for any ball $\mathcal{B}(y,r)$ in $\mathbb R^n$. In fact, since
\begin{equation*}
p_1\bigg(\frac{1}{\alpha_1}-\frac{1}{q_1}\bigg)=p_2\bigg(\frac{1}{\alpha_2}-\frac{1}{q_2}\bigg)
=\cdots=p_m\bigg(\frac{1}{\alpha_m}-\frac{1}{q_m}\bigg),
\end{equation*}
then the following equality certainly holds.
\begin{equation*}
p_1\bigg(\frac{1}{p_1}+\frac{1}{q_1}-\frac{1}{\alpha_1}\bigg)=p_2\bigg(\frac{1}{p_2}+\frac{1}{q_2}-\frac{1}{\alpha_2}\bigg)
=\cdots=p_m\bigg(\frac{1}{p_m}+\frac{1}{q_m}-\frac{1}{\alpha_m}\bigg).
\end{equation*}
Note that
\begin{equation*}
\frac{\,1\,}{p}=\sum_{k=1}^m\frac{1}{p_k}, \quad \frac{\,1\,}{q}=\sum_{k=1}^m\frac{1}{q_k}\quad \mbox{and} \quad\frac{\,1\,}{\alpha}=\sum_{k=1}^m \frac{1}{\alpha_k}.
\end{equation*}
A direct computation shows that
\begin{equation*}
p_k\bigg(\frac{1}{p_k}+\frac{1}{q_k}-\frac{1}{\alpha_k}\bigg)=p\bigg(\frac{\,1\,}{p}
+\frac{\,1\,}{q}-\frac{\,1\,}{\alpha}\bigg),\qquad k=1,2,\ldots,m,
\end{equation*}
or equivalently,
\begin{equation*}
\frac{1}{p_k}+\frac{1}{q_k}-\frac{1}{\alpha_k}=\frac{p}{p_k}\cdot\bigg(\frac{\,1\,}{p}
+\frac{\,1\,}{q}-\frac{\,1\,}{\alpha}\bigg)>0,\qquad k=1,2,\ldots,m.
\end{equation*}
This fact together with \eqref{wanghua1} implies that
\begin{equation*}
\begin{split}
&\prod_{k=1}^m \bigg(\int_{\mathcal{B}(y,r)}w_k(x)\,dx\bigg)^{1/{p_k}+1/{q_k}-1/{\alpha_k}}\\
&=\prod_{k=1}^m \bigg(\int_{\mathcal{B}(y,r)}w_k(x)\,dx\bigg)^{p/{p_k}\cdot(1/p+1/q-1/{\alpha})}\\
&=\bigg[\prod_{k=1}^m\bigg(\int_{\mathcal{B}(y,r)}w_k(x)\,dx\bigg)^{p/{p_k}}\bigg]^{1/p+1/q-1/{\alpha}}\\
&\lesssim\bigg(\int_{\mathcal{B}(y,r)}\nu_{\vec{w}}(x)\,dx\bigg)^{1/p+1/q-1/{\alpha}}.
\end{split}
\end{equation*}
Thus \eqref{C1} is proved. Substituting the inequality \eqref{C1} into \eqref{I0}, we thus obtain
\begin{equation*}
\begin{split}
I^{0,\dots,0}(y,r)
&\leq C\cdot\frac{\nu_{\vec{w}}(B(y,2r))^{1/p+1/q-1/{\alpha}}}{\nu_{\vec{w}}(B(y,r))^{1/p+1/q-1/{\alpha}}}\\
&\times\prod_{k=1}^m\bigg[w_k(B(y,2r))^{1/{\alpha_k}-1/{p_k}-1/{q_k}}
\big\|f_k\cdot\chi_{B(y,2r)}\big\|_{L^{p_k}(w_k)}\bigg].
\end{split}
\end{equation*}
In view of Lemma \ref{multi}, we have that $\nu_{\vec{w}}\in A_{mp}$ with $1/m<p<\infty$. Moreover, since $1/p+1/q-1/{\alpha}>0$, then by inequality \eqref{weights}, we get
\begin{equation}\label{doubling}
\frac{\nu_{\vec{w}}(B(y,2r))^{1/p+1/q-1/{\alpha}}}{\nu_{\vec{w}}(B(y,r))^{1/p+1/q-1/{\alpha}}}\leq C,
\end{equation}
from which we conclude that
\begin{equation}\label{I1yr}
\begin{split}
I^{0,\dots,0}(y,r)
&\leq C\cdot\prod_{k=1}^m\bigg[w_k(B(y,2r))^{1/{\alpha_k}-1/{p_k}-1/{q_k}}
\big\|f_k\cdot\chi_{B(y,2r)}\big\|_{L^{p_k}(w_k)}\bigg].
\end{split}
\end{equation}
To estimate the remaining terms in \eqref{I}, let us first consider the case when $\beta_1=\cdots=\beta_m=\infty$. We use a routine geometric observation(see Figure \ref{F} below for the bilinear case $m=2$):
\begin{equation*}
\overbrace{\big(\mathbb R^n\backslash B(y,2r)\big)\times\cdots\times\big(\mathbb R^n\backslash B(y,2r)\big)}^m
\subset
(\mathbb R^n)^m\backslash B(y,2r)^m,
\end{equation*}
and
\begin{equation*}
(\mathbb R^n)^m\backslash B(y,2r)^m=\bigcup_{j=1}^\infty B(y,2^{j+1}r)^m\backslash B(y,2^{j}r)^m,
\end{equation*}
where we have used the notation $E^m=\overbrace{E\times\cdots\times E}^m$ for a measurable set $E$ and a positive integer $m$.
\begin{figure}[htbp]
\centering
\setlength{\unitlength}{1mm}
\begin{picture}(60,70)(-20,8)
\linethickness{0.8pt}
\put(-40,40){\vector(1,0){78}}
\put(0,3){\vector(0,1){76}}
\put(-5,35){\line(1,0){10}}
\multiput(-7,35)(-3,0){12}{\line(1,0){2}}
\multiput(5,35)(3,0){11}{\line(1,0){2}}
\put(-5,45){\line(1,0){10}}
\multiput(-7,45)(-3,0){12}{\line(1,0){2}}
\multiput(5,45)(3,0){11}{\line(1,0){2}}
\put(-10,30){\line(1,0){20}}
\put(-10,50){\line(1,0){20}}
\put(-20,20){\line(1,0){40}}
\put(-20,60){\line(1,0){40}}
\put(-5,35){\line(0,1){10}}
\multiput(-5,45)(0,3){11}{\line(0,1){2}}
\multiput(-5,35)(0,-3){11}{\line(0,1){2}}
\put(5,35){\line(0,1){10}}
\multiput(5,45)(0,3){11}{\line(0,1){2}}
\multiput(5,35)(0,-3){11}{\line(0,1){2}}
\put(-10,30){\line(0,1){20}}
\put(10,30){\line(0,1){20}}
\put(-20,20){\line(0,1){40}}
\put(20,20){\line(0,1){40}}
\put(35,36.5){$\mathbb R^n$}
\put(1,77){$\mathbb R^n$}
\put(-4.5,37){$\scriptstyle 2B\times 2B$}
\put(-8,31){$\scriptstyle 2^2B\times 2^2B$}
\put(-18,23){$\scriptstyle 2^3B\times 2^3B$}
\thicklines
\end{picture}
\caption{$m=2$\label{F}}
\end{figure}
By the size condition \eqref{size} of the $\theta$-type Calder\'on--Zygmund kernel $K$, for any $x\in B(y,r)$, we obtain
\begin{align}\label{111}
&\big|T_{\theta}(f^\infty_1,\ldots,f^\infty_m)(x)\big|\notag\\
&\lesssim\int_{(\mathbb R^n)^m\backslash B(y,2r)^m}
\frac{|f_1(z_1)\cdots f_m(z_m)|}{(|x-z_1|+\cdots+|x-z_m|)^{mn}}\,dz_1\cdots dz_m\notag\\
&=\sum_{j=1}^\infty\int_{B(y,2^{j+1}r)^m\backslash B(y,2^{j}r)^m}
\frac{|f_1(z_1)\cdots f_m(z_m)|}{(|x-z_1|+\cdots+|x-z_m|)^{mn}}\,dz_1\cdots dz_m\notag\\
&\lesssim\sum_{j=1}^\infty\bigg(\frac{1}{|B(y,2^{j+1}r)|^m}
\int_{B(y,2^{j+1}r)^m\backslash B(y,2^{j}r)^m}\big|f_1(z_1)\cdots f_m(z_m)\big|dz_1\cdots dz_m\bigg)\notag\\
&\leq\sum_{j=1}^\infty\bigg(\frac{1}{|B(y,2^{j+1}r)|^m}\prod_{k=1}^m
\int_{B(y,2^{j+1}r)}\big|f_k(z_k)\big|\,dz_k\bigg)\notag\\
&=\sum_{j=1}^\infty\bigg(\prod_{k=1}^m\frac{1}{|B(y,2^{j+1}r)|}
\int_{B(y,2^{j+1}r)}\big|f_k(z_k)\big|\,dz_k\bigg),
\end{align}
where we have used the fact that $|x-z_1|+\cdots+|x-z_m|\approx2^{j+1}r\approx |B(y,2^{j+1}r)|^{1/n}$ when $x\in B(y,r)$ and $(z_1,\dots,z_m)\in B(y,2^{j+1}r)^m\backslash B(y,2^{j}r)^m$. Furthermore, by using H\"older's inequality, the multiple $A_{\vec{P}}$ condition on $\vec{w}$, we can deduce that
\begin{equation*}
\begin{split}
&\big|T_{\theta}(f^\infty_1,\ldots,f^\infty_m)(x)\big|\\
&\lesssim\sum_{j=1}^\infty\Bigg\{\prod_{k=1}^m
\frac{1}{|B(y,2^{j+1}r)|}\bigg(\int_{B(y,2^{j+1}r)}\big|f_k(z_k)\big|^{p_k}w_k(z_k)\,dz_k\bigg)^{1/{p_k}}
\bigg(\int_{B(y,2^{j+1}r)}w_k(z_k)^{-p'_k/{p_k}}\,dz_k\bigg)^{1/{p'_k}}\Bigg\}\\
&\lesssim\sum_{j=1}^\infty\Bigg\{\frac{1}{|B(y,2^{j+1}r)|^m}\cdot
\frac{|B(y,2^{j+1}r)|^{1/p+\sum_{k=1}^m(1-1/{p_k})}}{\nu_{\vec{w}}(B(y,2^{j+1}r))^{1/p}}
\prod_{k=1}^m\bigg(\big\|f_k\cdot\chi_{B(y,2^{j+1}r)}\big\|_{L^{p_k}(w_k)}\bigg)\Bigg\}\\
&=\sum_{j=1}^\infty\Bigg\{\frac{1}{\nu_{\vec{w}}(B(y,2^{j+1}r))^{1/p}}\cdot
\prod_{k=1}^m\big\|f_k\cdot\chi_{B(y,2^{j+1}r)}\big\|_{L^{p_k}(w_k)}\Bigg\},
\end{split}
\end{equation*}
where in the last step we have used the fact that $1/p+\sum_{i=1}^m(1-1/{p_i})=m$. Hence, from this pointwise estimate, we obtain
\begin{equation*}
\begin{split}
I^{\infty,\dots,\infty}(y,r)&\lesssim\nu_{\vec{w}}(B(y,r))^{1/{\alpha}-1/q}\\
&\times\sum_{j=1}^\infty\Bigg\{\frac{1}{\nu_{\vec{w}}(B(y,2^{j+1}r))^{1/p}}\cdot
\prod_{k=1}^m\big\|f_k\cdot\chi_{B(y,2^{j+1}r)}\big\|_{L^{p_k}(w_k)}\Bigg\}\\
&=\nu_{\vec{w}}(B(y,r))^{1/{\alpha}-1/q}
\times\sum_{j=1}^\infty\Bigg\{\frac{\prod_{k=1}^mw_k(B(y,2^{j+1}r))^{1/{p_k}+1/{q_k}-1/{\alpha_k}} }{\nu_{\vec{w}}(B(y,2^{j+1}r))^{1/p}}\\
&\prod_{k=1}^m\bigg[w_k(B(y,2^{j+1}r))^{1/{\alpha_k}-1/{p_k}-1/{q_k}}
\big\|f_k\cdot\chi_{B(y,2^{j+1}r)}\big\|_{L^{p_k}(w_k)}\bigg]\Bigg\}\\
&\lesssim\sum_{j=1}^\infty\Bigg\{\prod_{k=1}^m\bigg[w_k(B(y,2^{j+1}r))^{1/{\alpha_k}-1/{p_k}-1/{q_k}}
\big\|f_k\cdot\chi_{B(y,2^{j+1}r)}\big\|_{L^{p_k}(w_k)}\bigg]\\
&\times\frac{\nu_{\vec{w}}(B(y,r))^{1/{\alpha}-1/q}}{\nu_{\vec{w}}(B(y,2^{j+1}r))^{1/{\alpha}-1/q}}\Bigg\},
\end{split}
\end{equation*}
where in the last step we have used the estimate \eqref{C1}. We now consider the case where exactly $\ell$ of the $\beta_k$ are $\infty$ for some $1\leq\ell<m$. We only give the arguments for one of these cases. The rest are similar and can be easily obtained from the arguments below by permuting the indices. In this situation, by the same reason as above, we also have
\begin{equation*}
\overbrace{\big(\mathbb R^n\backslash B(y,2r)\big)\times\cdots\times\big(\mathbb R^n\backslash B(y,2r)\big)}^{\ell}
\subset
(\mathbb R^n)^{\ell}\backslash B(y,2r)^{\ell},
\end{equation*}
and
\begin{equation*}
(\mathbb R^n)^{\ell}\backslash B(y,2r)^{\ell}=\bigcup_{j=1}^\infty B(y,2^{j+1}r)^{\ell}\backslash B(y,2^{j}r)^{\ell},\quad 1\leq\ell<m.
\end{equation*}
Using the size condition \eqref{size} again, we deduce that for any $x\in B(y,r)$,
\begin{align}\label{112}
&\big|T_{\theta}(f^\infty_1,\ldots,f^\infty_\ell,f^0_{\ell+1},\ldots,f^0_m)(x)\big|\notag\\
&\lesssim\int_{(\mathbb R^n)^{\ell}\backslash B(y,2r)^{\ell}}\int_{B(y,2r)^{m-\ell}}
\frac{|f_1(z_1)\cdots f_m(z_m)|}{(|x-z_1|+\cdots+|x-z_m|)^{mn}}\,dz_1\cdots dz_m\notag\\
&\lesssim\prod_{k=\ell+1}^m\int_{B(y,2r)}\big|f_k(z_k)\big|\,dz_k\notag\\
&\times\sum_{j=1}^\infty\frac{1}{|B(y,2^{j+1}r)|^m}\int_{B(y,2^{j+1}r)^\ell\backslash B(y,2^{j}r)^\ell}
\big|f_1(z_1)\cdots f_{\ell}(z_\ell)\big|\,dz_1\cdots dz_\ell\notag\\
&\leq \prod_{k=\ell+1}^m\int_{B(y,2r)}\big|f_k(z_k)\big|\,dz_k
\times\sum_{j=1}^\infty\frac{1}{|B(y,2^{j+1}r)|^m}\prod_{k=1}^{\ell}
\int_{B(y,2^{j+1}r)}\big|f_k(z_k)\big|\,dz_k\notag\\
&\leq \sum_{j=1}^\infty\bigg(\prod_{k=1}^m\frac{1}{|B(y,2^{j+1}r)|}\int_{B(y,2^{j+1}r)}\big|f_k(z_k)\big|\,dz_k\bigg),
\end{align}
where in the last inequality we have used the inclusion relation $B(y,2r)\subseteq B(y,2^{j+1}r)$ for any $j\in \mathbb{N}$, and then we arrive at the same expression considered in the previous case. Hence, we can now argue exactly as we did in the estimation of $I^{\infty,\dots,\infty}(y,r)$ to get that for all $m$-tuples $(\beta_1,\dots,\beta_m)\in\mathfrak{L}$,
\begin{align}\label{I2yr}
I^{\beta_1,\dots,\beta_m}(y,r)
&\leq C\sum_{j=1}^\infty\Bigg\{\prod_{k=1}^m\bigg[w_k(B(y,2^{j+1}r))^{1/{\alpha_k}-1/{p_k}-1/{q_k}}
\big\|f_k\cdot\chi_{B(y,2^{j+1}r)}\big\|_{L^{p_k}(w_k)}\bigg]\notag\\
&\times\frac{\nu_{\vec{w}}(B(y,r))^{1/{\alpha}-1/q}}{\nu_{\vec{w}}(B(y,2^{j+1}r))^{1/{\alpha}-1/q}}\Bigg\}.
\end{align}
Furthermore, for given $\nu_{\vec{w}}\in A_{mp}\subset A_\infty$ with $1\leq mp<\infty$, it then follows directly from the inequality \eqref{compare} with exponent $\delta>0$ that
\begin{align}\label{psi1}
\frac{\nu_{\vec{w}}(B(y,r))^{1/{\alpha}-1/q}}{\nu_{\vec{w}}(B(y,2^{j+1}r))^{1/{\alpha}-1/q}}
&\leq C\left(\frac{|B(y,r)|}{|B(y,2^{j+1}r)|}\right)^{\delta(1/{\alpha}-1/q)}.
\end{align}
Therefore, by taking the $L^q({\mu})$-norm of both sides of \eqref{I}(with respect to the variable $y$), and then using Minkowski's inequality(note that $q\geq1$), \eqref{I1yr}, \eqref{I2yr} and \eqref{psi1}, we have
\begin{equation*}
\begin{split}
&\Big\|\nu_{\vec{w}}(B(y,r))^{1/{\alpha}-1/p-1/q}
\big\|T_\theta(\vec{f})\cdot\chi_{B(y,r)}\big\|_{L^p(\nu_{\vec{w}})}\Big\|_{L^q({\mu})}\\
&\leq\Big\|I^{0,\dots,0}(y,r)\Big\|_{L^q({\mu})}+
\sum_{(\beta_1,\dots,\beta_m)\in\mathfrak{L}}\Big\|I^{\beta_1,\dots,\beta_m}(y,r)\Big\|_{L^q({\mu})}\\
&\leq C\bigg\|\prod_{k=1}^m\bigg[w_k(B(y,2r))^{1/{\alpha_k}-1/{p_k}-1/{q_k}}
\big\|f_k\cdot\chi_{B(y,2r)}\big\|_{L^{p_k}(w_k)}\bigg]\bigg\|_{L^q({\mu})}\\
&+C\cdot2^m\sum_{j=1}^\infty\bigg\|\prod_{k=1}^m\bigg[w_k(B(y,2^{j+1}r))^{1/{\alpha_k}-1/{p_k}-1/{q_k}}
\big\|f_k\cdot\chi_{B(y,2^{j+1}r)}\big\|_{L^{p_k}(w_k)}\bigg]\bigg\|_{L^q({\mu})}\\
&\times\left(\frac{|B(y,r)|}{|B(y,2^{j+1}r)|}\right)^{\delta(1/{\alpha}-1/q)}.
\end{split}
\end{equation*}
Notice that
\begin{equation*}
\frac{\,1\,}{q}=\frac{1}{q_1}+\frac{1}{q_2}+\cdots+\frac{1}{q_m}.
\end{equation*}
A further application of H\"older's inequality leads to that
\begin{equation*}
\begin{split}
&\Big\|\nu_{\vec{w}}(B(y,r))^{1/{\alpha}-1/p-1/q}
\big\|T_\theta(\vec{f})\cdot\chi_{B(y,r)}\big\|_{L^p(\nu_{\vec{w}})}\Big\|_{L^q({\mu})}\\
&\leq C\prod_{k=1}^m\Big\|w_k(B(y,2r))^{1/{\alpha_k}-1/{p_k}-1/{q_k}}
\big\|f_k\cdot\chi_{B(y,2r)}\big\|_{L^{p_k}(w_k)}\Big\|_{L^{q_k}({\mu})}\\
&+C\sum_{j=1}^\infty\prod_{k=1}^m\Big\|w_k(B(y,2^{j+1}r))^{1/{\alpha_k}-1/{p_k}-1/{q_k}}
\big\|f_k\cdot\chi_{B(y,2^{j+1}r)}\big\|_{L^{p_k}(w_k)}\Big\|_{L^{q_k}({\mu})}\\
&\times\left(\frac{|B(y,r)|}{|B(y,2^{j+1}r)|}\right)^{\delta(1/{\alpha}-1/q)}\\
&\leq C\prod_{k=1}^m\big\|f_k\big\|_{(L^{p_k},L^{q_k})^{\alpha_k}(w_k;\mu)}+
C\prod_{k=1}^m\big\|f_k\big\|_{(L^{p_k},L^{q_k})^{\alpha_k}(w_k;\mu)}
\times\sum_{j=1}^\infty\left[\frac{1}{2^{(j+1)n}}\right]^{\delta(1/{\alpha}-1/q)}\\
&\leq C\prod_{k=1}^m\big\|f_k\big\|_{(L^{p_k},L^{q_k})^{\alpha_k}(w_k;\mu)},
\end{split}
\end{equation*}
where the last series is convergent since the exponent $\delta(1/{\alpha}-1/q)$ is positive.
Thus, by taking the supremum over all $r>0$, we complete the proof of Theorem \ref{mainthm:1}.
\end{proof}

\begin{proof}[Proof of Theorem $\ref{mainthm:2}$]
Let $1\leq p_k\leq\alpha_k<q_k<\infty$ and $\vec{f}=(f_1,\dots,f_m)$ be in $(L^{p_1},L^{q_1})^{\alpha_1}(w_1;\mu)\times\cdots
\times(L^{p_m},L^{q_m})^{\alpha_m}(w_m;\mu)$ with $(w_1,\dots,w_m)\in A_{\vec{P}}$ and $\mu\in\Delta_2$. For an arbitrary ball $B=B(y,r)\subset\mathbb R^n$ with $y\in\mathbb R^n$ and $r>0$, we decompose $f_k$ as
\begin{equation*}
f_k=f_k\cdot\chi_{2B}+f_k\cdot\chi_{(2B)^{\complement}}:=f^0_k+f^{\infty}_k,\quad \mbox{for}~~k=1,2,\dots,m;
\end{equation*}
then by Lemma \ref{WMin}($N=2^m$), one can write
\begin{align}\label{Iprime}
&\nu_{\vec{w}}(B(y,r))^{1/{\alpha}-1/p-1/q}\big\|T_\theta(\vec{f})\cdot\chi_{B(y,r)}\big\|_{WL^p(\nu_{\vec{w}})}\notag\\
&=\nu_{\vec{w}}(B(y,r))^{1/{\alpha}-1/p-1/q}
\big\|T_\theta(f_1,\dots,f_m)\cdot\chi_{B(y,r)}\big\|_{WL^p(\nu_{\vec{w}})} \notag\\
&\leq C\cdot\nu_{\vec{w}}(B(y,r))^{1/{\alpha}-1/p-1/q}
\big\|T_\theta(f^0_1,\dots,f^0_m)\cdot\chi_{B(y,r)}\big\|_{WL^p(\nu_{\vec{w}})}\notag\\
&+C\sum_{(\beta_1,\dots,\beta_m)\in\mathfrak{L}}\nu_{\vec{w}}(B(y,r))^{1/{\alpha}-1/p-1/q}
\big\|T_\theta(f^{\beta_1}_1,\ldots,f^{\beta_m}_m)\cdot\chi_{B(y,r)}\big\|_{WL^p(\nu_{\vec{w}})}\notag\\
&:=I^{0,\dots,0}_\ast(y,r)+\sum_{(\beta_1,\dots,\beta_m)\in\mathfrak{L}} I^{\beta_1,\dots,\beta_m}_\ast(y,r),
\end{align}
where
\begin{equation*}
\mathfrak{L}=\big\{(\beta_1,\dots,\beta_m):\beta_k\in\{0,\infty\},\mbox{there is at least one $\beta_k\neq0$},1\leq k\leq m\big\}.
\end{equation*}
By the weighted weak-type estimate of $T_{\theta}$ (see Theorem \ref{weak}) and \eqref{C1}, we have
\begin{align}\label{II0}
I^{0,\dots,0}_\ast(y,r)&\leq C\cdot\nu_{\vec{w}}(B(y,r))^{1/{\alpha}-1/p-1/q}
\prod_{k=1}^m\bigg(\int_{B(y,2r)}|f_k(x)|^{p_k}w_k(x)\,dx\bigg)^{1/{p_k}}\notag\\
&=C\cdot\nu_{\vec{w}}(B(y,r))^{1/{\alpha}-1/p-1/q}\prod_{k=1}^m w_k(B(y,2r))^{1/{p_k}+1/{q_k}-1/{\alpha_k}}\notag\\
&\times\prod_{k=1}^m\bigg[w_k(B(y,2r))^{1/{\alpha_k}-1/{p_k}-1/{q_k}}
\big\|f_k\cdot\chi_{B(y,2r)}\big\|_{L^{p_k}(w_k)}\bigg]\notag\\
&\leq C\cdot\frac{\nu_{\vec{w}}(B(y,2r))^{1/p+1/q-1/{\alpha}}}{\nu_{\vec{w}}(B(y,r))^{1/p+1/q-1/{\alpha}}}\notag\\
&\times\prod_{k=1}^m\bigg[w_k(B(y,2r))^{1/{\alpha_k}-1/{p_k}-1/{q_k}}
\big\|f_k\cdot\chi_{B(y,2r)}\big\|_{L^{p_k}(w_k)}\bigg].
\end{align}
Moreover, in view of Lemma \ref{multi} again, we also have $\nu_{\vec{w}}\in A_{mp}$ with $1/m\leq p<\infty$. Since $1/p+1/q-1/{\alpha}>0$, then we apply inequality \eqref{weights} to obtain that
\begin{equation}\label{doubling2}
\frac{\nu_{\vec{w}}(B(y,2r))^{1/p+1/q-1/{\alpha}}}{\nu_{\vec{w}}(B(y,2r))^{1/p+1/q-1/{\alpha}}}\leq C,
\end{equation}
from which we conclude that
\begin{equation}\label{WI1yr}
\begin{split}
I^{0,\dots,0}_\ast(y,r)
&\leq C\cdot\prod_{k=1}^m\bigg[w_k(B(y,2r))^{1/{\alpha_k}-1/{p_k}-1/{q_k}}
\big\|f_k\cdot\chi_{B(y,2r)}\big\|_{L^{p_k}(w_k)}\bigg].
\end{split}
\end{equation}
In the proof of Theorem \ref{mainthm:1}, we have already established the following pointwise estimate for all $m$-tuples $(\beta_1,\dots,\beta_m)\in\mathfrak{L}$ (see \eqref{111} and \eqref{112}).
\begin{equation}
\big|T_{\theta}(f^{\beta_1}_1,\ldots,f^{\beta_m}_m)(x)\big|
\lesssim\sum_{j=1}^\infty\bigg(\prod_{k=1}^m\frac{1}{|B(y,2^{j+1}r)|}\int_{B(y,2^{j+1}r)}\big|f_k(z_k)\big|\,dz_k\bigg).
\end{equation}
Without loss of generality, we may assume that
\begin{equation*}
p_1=\cdots=p_{\ell}=\min\{p_1,\ldots,p_m\}=1\quad \mbox{and} \quad p_{\ell+1},\ldots,p_m>1
\end{equation*}
with $1\leq\ell<m$. The case that $p_1=\cdots=p_m=1$ can be dealt with similarly and more easily. All the estimates given below continue to hold for $p_1=\cdots=p_m=1$. Using H\"older's inequality and the multiple $A_{\vec{P}}$ condition on $\vec{w}$, we obtain that for any $x\in B(y,r)$,
\begin{equation*}
\begin{split}
&\big|T_{\theta}(f^{\beta_1}_1,\ldots,f^{\beta_m}_m)(x)\big|\\
&\lesssim\sum_{j=1}^\infty\bigg(\prod_{k=1}^{\ell}\frac{1}{|B(y,2^{j+1}r)|}
\int_{B(y,2^{j+1}r)}\big|f_k(z_k)\big|\,dz_k\bigg)\\
&\times
\bigg(\prod_{k=\ell+1}^{m}\frac{1}{|B(y,2^{j+1}r)|}\int_{B(y,2^{j+1}r)}\big|f_k(z_k)\big|\,dz_k\bigg)\\
&\lesssim\sum_{j=1}^\infty\prod_{k=1}^{\ell}\frac{1}{|B(y,2^{j+1}r)|}
\int_{B(y,2^{j+1}r)}\big|f_k(z_k)\big|w_k(z_k)\,dz_k
\left(\inf_{z_k\in B(y,2^{j+1}r)}w_k(z_k)\right)^{-1}\\
&\times\prod_{k=\ell+1}^{m}\frac{1}{|B(y,2^{j+1}r)|}
\bigg(\int_{B(y,2^{j+1}r)}\big|f_k(z_k)\big|^{p_k}w_k(z_k)\,dz_k\bigg)^{1/{p_k}}
\bigg(\int_{B(y,2^{j+1}r)}w_k(z_k)^{-p'_k/{p_k}}\,dz_k\bigg)^{1/{p'_k}}\\
&\lesssim\sum_{j=1}^\infty\Bigg\{\frac{1}{\nu_{\vec{w}}(B(y,2^{j+1}r))^{1/p}}\cdot
\prod_{k=1}^m\big\|f_k\cdot\chi_{B(y,2^{j+1}r)}\big\|_{L^{p_k}(w_k)}\Bigg\}.
\end{split}
\end{equation*}
Observe that $\nu_{\vec{w}}\in A_{mp}$ with $1\le mp<\infty$. Thus, it follows directly from Chebyshev's inequality and the above pointwise estimate that
\begin{align}\label{WI2yr}
&I^{\beta_1,\dots,\beta_m}_\ast(y,r)\notag\\
&\leq C\cdot\nu_{\vec{w}}(B(y,r))^{1/{\alpha}-1/p-1/q}
\bigg(\int_{B(y,r)}\big|T_\theta(f^{\beta_1}_1,\ldots,f^{\beta_m}_m)(x)\big|^p\nu_{\vec{w}}(x)\,dx\bigg)^{1/p}\notag\\
&\leq C\cdot\nu_{\vec{w}}(B(y,r))^{1/{\alpha}-1/q}
\sum_{j=1}^\infty\Bigg\{\frac{1}{\nu_{\vec{w}}(B(y,2^{j+1}r))^{1/p}}\cdot
\prod_{k=1}^m\big\|f_k\cdot\chi_{B(y,2^{j+1}r)}\big\|_{L^{p_k}(w_k)}\Bigg\}\notag\\
&\leq C\sum_{j=1}^\infty\Bigg\{\prod_{k=1}^m\bigg[w_k(B(y,2^{j+1}r))^{1/{\alpha_k}-1/{p_k}-1/{q_k}}
\big\|f_k\cdot\chi_{B(y,2^{j+1}r)}\big\|_{L^{p_k}(w_k)}\bigg]\notag\\
&\times\frac{\nu_{\vec{w}}(B(y,r))^{1/{\alpha}-1/q}}{\nu_{\vec{w}}(B(y,2^{j+1}r))^{1/{\alpha}-1/q}}\Bigg\},
\end{align}
where the last inequality follows from \eqref{C1}. Therefore, by taking the $L^q({\mu})$-norm of both sides of \eqref{Iprime}(with respect to the variable $y$), and then using Minkowski's inequality($q\geq1$), \eqref{WI1yr} and \eqref{WI2yr}, we have
\begin{equation*}
\begin{split}
&\Big\|\nu_{\vec{w}}(B(y,r))^{1/{\alpha}-1/p-1/q}
\big\|T_\theta(\vec{f})\cdot\chi_{B(y,r)}\big\|_{WL^p(\nu_{\vec{w}})}\Big\|_{L^q({\mu})}\\
&\leq\Big\|I^{0,\dots,0}_\ast(y,r)\Big\|_{L^q({\mu})}+
\sum_{(\beta_1,\dots,\beta_m)\in\mathfrak{L}}\Big\|I^{\beta_1,\dots,\beta_m}_\ast(y,r)\Big\|_{L^q({\mu})}\\
&\leq C\bigg\|\prod_{k=1}^m\bigg[w_k(B(y,2r))^{1/{\alpha_k}-1/{p_k}-1/{q_k}}
\big\|f_k\cdot\chi_{B(y,2r)}\big\|_{L^{p_k}(w_k)}\bigg]\bigg\|_{L^q({\mu})}\\
&+C\cdot2^m\sum_{j=1}^\infty\bigg\|\prod_{k=1}^m\bigg[w_k(B(y,2^{j+1}r))^{1/{\alpha_k}-1/{p_k}-1/{q_k}}
\big\|f_k\cdot\chi_{B(y,2^{j+1}r)}\big\|_{L^{p_k}(w_k)}\bigg]\bigg\|_{L^q({\mu})}\\
&\times\left(\frac{|B(y,r)|}{|B(y,2^{j+1}r)|}\right)^{\delta(1/{\alpha}-1/q)},
\end{split}
\end{equation*}
where in the last step we have used inequality \eqref{psi1}. A further application of H\"older's inequality leads to that
\begin{equation*}
\begin{split}
&\Big\|\nu_{\vec{w}}(B(y,r))^{1/{\alpha}-1/p-1/q}
\big\|T_\theta(\vec{f})\cdot\chi_{B(y,r)}\big\|_{WL^p(\nu_{\vec{w}})}\Big\|_{L^q({\mu})}\\
&\leq C\prod_{k=1}^m\Big\|w_k(B(y,2r))^{1/{\alpha_k}-1/{p_k}-1/{q_k}}
\big\|f_k\cdot\chi_{B(y,2r)}\big\|_{L^{p_k}(w_k)}\Big\|_{L^{q_k}({\mu})}\\
&+C\sum_{j=1}^\infty\prod_{k=1}^m\Big\|w_k(B(y,2^{j+1}r))^{1/{\alpha_k}-1/{p_k}-1/{q_k}}
\big\|f_k\cdot\chi_{B(y,2^{j+1}r)}\big\|_{L^{p_k}(w_k)}\Big\|_{L^{q_k}({\mu})}\\
&\times\left(\frac{|B(y,r)|}{|B(y,2^{j+1}r)|}\right)^{\delta(1/{\alpha}-1/q)}\\
&\leq C\prod_{k=1}^m\big\|f_k\big\|_{(L^{p_k},L^{q_k})^{\alpha_k}(w_k;\mu)}+
C\prod_{k=1}^m\big\|f_k\big\|_{(L^{p_k},L^{q_k})^{\alpha_k}(w_k;\mu)}
\times\sum_{j=1}^\infty\left(\frac{|B(y,r)|}{|B(y,2^{j+1}r)|}\right)^{\delta(1/{\alpha}-1/q)}\\
&\leq C\prod_{k=1}^m\big\|f_k\big\|_{(L^{p_k},L^{q_k})^{\alpha_k}(w_k;\mu)},
\end{split}
\end{equation*}
where the last inequality holds since $\delta>0$ and $(1/{\alpha}-1/q)>0$. Thus, by taking the supremum over all $r>0$, we finish the proof of Theorem \ref{mainthm:2}.
\end{proof}

Let $1\leq p_1,\dots,p_m\leq+\infty$. We say that $\vec{w}=(w_1,\ldots,w_m)\in \prod_{k=1}^m A_{p_k}$, if each $w_k$ is in $A_{p_k}$, $k=1,2,\dots,m$. By using H\"older's inequality, it is easy to check that
\begin{equation}\label{includewh}
\prod_{k=1}^m A_{p_k}\subset A_{\vec{P}}.
\end{equation}
Moreover, it was shown in \cite[Remark 7.2]{lerner} that this inclusion \eqref{includewh} is strict. It is clear that $\prod_{k=1}^m A_{p_k}\subset \prod_{k=1}^m A_\infty$. So we have
\begin{equation}\label{include}
\prod_{k=1}^m A_{p_k}\subset A_{\vec{P}}\bigcap\prod_{k=1}^m A_\infty.
\end{equation}
This leads to a natural conjecture that whether the above inclusion is also strict. Currently, it is not clear whether one may also prove the converse of this inclusion relation. Thus, as a direct consequence of Theorems \ref{mainthm:1} and \ref{mainthm:2}, we immediately obtain the following results.
\begin{corollary}
Let $m\geq 2$ and $T_{\theta}$ be an $m$-linear $\theta$-type Calder\'on--Zygmund operator with $\theta$ satisfying the condition \eqref{theta1}. Suppose that $1<p_k\leq\alpha_k<q_k<\infty$, $k=1,2,\ldots,m$ and $p\in(1/m,\infty)$ with $1/p=\sum_{k=1}^m 1/{p_k}$, $q\in[1,\infty)$ with $1/q=\sum_{k=1}^m 1/{q_k}$ and $1/{\alpha}=\sum_{k=1}^m 1/{\alpha_k};$ $\vec{w}=(w_1,\ldots,w_m)\in\prod_{k=1}^m A_{p_k}$ and $\mu\in\Delta_2$. In addition, suppose that
\begin{equation*}
p_1\bigg(\frac{1}{\alpha_1}-\frac{1}{q_1}\bigg)=p_2\bigg(\frac{1}{\alpha_2}-\frac{1}{q_2}\bigg)
=\cdots=p_m\bigg(\frac{1}{\alpha_m}-\frac{1}{q_m}\bigg).
\end{equation*}
Then there exists a constant $C>0$ such that for all $\vec{f}=(f_1,\ldots,f_m)\in(L^{p_1},L^{q_1})^{\alpha_1}(w_1;\mu)\times\cdots
\times(L^{p_m},L^{q_m})^{\alpha_m}(w_m;\mu)$,
\begin{equation*}
\big\|T_\theta(\vec{f})\big\|_{(L^p,L^q)^{\alpha}(\nu_{\vec{w}};\mu)}\leq C\prod_{k=1}^m\big\|f_k\big\|_{(L^{p_k},L^{q_k})^{\alpha_k}(w_k;\mu)}
\end{equation*}
with $\nu_{\vec{w}}=\prod_{k=1}^m w_k^{p/{p_k}}$.
\end{corollary}

\begin{corollary}
Let $m\geq 2$ and $T_{\theta}$ be an $m$-linear $\theta$-type Calder\'on--Zygmund operator with $\theta$ satisfying the condition \eqref{theta1}. Suppose that $1\leq p_k\leq\alpha_k<q_k<\infty$, $k=1,2,\ldots,m$, $\min\{p_1,\ldots,p_m\}=1$ and $p\in[1/m,\infty)$ with $1/p=\sum_{k=1}^m 1/{p_k}$, $q\in[1,\infty)$ with $1/q=\sum_{k=1}^m 1/{q_k}$ and $1/{\alpha}=\sum_{k=1}^m 1/{\alpha_k};$ $\vec{w}=(w_1,\ldots,w_m)\in\prod_{k=1}^m A_{p_k}$ and $\mu\in\Delta_2$. In addition, suppose that
\begin{equation*}
p_1\bigg(\frac{1}{\alpha_1}-\frac{1}{q_1}\bigg)=p_2\bigg(\frac{1}{\alpha_2}-\frac{1}{q_2}\bigg)
=\cdots=p_m\bigg(\frac{1}{\alpha_m}-\frac{1}{q_m}\bigg).
\end{equation*}
Then there exists a constant $C>0$ such that for all $\vec{f}=(f_1,\ldots,f_m)\in(L^{p_1},L^{q_1})^{\alpha_1}(w_1;\mu)\times\cdots
\times(L^{p_m},L^{q_m})^{\alpha_m}(w_m;\mu)$,
\begin{equation*}
\big\|T_\theta(\vec{f})\big\|_{(WL^p,L^q)^{\alpha}(\nu_{\vec{w}};\mu)}\leq C\prod_{k=1}^m\big\|f_k\big\|_{(L^{p_k},L^{q_k})^{\alpha_k}(w_k;\mu)}
\end{equation*}
with $\nu_{\vec{w}}=\prod_{k=1}^m w_k^{p/{p_k}}$.
\end{corollary}

\section{Proofs of Theorems \ref{mainthm:3} and \ref{mainthm:4}}
\label{sec5}
As pointed out in Section \ref{sec1}, the conclusions of Theorem \ref{comm} and Theorem \ref{commwh} still hold with \eqref{theta2} replaced by the weaker growth condition \eqref{theta1}. In this section, let us begin by proving the following estimates.
\begin{theorem}\label{comm2}
Let $m\in\mathbb N$ and $\vec{b}=(b_1,\ldots,b_m)\in\mathrm{BMO}^m$; let $\theta$ satisfy the condition \eqref{theta1}.
If $p_1,\ldots,p_m\in(1,\infty)$ and $p\in(1/m,\infty)$ with $1/p=\sum_{k=1}^m 1/{p_k}$, and $\vec{w}=(w_1,\ldots,w_m)\in A_{\vec{P}}$, then there exists a constant $C>0$ independent of $\vec{f}$ such that
\begin{equation*}
\Big\|\big[\Sigma\vec{b},T_\theta\big](\vec{f})\Big\|_{L^p(\nu_{\vec{w}})}\leq
C\cdot\big\|\vec{b}\big\|_{\mathrm{BMO}^m}\prod_{k=1}^m\big\|f_k\big\|_{L^{p_k}(w_k)},
\end{equation*}
and
\begin{equation*}
\Big\|\big[\Pi\vec{b},T_\theta\big](\vec{f})\Big\|_{L^p(\nu_{\vec{w}})}\leq
C\cdot\prod_{k=1}^m\big\|b_k\big\|_{*}\prod_{k=1}^m\big\|f_k\big\|_{L^{p_k}(w_k)},
\end{equation*}
where $\nu_{\vec{w}}=\prod_{k=1}^m w_k^{p/{p_k}}$ and $\vec{f}=(f_1,\dots,f_m)\in L^{p_1}(w_1)\times\cdots\times L^{p_m}(w_m)$.
\end{theorem}
\begin{proof}
We only need to establish strong-type estimates for iterated commutators of multilinear $\theta$-type Calder\'on--Zygmund operators. The corresponding estimates for multilinear commutators can be obtained in a similar way. Some ideas of the proof of Theorem \ref{comm2} come from \cite{alvarez,ding} and \cite[Proposition 3.1]{perez3}.
For given $b_k\in\mathrm{BMO}(\mathbb R^n)$ with $1\leq k\leq m$, we denote $F_k(\xi)=e^{\xi[b_k(x)-b_k(y)]}$, $\xi\in\mathbb C$. Then by the analyticity of $F_k(\xi)$ on $\mathbb C$ and the Cauchy integral formula, we get
\begin{equation*}
\begin{split}
b_k(x)-b_k(y)&=F'_k(0)=\frac{1}{2\pi i}\int_{|\xi|=1}\frac{F_k(\xi)}{\xi^2}\,d\xi\\
&=\frac{1}{2\pi}\int_0^{2\pi}e^{e^{i \varphi_k}[b_k(x)-b_k(y)]}\cdot e^{-i \varphi_k}d\varphi_k.
\end{split}
\end{equation*}
Hence
\begin{equation*}
\begin{split}
&\big[\Pi\vec{b},T_\theta\big](f_1,\ldots,f_m)(x)\\
&=\int_{(\mathbb R^n)^m}\prod_{k=1}^m\big[b_k(x)-b_k(y_k)\big]K(x,y_1,\dots,y_m)f_1(y_1)\cdots f_m(y_m)\,dy_1\cdots dy_m\\
&=\int_{(\mathbb R^n)^m}\prod_{k=1}^m\bigg(\frac{1}{2\pi}\int_0^{2\pi}e^{e^{i\varphi_k}[b_k(x)-b_k(y_k)]}\cdot e^{-i\varphi_k}d\varphi_k\bigg)\\
&\times K(x,y_1,\dots,y_m)f_1(y_1)\cdots f_m(y_m)\,dy_1\cdots dy_m\\
&=\int_{(\mathbb R^n)^m}\bigg[\frac{1}{(2\pi)^m}\int_{[0,2\pi]^m}\bigg(\prod_{k=1}^me^{e^{i\varphi_k}b_k(x)}\cdot e^{-i\varphi_k}\bigg)d\varphi_1\cdots d\varphi_m\bigg]\\
&\times K(x,y_1,\dots,y_m)\prod_{k=1}^m e^{-e^{i\varphi_k}b_k(y_k)}\cdot f_k(y_k)\,dy_1\cdots dy_m\\
&=\frac{1}{(2\pi)^m}\int_{[0,2\pi]^m}T_{\theta}\Big(e^{-e^{i\varphi_1}b_1}\cdot f_1,\dots,e^{-e^{i\varphi_m}b_m}\cdot f_m\Big)(x)
\bigg(\prod_{k=1}^me^{e^{i\varphi_k}b_k(x)}\cdot e^{-i\varphi_k}\bigg)d\varphi_1\cdots d\varphi_m.
\end{split}
\end{equation*}
So we have
\begin{equation*}
\begin{split}
&\big|\big[\Pi\vec{b},T_\theta\big](f_1,\ldots,f_m)(x)\big|\\
\leq &\frac{1}{(2\pi)^m}\int_{[0,2\pi]^m}\Big|T_{\theta}\Big(e^{-e^{i\varphi_1}b_1}\cdot f_1,\dots,e^{-e^{i\varphi_m}b_m}\cdot f_m\Big)(x)\Big|\bigg(\prod_{k=1}^me^{\cos\varphi_k b_k(x)}\bigg)d\varphi_1\cdots d\varphi_m.
\end{split}
\end{equation*}
For any $(\varphi_1,\dots,\varphi_m)\in[0,2\pi]^m$, define $m$-tuples
\begin{equation*}
\vec{g}_{\varphi}=\big(g^1_{\varphi_1},\dots,g^m_{\varphi_m}\big),\quad\mbox{where}\;\; g^k_{\varphi_k}=e^{-e^{i\varphi_k}b_k}\cdot f_k,\;\; k=1,2,\dots,m,
\end{equation*}
and define
\begin{equation*}
\vec{w}_{\varphi}=\big(w^1_{\varphi_1},\dots,w^m_{\varphi_m}\big),\quad \mbox{where}\;\;
w^k_{\varphi_k}=w_k\cdot e^{p_k\cos\varphi_k b_k},\;\; k=1,2,\dots,m.
\end{equation*}
Set
\begin{equation*}
\nu^\ast_{\vec{w}}=\prod_{k=1}^m\big(w^k_{\varphi_k}\big)^{p/{p_k}}.
\end{equation*}
Then we have
\begin{equation*}
\nu^\ast_{\vec{w}}=\prod_{k=1}^m\big(w_k\cdot e^{p_k\cos\varphi_k b_k}\big)^{p/{p_k}}
=\nu_{\vec{w}}\cdot\prod_{k=1}^m e^{p\cos\varphi_k b_k}.
\end{equation*}
Using Minkowski's inequality for integrals, we thus obtain
\begin{equation*}
\begin{split}
\Big\|\big[\Pi\vec{b},T_\theta\big](\vec{f})\Big\|_{L^p(\nu_{\vec{w}})}
&\leq \frac{1}{(2\pi)^m}\int_{[0,2\pi]^m}\bigg\|T_{\theta}(\vec{g}_{\varphi})\prod_{k=1}^me^{\cos\varphi_k b_k}\bigg\|_{L^p(\nu_{\vec{w}})}
d\varphi_1\cdots d\varphi_m\\
&=\frac{1}{(2\pi)^m}\int_{[0,2\pi]^m}\big\|T_{\theta}(\vec{g}_{\varphi})\big\|_{L^p(\nu^\ast_{\vec{w}})}
d\varphi_1\cdots d\varphi_m.
\end{split}
\end{equation*}
Since $\vec{w}=(w_1,\ldots,w_m)\in A_{\vec{P}}$, we have $\nu_{\vec{w}}\in A_{mp}$ and $w_k^{1-p'_k}\in A_{mp'_k}$,$k=1,2,\ldots,m$, by using Lemma \ref{multi}. Hence, by the self-improvement property of $A_p$ weights (see \cite{duoand,garcia}), there exist some positive numbers $\varepsilon',\varepsilon_1,\dots,\varepsilon_m>0$(sufficiently small) such that
\begin{equation*}
\nu_{\vec{w}}^{1+\varepsilon'}\in A_{mp}\quad \&\quad\big(w_k^{1-p'_k}\big)^{1+\varepsilon_k}\in A_{mp'_k},~k=1,2,\dots,m.
\end{equation*}
Now choose
\begin{equation*}
\varepsilon:=\min\big\{\varepsilon',\varepsilon_1,\dots,\varepsilon_m\big\}.
\end{equation*}
Then we have
\begin{equation*}
\nu_{\vec{w}}^{1+\varepsilon}\in A_{mp}\quad \&\quad\big(w_k^{1-p'_k}\big)^{1+\varepsilon}=\big(w_k^{1+\varepsilon}\big)^{1-p'_k}\in A_{mp'_k},~k=1,2,\dots,m,
\end{equation*}
which implies $(\vec{w})^{1+\varepsilon}:=(w_1^{1+\varepsilon},\ldots,w_m^{1+\varepsilon})\in A_{\vec{P}}$ by using Lemma \ref{multi} again.
Note that
\begin{equation*}
\prod_{k=1}^m\big(w_i^{1+\epsilon}\big)^{p/{p_i}}=\bigg(\prod_{i=1}^m w_i^{p/{p_i}}\bigg)^{1+\epsilon}
=(\nu_{\vec{w}})^{1+\epsilon}.
\end{equation*}
Thus by Theorem \ref{strong},
\begin{equation}\label{inter1}
T_{\theta}:L^{p_1}(w_1^{1+\epsilon})\times\cdots\times L^{p_m}(w_m^{1+\epsilon})
\longrightarrow L^p((\nu_{\vec{w}})^{1+\epsilon}).
\end{equation}
On the other hand, for any fixed $\eta>0$, it is known that when $b\in\mathrm{BMO}(\mathbb R^n)$ with $\|b\|_{*}<\min\{C_2/{\eta},C_2(p-1)/{\eta}\}$, where $C_2$ is the constant in the John--Nirenberg inequality mentioned above, we have $e^{\eta b(x)}\in A_p$ for $1<p<\infty$
(see \cite[Lemma 1]{ding}). For $b_k\in\mathrm{BMO}(\mathbb R^n)$($1\leq k\leq m$), we now choose
\begin{equation*}
\eta_k:=\frac{p_k(1+\varepsilon)}{\varepsilon}.
\end{equation*}
For such $\eta_k>0$, we may assume that $\|b_k\|_{*}<\min\{C_2/{\eta_k},C_2(p_k-1)/{\eta_k}\}$. The general case can be proved using the linearity of $T_{\theta}$ as well. Then for any $\varphi_k\in[0,2\pi]$, we have $\cos\varphi_k\cdot b_k(x)\in \mathrm{BMO}(\mathbb R^n)$, and
\begin{equation*}
\|\cos\varphi_k\cdot b_k\|_{*}\leq\|b_k\|_{*}<\min\big\{C_2/{\eta_k},C_2(p_k-1)/{\eta_k}\big\},
\end{equation*}
which implies that each $\nu_k(x):=e^{\eta_k\cos\varphi_kb_k(x)}\in A_{p_k}$ for $1<p_k<\infty$, $k=1,2,\dots,m$. Notice that
\begin{equation*}
\prod_{k=1}^m\big(e^{\frac{1+\varepsilon}{\varepsilon}p_k\cos\varphi_kb_k}\big)^{p/{p_k}}=
\prod_{k=1}^m\big(e^{\frac{1+\varepsilon}{\varepsilon}p\cos\varphi_kb_k}\big)=
\bigg(\prod_{k=1}^m e^{p\cos\varphi_kb_k}\bigg)^{\frac{1+\varepsilon}{\varepsilon}}.
\end{equation*}
This fact along with \eqref{include} and Theorem \ref{strong} gives us that
\begin{equation}\label{inter2}
T_{\theta}:L^{p_1}\big(e^{\frac{1+\varepsilon}{\varepsilon}p_1\cos\varphi_1b_1}\big)
\times\cdots
\times L^{p_m}\big(e^{\frac{1+\varepsilon}{\varepsilon}p_m\cos\varphi_mb_m}\big)\longrightarrow
L^p\Big(\big(\prod_{k=1}^m e^{p\cos\varphi_kb_k}\big)^{\frac{1+\varepsilon}{\varepsilon}}\Big).
\end{equation}
Interpolating between \eqref{inter1} and \eqref{inter2}(see \cite{bergh,stein}) we obtain that
\begin{equation*}
T_{\theta}:L^{p_1}\big(w_1e^{p_1\cos\varphi_1b_1}\big)
\times\cdots
\times L^{p_m}\big(w_me^{p_m\cos\varphi_mb_m}\big)\longrightarrow
L^p\Big(\nu_{\vec{w}}\prod_{k=1}^m e^{p\cos\varphi_kb_k}\Big);
\end{equation*}
that is
\begin{equation}\label{inter3}
T_{\theta}:L^{p_1}\big(w^1_{\varphi_1}\big)
\times\cdots
\times L^{p_m}\big(w^m_{\varphi_m}\big)\longrightarrow
L^p\big(\nu^\ast_{\vec{w}}\big).
\end{equation}
By \eqref{inter3} we have
\begin{equation}
\big\|T_{\theta}(\vec{g}_{\varphi})\big\|_{L^p(\nu^\ast_{\vec{w}})}
\leq C\prod_{k=1}^m\big\|g^k_{\varphi_k}\big\|_{L^{p_k}(w^k_{\varphi_k})}.
\end{equation}
Since $f_k\in L^{p_k}(w_k)$, it is easy to check that for any $\varphi_k\in[0,2\pi]$,
\begin{equation*}
\begin{split}
\big\|g^k_{\varphi_k}\big\|_{L^{p_k}(w^k_{\varphi_k})}
&=\bigg(\int_{\mathbb R^n}\big|g^k_{\varphi_k}(x)\big|^{p_k} w_k(x)\cdot e^{p_k\cos\varphi_k b_k(x)}dx\bigg)^{1/{p_k}}\\
&=\bigg(\int_{\mathbb R^n}\big|f_k(x)\big|^{p_k}e^{-p_k\cos\varphi_k b_k(x)}\cdot w_k(x)\cdot e^{p_k\cos\varphi_k b_k(x)}dx\bigg)^{1/{p_k}}\\
&=\bigg(\int_{\mathbb R^n}\big|f_k(x)\big|^{p_k} w_k(x)dx\bigg)^{1/{p_k}}=\big\|f_k\big\|_{L^{p_k}(w_k)}.
\end{split}
\end{equation*}
Therefore
\begin{equation*}
\begin{split}
\Big\|\big[\Pi\vec{b},T_\theta\big](\vec{f})\Big\|_{L^p(\nu_{\vec{w}})}
&\leq C\frac{1}{(2\pi)^m}\int_{[0,2\pi]^m}\prod_{k=1}^m\big\|g^k_{\varphi_k}\big\|_{L^{p_k}(w^k_{\varphi_k})}
d\varphi_1\cdots d\varphi_m\\
&=C\frac{1}{(2\pi)^m}\int_{[0,2\pi]^m}\prod_{k=1}^m\big\|f_k\big\|_{L^{p_k}(w_k)}
d\varphi_1\cdots d\varphi_m\\
&\leq C\prod_{k=1}^m\big\|f_k\big\|_{L^{p_k}(w_k)},
\end{split}
\end{equation*}
which is our desired estimate. This gives the proof in the special case. We now proceed to the general case. To do this, we set
\begin{equation*}
\widetilde{b}_k(x):=\widetilde{\eta}_k\cdot\frac{b_k(x)}{\|b_k\|_{*}}\quad \& \quad \vec{b}_\sharp:=\big(\widetilde{b}_1,\dots,\widetilde{b}_m\big).
\end{equation*}
Here $\widetilde{\eta}_k$ is chosen so that $0<\widetilde{\eta}_k<\min\{C_2/{\eta_k},C_2(p_k-1)/{\eta_k}\}$, $k=1,2,\dots,m$. Then
\begin{equation*}
\|\widetilde{b}_k\|_{*}=\widetilde{\eta}_k<\min\big\{C_2/{\eta_k},C_2(p_k-1)/{\eta_k}\big\}.
\end{equation*}
From the previous proof, it actually follows that
\begin{equation*}
\begin{split}
\Big\|\big[\Pi\vec{b},T_\theta\big](\vec{f})\Big\|_{L^p(\nu_{\vec{w}})}
&=\Big\|\prod_{k=1}^m\frac{\|b_k\|_{*}}{\widetilde{\eta}_k}\cdot\big[\Pi\vec{b}_\sharp,T_\theta\big](\vec{f})\Big\|_{L^p(\nu_{\vec{w}})}\\
&\leq C\cdot\prod_{k=1}^m\big\|b_k\big\|_{*}\prod_{k=1}^m\big\|f_k\big\|_{L^{p_k}(w_k)}.
\end{split}
\end{equation*}
This completes the proof of Theorem \ref{comm2}.
\end{proof}

To prove our main theorems for multilinear commutators in this section, let us first set up two auxiliary lemmas about $\mathrm{BMO}$ functions, which play an important role in our proofs of main theorems.
\begin{lemma}\label{BMO}
Let $b$ be a function in $\mathrm{BMO}(\mathbb R^n)$. Then

$(i)$ For every ball $B$ in $\mathbb R^n$ and for all $j\in\mathbb{N}$,
\begin{equation*}
\big|b_{2^{j+1}B}-b_B\big|\leq C\cdot(j+1)\|b\|_*.
\end{equation*}

$(ii)$ Let $1\leq p<\infty$. For every ball $B$ in $\mathbb R^n$ and for all $\omega\in A_{\infty}$,
\begin{equation*}
\bigg(\int_B\big|b(x)-b_B\big|^p\omega(x)\,dx\bigg)^{1/p}\leq C\|b\|_*\cdot\omega(B)^{1/p}.
\end{equation*}
\end{lemma}
\begin{proof}
For the proofs of $(i)$ and $(ii)$, we refer the reader to \cite{stein2}.
\end{proof}
Based on Lemma \ref{BMO}, we now assert that for any $j\in \mathbb{N}$ and any $\omega\in A_{\infty}$, the following inequality
\begin{equation}\label{j1cb}
\bigg(\int_{2^{j+1}B}\big|b(x)-b_B\big|^p\omega(x)\,dx\bigg)^{1/p}\leq C(j+1)\|b\|_*\cdot\omega(2^{j+1}B)^{1/p}
\end{equation}
holds whenever $b\in\mathrm{BMO}(\mathbb R^n)$ and $1\leq p<\infty$. Indeed, by using Lemma \ref{BMO} $(i)$ and $(ii)$, we can deduce that
\begin{equation*}
\begin{split}
&\bigg(\int_{2^{j+1}B}\big|b(x)-b_B\big|^p\omega(x)\,dx\bigg)^{1/p}\\
&\leq \bigg(\int_{2^{j+1}B}\big|b(x)-b_{2^{j+1}B}\big|^p\omega(x)\,dx\bigg)^{1/p}
+\bigg(\int_{2^{j+1}B}\big|b_{2^{j+1}B}-b_B\big|^p\omega(x)\,dx\bigg)^{1/p}\\
&\leq C\|b\|_*\cdot \omega(2^{j+1}B)^{1/p}+C(j+1)\|b\|_*\cdot \omega(2^{j+1}B)^{1/p}\\
&\leq C(j+1)\|b\|_*\cdot \omega(2^{j+1}B)^{1/p},
\end{split}
\end{equation*}
as desired.
\begin{lemma}\label{BMO3}
Let $b$ be a function in $\mathrm{BMO}(\mathbb R^n)$. Then for any ball $B$ in $\mathbb R^n$ and any $\omega\in A_{\infty}$, we have
\begin{equation}\label{BMOwang}
\big\|b-b_{B}\big\|_{\exp L(\omega),B}\leq C\|b\|_*.
\end{equation}
\end{lemma}
\begin{proof}
By the well-known John--Nirenberg's inequality (see \cite{john,duoand}), we know that there exist two positive constants $C_1$ and $C_2$, depending only on the dimension $n$, such that for any $\lambda>0$,
\begin{equation*}
\big|\big\{x\in B:|b(x)-b_B|>\lambda\big\}\big|\leq C_1|B|\exp\bigg\{-\frac{C_2\lambda}{\|b\|_{*}}\bigg\}.
\end{equation*}
This result shows that in some sense logarithmic growth is the maximum possible for BMO functions (more precisely, we can take $C_1=\sqrt{2}$, $C_2=\log 2/{2^{n+2}}$, see \cite[p.123--125]{duoand}). Applying the comparison property \eqref{compare} of $A_{\infty}$ weights, there is a positive number $\delta>0$ such that
\begin{equation*}
\omega\big(\big\{x\in B:|b(x)-b_B|>\lambda\big\}\big)\leq C_1\omega(B)\exp\bigg\{-\frac{C_2\delta\lambda}{\|b\|_{*}}\bigg\}.
\end{equation*}
From this, it follows that ($c_0$ and $C$ are two positive constants which are independent of the choice of $B$)
\begin{equation*}
\frac{1}{\omega(B)}\int_B\exp\bigg(\frac{|b(y)-b_B|}{c_0\|b\|_*}\bigg)\omega(y)\,dy\leq C,
\end{equation*}
which is equivalent to \eqref{BMOwang}.
\end{proof}
Furthermore, by \eqref{BMOwang} and Lemma \ref{BMO}$(i)$, it is easy to check that for any $\omega\in A_\infty$ and any given ball $B$ in $\mathbb R^n$,
\begin{equation}\label{Jensen}
\big\|b-b_{B}\big\|_{\exp L(\omega),2^{j+1}B}\leq C(j+1)\|b\|_\ast,\quad j\in \mathbb{N}.
\end{equation}

We are now in a position to give the proofs of Theorems \ref{mainthm:3} and \ref{mainthm:4}.
\begin{proof}[Proof of Theorem $\ref{mainthm:3}$]
Let $1<p_k\leq\alpha_k<q_k<\infty$ and $\vec{f}=(f_1,\dots,f_m)$ be in $(L^{p_1},L^{q_1})^{\alpha_1}(w_1;\mu)\times\cdots
\times(L^{p_m},L^{q_m})^{\alpha_m}(w_m;\mu)$ with $(w_1,\dots,w_m)\in A_{\vec{P}}$ and $\mu\in\Delta_2$.
As was pointed out in \cite{lerner}, by linearity it is enough to consider the multilinear commutator $[\Sigma b,T_{\theta}]$ with only one symbol. Without loss of generality, we fix $b\in\mathrm{BMO}(\mathbb R^n)$ and then consider the commutator operator $[b,T_{\theta}]_1$ given by
\begin{equation*}
\big[b,T_\theta\big]_1(\vec{f})(x)=b(x)\cdot T_\theta(f_1,f_2,\dots,f_m)(x)-T_{\theta}(bf_1,f_2,\dots,f_m)(x).
\end{equation*}
For any fixed ball $B=B(y,r)\subset\mathbb R^n$ with $y\in\mathbb R^n$ and $r\in(0,+\infty)$, as before, we split each $f_k$ as
\begin{equation*}
f_k=f^0_k+f^{\infty}_k,\quad k=1,2,\dots,m,
\end{equation*}
where $f^0_k=f_k\cdot\chi_{2B}$, $f^{\infty}_k=f_k\cdot\chi_{(2B)^{\complement}}$ and $2B=B(y,2r)\subset\mathbb R^n$. Let $\mathfrak{L}$ be the same as before. By using Lemma \ref{Min}($N=2^m$), we can write
\begin{align}\label{J}
&\nu_{\vec{w}}(B(y,r))^{1/{\alpha}-1/p-1/q}
\big\|[b,T_\theta]_1(\vec{f})\cdot\chi_{B(y,r)}\big\|_{L^p(\nu_{\vec{w}})}\notag\\
&=\nu_{\vec{w}}(B(y,r))^{1/{\alpha}-1/p-1/q}
\bigg(\int_{B(y,r)}\big|[b,T_\theta]_1(f_1,\dots,f_m)(x)\big|^p\nu_{\vec{w}}(x)\,dx\bigg)^{1/p}\notag\\
&\leq C\cdot\nu_{\vec{w}}(B(y,r))^{1/{\alpha}-1/p-1/q}
\bigg(\int_{B(y,r)}\big|[b,T_\theta]_1(f^0_1,\dots,f^0_m)(x)\big|^p\nu_{\vec{w}}(x)\,dx\bigg)^{1/p}\notag\\
&+C\sum_{(\beta_1,\dots,\beta_m)\in\mathfrak{L}}\nu_{\vec{w}}(B(y,r))^{1/{\alpha}-1/p-1/q}
\bigg(\int_{B(y,r)}\big|[b,T_\theta]_1(f^{\beta_1}_1,\ldots,f^{\beta_m}_m)(x)\big|^p\nu_{\vec{w}}(x)\,dx\bigg)^{1/p}\notag\\
&:=J^{0,\dots,0}(y,r)+\sum_{(\beta_1,\dots,\beta_m)\in\mathfrak{L}} J^{\beta_1,\dots,\beta_m}(y,r).
\end{align}
To estimate the first term in \eqref{J}, applying Theorem \ref{comm2} along with \eqref{C1} and \eqref{doubling}, we get
\begin{align}\label{J1yr}
J^{0,\dots,0}(y,r)&\leq C\cdot\nu_{\vec{w}}(B(y,r))^{1/{\alpha}-1/p-1/q}
\prod_{k=1}^m\bigg(\int_{B(y,2r)}|f_k(x)|^{p_k}w_k(x)\,dx\bigg)^{1/{p_k}}\notag\\
&=C\cdot\nu_{\vec{w}}(B(y,r))^{1/{\alpha}-1/p-1/q}
\prod_{k=1}^m w_k(B(y,2r))^{1/{p_k}+1/{q_k}-1/{\alpha_k}}\notag\\
&\times\prod_{k=1}^m\bigg[w_k(B(y,2r))^{1/{\alpha_k}-1/{p_k}-1/{q_k}}
\big\|f_k\cdot\chi_{B(y,2r)}\big\|_{L^{p_k}(w_k)}\bigg]\notag\\
&\leq C\cdot\frac{\nu_{\vec{w}}(B(y,2r))^{1/p+1/q-1/{\alpha}}}{\nu_{\vec{w}}(B(y,r))^{1/p+1/q-1/{\alpha}}}\notag\\
&\times\prod_{k=1}^m\bigg[w_k(B(y,2r))^{1/{\alpha_k}-1/{p_k}-1/{q_k}}
\big\|f_k\cdot\chi_{B(y,2r)}\big\|_{L^{p_k}(w_k)}\bigg]\\
&\leq C\cdot\prod_{k=1}^m\bigg[w_k(B(y,2r))^{1/{\alpha_k}-1/{p_k}-1/{q_k}}
\big\|f_k\cdot\chi_{B(y,2r)}\big\|_{L^{p_k}(w_k)}\bigg]\notag.
\end{align}
To estimate the remaining terms in \eqref{J}, let us first deal with the case when $\beta_1=\cdots=\beta_m=\infty$. It is easily seen that for any $x\in B(y,r)$,
\begin{equation*}
\big[b,T_\theta\big]_1(\vec{f})(x)=[b(x)-b_B]\cdot T_\theta(f_1,f_2,\dots,f_m)(x)-T_{\theta}((b-b_B)f_1,f_2,\dots,f_m)(x).
\end{equation*}
From it, the term $J^{\infty,\dots,\infty}(y,r)$ will be divided into two parts.
\begin{equation*}
\begin{split}
&J^{\infty,\dots,\infty}(y,r)\\
&\leq C\cdot\nu_{\vec{w}}(B(y,r))^{1/{\alpha}-1/p-1/q}
\bigg(\int_{B(y,r)}\big|[b(x)-b_B]\cdot T_\theta(f^\infty_1,f^\infty_2,\dots,f^\infty_m)(x)\big|^p\nu_{\vec{w}}(x)\,dx\bigg)^{1/p}\\
&+C\cdot\nu_{\vec{w}}(B(y,r))^{1/{\alpha}-1/p-1/q}
\bigg(\int_{B(y,r)}\big|T_\theta((b-b_B)f^\infty_1,f^\infty_2,\dots,f^\infty_m)(x)\big|^p
\nu_{\vec{w}}(x)\,dx\bigg)^{1/p}\\
&:=J^{\infty,\dots,\infty}_{\star}(y,r)+J^{\infty,\dots,\infty}_{\star\star}(y,r).
\end{split}
\end{equation*}
We analyze each term separately. In the proof of Theorem \ref{mainthm:1}, we have already established the following estimate for $T_{\theta}(f^\infty_1,\ldots,f^\infty_m)$ (see \eqref{111}).
\begin{equation*}
\big|T_{\theta}(f^\infty_1,f^\infty_2,\ldots,f^\infty_m)(x)\big|
\lesssim\sum_{j=1}^\infty\bigg(\prod_{k=1}^m\frac{1}{|B(y,2^{j+1}r)|}
\int_{B(y,2^{j+1}r)}\big|f_k(z_k)\big|\,dz_k\bigg).
\end{equation*}
Note that $\nu_{\vec{w}}\in A_{mp}\subset A_{\infty}$. Consequently, from $(ii)$ of Lemma \ref{BMO}, it follows that
\begin{equation*}
\begin{split}
J^{\infty,\dots,\infty}_{\star}(y,r)
&\leq C\cdot\nu_{\vec{w}}(B(y,r))^{1/{\alpha}-1/p-1/q}
\sum_{j=1}^\infty\bigg(\prod_{k=1}^m\frac{1}{|B(y,2^{j+1}r)|}
\int_{B(y,2^{j+1}r)}\big|f_k(z_k)\big|\,dz_k\bigg)\\
&\times\bigg(\int_{B(y,r)}\big|b(x)-b_{B(y,r)}\big|^p\nu_{\vec{w}}(x)\,dx\bigg)^{1/p}\\
&\leq C\|b\|_\ast\cdot\nu_{\vec{w}}(B(y,r))^{1/{\alpha}-1/q}
\sum_{j=1}^\infty\bigg(\prod_{k=1}^m\frac{1}{|B(y,2^{j+1}r)|}
\int_{B(y,2^{j+1}r)}\big|f_k(z_k)\big|\,dz_k\bigg).
\end{split}
\end{equation*}
We now proceed in the same way as in the estimation of $I^{\infty,\dots,\infty}(y,r)$, and obtain
\begin{equation*}
\begin{split}
J^{\infty,\dots,\infty}_{\star}(y,r)&\lesssim\|b\|_\ast\cdot\nu_{\vec{w}}(B(y,r))^{1/{\alpha}-1/q}\\
&\times\sum_{j=1}^\infty\Bigg\{\frac{1}{\nu_{\vec{w}}(B(y,2^{j+1}r))^{1/p}}\cdot
\prod_{k=1}^m\big\|f_k\cdot\chi_{B(y,2^{j+1}r)}\big\|_{L^{p_k}(w_k)}\Bigg\}\\
&=\|b\|_\ast\cdot\nu_{\vec{w}}(B(y,r))^{1/{\alpha}-1/q}
\times\sum_{j=1}^\infty\Bigg\{\frac{\prod_{k=1}^mw_k(B(y,2^{j+1}r))^{1/{p_k}+1/{q_k}-1/{\alpha_k}} }{\nu_{\vec{w}}(B(y,2^{j+1}r))^{1/p}}\\
&\times\prod_{k=1}^m\bigg[w_k(B(y,2^{j+1}r))^{1/{\alpha_k}-1/{p_k}-1/{q_k}}
\big\|f_k\cdot\chi_{B(y,2^{j+1}r)}\big\|_{L^{p_k}(w_k)}\bigg]\Bigg\}\\
&\lesssim\|b\|_\ast\sum_{j=1}^\infty\Bigg\{\prod_{k=1}^m\bigg[w_k(B(y,2^{j+1}r))^{1/{\alpha_k}-1/{p_k}-1/{q_k}}
\big\|f_k\cdot\chi_{B(y,2^{j+1}r)}\big\|_{L^{p_k}(w_k)}\bigg]\\
&\times\frac{\nu_{\vec{w}}(B(y,r))^{1/{\alpha}-1/q}}{\nu_{\vec{w}}(B(y,2^{j+1}r))^{1/{\alpha}-1/q}}\Bigg\}.
\end{split}
\end{equation*}
Following the same arguments as in the proof of Theorem $\ref{mainthm:1}$, we can also deduce that for any $x\in B(y,r)$,
\begin{align}\label{pointwise3}
&\big|T_\theta((b-b_B)f^\infty_1,f^\infty_2,\dots,f^\infty_m)(x)\big|\\
&\lesssim\int_{(\mathbb R^n)^m\backslash B(y,2r)^m}
\frac{|(b(z_1)-b_{B})f_1(z_1)|\cdot|f_2(z_2)\cdots f_m(z_m)|}{(|x-z_1|+\cdots+|x-z_m|)^{mn}}\,dz_1\cdots dz_m\notag\\
&=\sum_{j=1}^\infty\int_{B(y,2^{j+1}r)^m\backslash B(y,2^{j}r)^m}
\frac{|(b(z_1)-b_{B})f_1(z_1)|\cdot|f_2(z_2)\cdots f_m(z_m)|}{(|x-z_1|+\cdots+|x-z_m|)^{mn}}\,dz_1\cdots dz_m\notag\\
&\lesssim\sum_{j=1}^\infty\bigg(\frac{1}{|B(y,2^{j+1}r)|^m}
\int_{B(y,2^{j+1}r)^m\backslash B(y,2^{j}r)^m}\big|(b(z_1)-b_{B})f_1(z_1)\big|\cdot\big|f_2(z_2)\cdots f_m(z_m)\big|\,dz_1\cdots dz_m\bigg)\notag\\
&\leq\sum_{j=1}^\infty\bigg(\frac{1}{|B(y,2^{j+1}r)|^m}\int_{B(y,2^{j+1}r)}\big|(b(z_1)-b_{B})f_1(z_1)\big|\,dz_1
\prod_{k=2}^m\int_{B(y,2^{j+1}r)}\big|f_k(z_k)\big|\,dz_k\bigg)\notag\\
&=\sum_{j=1}^\infty\bigg(\frac{1}{|B(y,2^{j+1}r)|}\int_{B(y,2^{j+1}r)}\big|(b(z_1)-b_{B})f_1(z_1)\big|\,dz_1\bigg)
\bigg(\prod_{k=2}^m\frac{1}{|B(y,2^{j+1}r)|}\int_{B(y,2^{j+1}r)}\big|f_k(z_k)\big|\,dz_k\bigg)\notag.
\end{align}
Then we have
\begin{align}\label{substiw1}
J^{\infty,\dots,\infty}_{\star\star}(y,r)&\lesssim\nu_{\vec{w}}(B(y,r))^{1/{\alpha}-1/q}\sum_{j=1}^\infty
\bigg(\frac{1}{|B(y,2^{j+1}r)|}\int_{B(y,2^{j+1}r)}\big|(b(z_1)-b_{B})f_1(z_1)\big|\,dz_1\bigg)\notag\\
&\times\bigg(\prod_{k=2}^m\frac{1}{|B(y,2^{j+1}r)|}\int_{B(y,2^{j+1}r)}\big|f_k(z_k)\big|\,dz_k\bigg).
\end{align}
For each $2\leq k\leq m$, by using H\"older's inequality with exponent $p_k$, we obtain that
\begin{equation*}
\begin{split}
&\int_{B(y,2^{j+1}r)}\big|f_k(z_k)\big|\,dz_k\\
&\leq
\bigg(\int_{B(y,2^{j+1}r)}\big|f_k(z_k)\big|^{p_k}w_k(z_k)\,dz_k\bigg)^{1/{p_k}}
\bigg(\int_{B(y,2^{j+1}r)}w_k(z_k)^{-p'_k/{p_k}}\,dz_k\bigg)^{1/{p'_k}}.
\end{split}
\end{equation*}
According to Lemma \ref{multi}, we have $w_k^{1-p'_k}=w_k^{-p'_k/{p_k}}\in A_{mp'_k}\subset A_{\infty}$, $k=1,2,\dots,m$. By using H\"older's inequality again with exponent $p_1$ and \eqref{j1cb}, we deduce that
\begin{equation*}
\begin{split}
&\int_{B(y,2^{j+1}r)}|(b(z_1)-b_{B})f_1(z_1)|\,dz_1\\
&\leq\bigg(\int_{B(y,2^{j+1}r)}\big|f_1(z_1)\big|^{p_1}w_1(z_1)\,dz_1\bigg)^{1/{p_1}}
\bigg(\int_{B(y,2^{j+1}r)}|b(z_1)-b_{B(y,r)}|^{p'_1}w_1(z_1)^{-p'_1/{p_1}}\,dz_1\bigg)^{1/{p'_1}}\\
&\lesssim\bigg(\int_{B(y,2^{j+1}r)}\big|f_1(z_1)\big|^{p_1}w_1(z_1)\,dz_1\bigg)^{1/{p_1}}
(j+1)\|b\|_*\cdot\bigg(\int_{B(y,2^{j+1}r)}w_1(z_1)^{-p'_1/{p_1}}\,dz_1\bigg)^{1/{p'_1}},
\end{split}
\end{equation*}
where the last inequality follows from the fact that $w_1^{-p'_1/{p_1}}\in A_{\infty}$.
Hence, from the above two estimates and the $A_{\vec{P}}$ condition on $\vec{w}$, it follows that \eqref{substiw1} is bounded by
\begin{equation*}
\begin{split}
&\|b\|_*\cdot\nu_{\vec{w}}(B(y,r))^{1/{\alpha}-1/q}\sum_{j=1}^\infty(j+1)\Bigg\{\prod_{k=1}^m
\frac{1}{|B(y,2^{j+1}r)|}\bigg(\int_{B(y,2^{j+1}r)}\big|f_k(z_k)\big|^{p_k}w_k(z_k)\,dz_k\bigg)^{1/{p_k}}\\
&\bigg(\int_{B(y,2^{j+1}r)}w_k(z_k)^{-p'_k/{p_k}}\,dz_k\bigg)^{1/{p'_k}}\Bigg\}\\
&\lesssim\|b\|_*\cdot\nu_{\vec{w}}(B(y,r))^{1/{\alpha}-1/q}\sum_{j=1}^\infty(j+1)
\Bigg\{\frac{1}{\nu_{\vec{w}}(B(y,2^{j+1}r))^{1/p}}\cdot
\prod_{k=1}^m\big\|f_k\cdot\chi_{B(y,2^{j+1}r)}\big\|_{L^{p_k}(w_k)}\Bigg\}.
\end{split}
\end{equation*}
Therefore, in view of the estimate \eqref{C1}, we conclude that
\begin{equation*}
\begin{split}
J^{\infty,\dots,\infty}_{\star\star}(y,r)
&\lesssim\|b\|_\ast\sum_{j=1}^\infty(j+1)\Bigg\{\prod_{k=1}^m\bigg[w_k(B(y,2^{j+1}r))^{1/{\alpha_k}-1/{p_k}-1/{q_k}}
\big\|f_k\cdot\chi_{B(y,2^{j+1}r)}\big\|_{L^{p_k}(w_k)}\bigg]\\
&\times\frac{\nu_{\vec{w}}(B(y,r))^{1/{\alpha}-1/q}}{\nu_{\vec{w}}(B(y,2^{j+1}r))^{1/{\alpha}-1/q}}\Bigg\}.
\end{split}
\end{equation*}
Summarizing the estimates for $J^{\infty,\dots,\infty}_{\star}(y,r)$ and $J^{\infty,\dots,\infty}_{\star\star}(y,r)$ derived above, we get
\begin{equation}\label{J2yr}
\begin{split}
J^{\infty,\dots,\infty}(y,r)&\lesssim\|b\|_\ast\sum_{j=1}^\infty(j+1)\Bigg\{\prod_{k=1}^m\bigg[w_k(B(y,2^{j+1}r))^{1/{\alpha_k}-1/{p_k}-1/{q_k}}
\big\|f_k\cdot\chi_{B(y,2^{j+1}r)}\big\|_{L^{p_k}(w_k)}\bigg]\\
&\times\frac{\nu_{\vec{w}}(B(y,r))^{1/{\alpha}-1/q}}{\nu_{\vec{w}}(B(y,2^{j+1}r))^{1/{\alpha}-1/q}}\Bigg\}.
\end{split}
\end{equation}
We now consider the case where exactly $\ell$ of the $\beta_k$ are $\infty$ for some $1\le\ell<m$. We only give the arguments for one of these cases. The rest are similar and can be easily obtained from the arguments below by permuting the indices. Meanwhile, we only consider the case $\beta_1=\infty$ here since the other case can be proved in the same way. We now estimate the term $J^{\beta_1,\dots,\beta_m}(y,r)$ when
\begin{equation*}
\beta_1=\cdots=\beta_{\ell}=\infty \quad \&\quad\beta_{\ell+1}=\cdots=\beta_m=0.
\end{equation*}
In the present situation, we first divide the term $J^{\beta_1,\dots,\beta_m}(y,r)$ into two parts as follows.
\begin{equation*}
\begin{split}
&J^{\beta_1,\dots,\beta_m}(y,r)\\
&\leq C\cdot\nu_{\vec{w}}(B(y,r))^{1/{\alpha}-1/p-1/q}
\bigg(\int_{B(y,r)}\big|[b(x)-b_B]\cdot T_\theta(f^\infty_1,\ldots,f^\infty_\ell,f^0_{\ell+1},\ldots,f^0_m)(x)\big|^p\nu_{\vec{w}}(x)\,dx\bigg)^{1/p}\\
&+C\cdot\nu_{\vec{w}}(B(y,r))^{1/{\alpha}-1/p-1/q}
\bigg(\int_{B(y,r)}\big|T_\theta((b-b_B)f^\infty_1,\ldots,f^\infty_\ell,f^0_{\ell+1},\ldots,f^0_m)(x)\big|^p
\nu_{\vec{w}}(x)\,dx\bigg)^{1/p}\\
&:=J^{\beta_1,\dots,\beta_m}_{\star}(y,r)+J^{\beta_1,\dots,\beta_m}_{\star\star}(y,r).
\end{split}
\end{equation*}
We estimate each term respectively. Recall that the following result has been proved in Theorem \ref{mainthm:1}(see \eqref{112}).
\begin{equation*}
\big|T_{\theta}(f^\infty_1,\ldots,f^\infty_\ell,f^0_{\ell+1},\ldots,f^0_m)(x)\big|
\lesssim\sum_{j=1}^\infty\bigg(\prod_{k=1}^m\frac{1}{|B(y,2^{j+1}r)|}\int_{B(y,2^{j+1}r)}\big|f_k(z_k)\big|\,dz_k\bigg).
\end{equation*}
Then it follows directly from $(ii)$ of Lemma \ref{BMO} that
\begin{equation*}
\begin{split}
J^{\beta_1,\dots,\beta_m}_{\star}(y,r)
&\leq C\cdot\nu_{\vec{w}}(B(y,r))^{1/{\alpha}-1/p-1/q}
\sum_{j=1}^\infty\bigg(\prod_{k=1}^m\frac{1}{|B(y,2^{j+1}r)|}
\int_{B(y,2^{j+1}r)}\big|f_k(z_k)\big|\,dz_k\bigg)\\
&\times\bigg(\int_{B(y,r)}\big|b(x)-b_{B(y,r)}\big|^p\nu_{\vec{w}}(x)\,dx\bigg)^{1/p}\\
&\leq C\|b\|_\ast\cdot\nu_{\vec{w}}(B(y,r))^{1/{\alpha}-1/q}
\sum_{j=1}^\infty\bigg(\prod_{k=1}^m\frac{1}{|B(y,2^{j+1}r)|}
\int_{B(y,2^{j+1}r)}\big|f_k(z_k)\big|\,dz_k\bigg).
\end{split}
\end{equation*}
We can now argue exactly as we did in the estimation of $I^{\infty,\dots,\infty}(y,r)$ to get that
\begin{equation*}
\begin{split}
J^{\beta_1,\dots,\beta_m}_{\star}(y,r)  &\lesssim\|b\|_\ast\sum_{j=1}^\infty\Bigg\{\prod_{k=1}^m\bigg[w_k(B(y,2^{j+1}r))^{1/{\alpha_k}-1/{p_k}-1/{q_k}}
\big\|f_k\cdot\chi_{B(y,2^{j+1}r)}\big\|_{L^{p_k}(w_k)}\bigg]\\
&\times\frac{\nu_{\vec{w}}(B(y,r))^{1/{\alpha}-1/q}}{\nu_{\vec{w}}(B(y,2^{j+1}r))^{1/{\alpha}-1/q}}\Bigg\}.
\end{split}
\end{equation*}
On the other hand, we will adopt the same method as in Theorem \ref{mainthm:1} and obtain
\begin{align}\label{pointwise5}
&\big|T_{\theta}((b-b_B)f^\infty_1,\ldots,f^\infty_\ell,f^0_{\ell+1},\ldots,f^0_m)(x)\big|\notag\\
&\lesssim\int_{(\mathbb R^n)^{\ell}\backslash B(y,2r)^{\ell}}\int_{B(y,2r)^{m-\ell}}
\frac{|(b(z_1)-b_{B})f_1(z_1)|\cdot|f_2(z_2)\cdots f_m(z_m)|}{(|x-z_1|+\cdots+|x-z_m|)^{mn}}\,dz_1\cdots dz_m\notag\\
&\lesssim\prod_{k=\ell+1}^m\int_{B(y,2r)}\big|f_k(z_k)\big|\,dz_k\notag\\
&\times\sum_{j=1}^\infty\frac{1}{|B(y,2^{j+1}r)|^m}\int_{B(y,2^{j+1}r)^\ell\backslash B(y,2^{j}r)^\ell}
|(b(z_1)-b_{B})f_1(z_1)|\cdot\big|f_2(z_2)\cdots f_{\ell}(z_\ell)\big|\,dz_1\cdots dz_\ell\notag\\
&\leq \prod_{k=\ell+1}^m\int_{B(y,2r)}\big|f_k(z_k)\big|\,dz_k\notag\\
&\times\sum_{j=1}^\infty\frac{1}{|B(y,2^{j+1}r)|^m}\int_{B(y,2^{j+1}r)}|(b(z_1)-b_{B})f_1(z_1)|\,dz_1
\prod_{k=2}^{\ell}\int_{B(y,2^{j+1}r)}\big|f_k(z_k)\big|\,dz_k\notag\\
&\leq \sum_{j=1}^\infty\bigg(\frac{1}{|B(y,2^{j+1}r)|^m}
\int_{B(y,2^{j+1}r)}|(b(z_1)-b_{B})f_1(z_1)|\,dz_1\prod_{k=2}^m\int_{B(y,2^{j+1}r)}\big|f_k(z_k)\big|\,dz_k\bigg),
\end{align}
where in the last inequality we have used the inclusion relation $2B\subseteq 2^{j+1}B$ with $j\in \mathbb{N}$. Repeating arguments as above show that
\begin{equation*}
\begin{split}
J^{\beta_1,\dots,\beta_m}_{\star\star}(y,r)
&\lesssim\|b\|_\ast\sum_{j=1}^\infty(j+1)\Bigg\{\prod_{k=1}^m\bigg[w_k(B(y,2^{j+1}r))^{1/{\alpha_k}-1/{p_k}-1/{q_k}}
\big\|f_k\cdot\chi_{B(y,2^{j+1}r)}\big\|_{L^{p_k}(w_k)}\bigg]\\
&\times\frac{\nu_{\vec{w}}(B(y,r))^{1/{\alpha}-1/q}}{\nu_{\vec{w}}(B(y,2^{j+1}r))^{1/{\alpha}-1/q}}\Bigg\}.
\end{split}
\end{equation*}
Summarizing the estimates derived above, we get
\begin{align}\label{J3yr}
J^{\beta_1,\dots,\beta_m}(y,r)
&\lesssim\|b\|_\ast\sum_{j=1}^\infty(j+1)\Bigg\{\prod_{k=1}^m\bigg[w_k(B(y,2^{j+1}r))^{1/{\alpha_k}-1/{p_k}-1/{q_k}}
\big\|f_k\cdot\chi_{B(y,2^{j+1}r)}\big\|_{L^{p_k}(w_k)}\bigg]\notag\\
&\times\frac{\nu_{\vec{w}}(B(y,r))^{1/{\alpha}-1/q}}{\nu_{\vec{w}}(B(y,2^{j+1}r))^{1/{\alpha}-1/q}}\Bigg\}.
\end{align}
Therefore, by taking the $L^q({\mu})$-norm of both sides of \eqref{J}(with respect to the variable $y$), and then using Minkowski's inequality($q\geq1$), \eqref{J1yr}, \eqref{J2yr} and \eqref{J3yr}, we have
\begin{equation*}
\begin{split}
&\Big\|\nu_{\vec{w}}(B(y,r))^{1/{\alpha}-1/p-1/q}
\big\|[b,T_\theta]_1(\vec{f})\cdot\chi_{B(y,r)}\big\|_{L^p(\nu_{\vec{w}})}\Big\|_{L^q({\mu})}\\
&\leq\Big\|J^{0,\dots,0}(y,r)\Big\|_{L^q({\mu})}+
\sum_{(\beta_1,\dots,\beta_m)\in\mathfrak{L}}\Big\|J^{\beta_1,\dots,\beta_m}(y,r)\Big\|_{L^q({\mu})}\\
&\leq C\bigg\|\prod_{k=1}^m\bigg[w_k(B(y,2r))^{1/{\alpha_k}-1/{p_k}-1/{q_k}}
\big\|f_k\cdot\chi_{B(y,2r)}\big\|_{L^{p_k}(w_k)}\bigg]\bigg\|_{L^q({\mu})}\\
&+C\cdot2^m\sum_{j=1}^\infty\bigg\|\prod_{k=1}^m\bigg[w_k(B(y,2^{j+1}r))^{1/{\alpha_k}-1/{p_k}-1/{q_k}}
\big\|f_k\cdot\chi_{B(y,2^{j+1}r)}\big\|_{L^{p_k}(w_k)}\bigg]\bigg\|_{L^q({\mu})}\\
&\times(j+1)\left(\frac{|B(y,r)|}{|B(y,2^{j+1}r)|}\right)^{\delta(1/{\alpha}-1/q)},
\end{split}
\end{equation*}
where in the last step we have used the estimate \eqref{psi1}. Another application of H\"older's inequality gives us that
\begin{equation*}
\begin{split}
&\Big\|\nu_{\vec{w}}(B(y,r))^{1/{\alpha}-1/p-1/q}
\big\|[b,T_\theta]_1(\vec{f})\cdot\chi_{B(y,r)}\big\|_{L^p(\nu_{\vec{w}})}\Big\|_{L^q({\mu})}\\
&\leq C\prod_{k=1}^m\Big\|w_k(B(y,2r))^{1/{\alpha_k}-1/{p_k}-1/{q_k}}
\big\|f_k\cdot\chi_{B(y,2r)}\big\|_{L^{p_k}(w_k)}\Big\|_{L^{q_k}({\mu})}\\
&+C\sum_{j=1}^\infty\prod_{k=1}^m\Big\|w_k(B(y,2^{j+1}r))^{1/{\alpha_k}-1/{p_k}-1/{q_k}}
\big\|f_k\cdot\chi_{B(y,2^{j+1}r)}\big\|_{L^{p_k}(w_k)}\Big\|_{L^{q_k}({\mu})}\\
&\times(j+1)\left(\frac{|B(y,r)|}{|B(y,2^{j+1}r)|}\right)^{\delta(1/{\alpha}-1/q)}\\
&\leq C\prod_{k=1}^m\big\|f_k\big\|_{(L^{p_k},L^{q_k})^{\alpha_k}(w_k;\mu)}+
C\prod_{k=1}^m\big\|f_k\big\|_{(L^{p_k},L^{q_k})^{\alpha_k}(w_k;\mu)}
\times\sum_{j=1}^\infty(j+1)\left[\frac{1}{2^{(j+1)n}}\right]^{\delta(1/{\alpha}-1/q)}\\
&\leq C\prod_{k=1}^m\big\|f_k\big\|_{(L^{p_k},L^{q_k})^{\alpha_k}(w_k;\mu)},
\end{split}
\end{equation*}
where the last series is convergent since $\delta>0$ and $1/{\alpha}-1/q>0$. We end the proof by taking the supremum over all $r>0$.
\end{proof}

\begin{proof}[Proof of Theorem $\ref{mainthm:4}$]
Given $\vec{f}=(f_1,f_2,\dots,f_m)$, for any fixed ball $B=B(y,r)$ in $\mathbb R^n$, as before, we split each $f_k$ as
\begin{equation*}
f_k=f_k^0+f_k^{\infty},~~k=1,2,\dots,m,
\end{equation*}
where $f_k^0=f_k\cdot\chi_{2B}$, $f_k^{\infty}=f_k\cdot\chi_{(2B)^{\complement}}$ and $2B=B(y,2r)\subset\mathbb R^n$. Again, we only need to consider here the multilinear commutator with only one symbol by linearity; that is, fix $b\in \mathrm{BMO}(\mathbb R^n)$ and consider the commutator operator
\begin{equation*}
\big[b,T_\theta\big]_1(\vec{f})(x)=b(x)\cdot T_\theta(f_1,f_2,\dots,f_m)(x)-T_{\theta}(bf_1,f_2,\dots,f_m)(x).
\end{equation*}
Let $\mathfrak{L}$ be the same as before. Then for any given $\lambda>0$, by using Lemma \ref{WMin}($N=2^m$), one can write
\begin{align}\label{Jprime}
&\nu_{\vec{w}}(B(y,r))^{1/{\alpha}-m-1/q}\cdot \Big[\nu_{\vec{w}}\big(\big\{x\in B(y,r):\big|\big[{b},T_\theta\big]_1(\vec{f})(x)\big|>\lambda^m\big\}\big)\Big]^m\notag\\
\leq & C\cdot\nu_{\vec{w}}(B(y,r))^{1/{\alpha}-m-1/q}\cdot
\Big[\nu_{\vec{w}}\big(\big\{x\in B(y,r):\big|\big[{b},T_\theta\big]_1(f^0_1,\dots,f^0_m)(x)\big|>\lambda^m/{2^m}\big\}\big)\Big]^m\notag\\
+&C\sum_{(\beta_1,\dots,\beta_m)\in\mathfrak{L}}\nu_{\vec{w}}(B(y,r))^{1/{\alpha}-m-1/q}
\Big[\nu_{\vec{w}}\big(\big\{x\in B(y,r):\big|\big[{b},T_\theta\big]_1(f^{\beta_1}_1,\ldots,f^{\beta_m}_m)(x)\big|>\lambda^m/{2^m}\big\}\big)\Big]^m\notag\\
:=&J^{0,\dots,0}_\ast(y,r)+\sum_{(\beta_1,\dots,\beta_m)\in\mathfrak{L}} J^{\beta_1,\dots,\beta_m}_\ast(y,r).
\end{align}
Observe that the Young function $\Phi(t)=t\cdot(1+\log^+t)$ satisfies the doubling condition, that is, there is a constant $C_{\Phi}>0$ such that for every $t>0$,
\begin{equation*}
\Phi(2t)\leq C_{\Phi}\,\Phi(t).
\end{equation*}
This fact together with Theorem \ref{Wcomm} and inequality \eqref{main esti1} implies
\begin{equation*}
\begin{split}
J^{0,\dots,0}_\ast(y,r)&\leq C\cdot\nu_{\vec{w}}(B(y,r))^{1/{\alpha}-m-1/q}
\prod_{k=1}^m\bigg(\int_{\mathbb R^n}\Phi\bigg(\frac{2|f^0_k(x)|}{\lambda}\bigg)\cdot w_k(x)\,dx\bigg)\\
&\leq C\cdot\nu_{\vec{w}}(B(y,r))^{1/{\alpha}-m-1/q}
\prod_{k=1}^m\bigg(\int_{B(y,2r)}\Phi\bigg(\frac{|f_k(x)|}{\lambda}\bigg)\cdot w_k(x)\,dx\bigg)\\
&= C\cdot\nu_{\vec{w}}(B(y,r))^{1/{\alpha}-m-1/q}
\prod_{k=1}^m w_k(B(y,2r))\\
&\times\bigg[\frac{1}{w_k(B(y,2r))}\int_{B(y,2r)}\Phi\bigg(\frac{|f_k(x)|}{\lambda}\bigg)\cdot w_k(x)\,dx\bigg]\\
&\leq C\cdot\nu_{\vec{w}}(B(y,r))^{1/{\alpha}-m-1/q}
\prod_{k=1}^m w_k(B(y,2r))
\cdot\bigg\|\Phi\bigg(\frac{|f_k|}{\lambda}\bigg)\bigg\|_{L\log L(w_k),B(y,2r)}.
\end{split}
\end{equation*}
Since $\vec{w}=(w_1,\ldots,w_m)\in A_{(1,\dots,1)}$, by definition, we know that
\begin{equation}\label{A11}
\bigg(\frac{1}{|\mathcal{B}|}\int_{\mathcal{B}}\nu_{\vec{w}}(x)\,dx\bigg)^{m}
\leq C\prod_{k=1}^m\inf_{x\in \mathcal{B}}w_k(x)
\end{equation}
holds for any ball $\mathcal{B}=\mathcal{B}(y,r)$ in $\mathbb R^n$ with $y\in\mathbb R^n$ and $r>0$, where $\nu_{\vec{w}}=\prod_{k=1}^m w_k^{1/{m}}$. We can rewrite this inequality as
\begin{equation*}
\begin{split}
\bigg(\frac{1}{|\mathcal{B}|}\int_{\mathcal{B}}\nu_{\vec{w}}(x)\,dx\bigg)
&\leq C\bigg(\prod_{k=1}^m\inf_{x\in \mathcal{B}}w_k(x)\bigg)^{1/m}
=C\bigg(\prod_{k=1}^m\inf_{x\in \mathcal{B}}w_k(x)^{1/m}\bigg)\\
&\leq C\bigg(\inf_{x\in \mathcal{B}}\prod_{k=1}^mw_k(x)^{1/m}\bigg)=C\cdot\inf_{x\in \mathcal{B}}\nu_{\vec{w}}(x),
\end{split}
\end{equation*}
which means that $\nu_{\vec{w}}\in A_1$. Moreover, for each $w_k$, $k=1,2,\dots,m$, it is easy to see that
\begin{equation*}
\begin{split}
\bigg(\prod_{j\neq k}\inf_{x\in \mathcal{B}}w_j(x)^{1/{m}}\bigg)^m\bigg(\frac{1}{|\mathcal{B}|}\int_{\mathcal{B}}w_k(x)^{1/{m}}\,dx\bigg)^{m}
&\leq\bigg(\frac{1}{|\mathcal{B}|}\int_{\mathcal{B}}w_k(x)^{1/{m}}\cdot\prod_{j\neq k}w_j(x)^{1/{m}}\,dx\bigg)^{m}\\
&\leq C\prod_{j=1}^m\inf_{x\in \mathcal{B}}w_j(x).
\end{split}
\end{equation*}
Also observe that
\begin{equation*}
\bigg(\prod_{j\neq k}\inf_{x\in \mathcal{B}}w_j(x)^{1/{m}}\bigg)^m=\prod_{j\neq k}\inf_{x\in \mathcal{B}}w_j(x).
\end{equation*}
From this, it follows that
\begin{equation*}
\bigg(\frac{1}{|\mathcal{B}|}\int_{\mathcal{B}}w_k(x)^{1/{m}}\,dx\bigg)^{m}\leq C\cdot\inf_{x\in \mathcal{B}}w_k(x),
\end{equation*}
which implies that $w_k^{1/{m}}\in A_1$ ($k=1,2,\dots,m$). Recall that for any ball $\mathcal{B}=\mathcal{B}(y,r)$ in $\mathbb R^n$, the following result holds (taking $p_1=\cdots=p_m=1$ and $p=1/m$ in \eqref{C1}):
\begin{equation}\label{C2}
\prod_{k=1}^m w_k(\mathcal{B}(y,r))^{1+1/{q_k}-1/{\alpha_k}}
\lesssim\nu_{\vec{w}}(\mathcal{B}(y,r))^{m+1/q-1/{\alpha}}.
\end{equation}
Hence,
\begin{equation*}
\begin{split}
J^{0,\dots,0}_\ast(y,r)
&\leq C\cdot\nu_{\vec{w}}(B(y,r))^{1/{\alpha}-m-1/q}
\prod_{k=1}^m w_k(B(y,2r))^{1+1/{q_k}-1/{\alpha_k}}\\
&\times\prod_{k=1}^m\bigg[w_k(B(y,2r))^{1/{\alpha_k}-1/{q_k}}
\bigg\|\Phi\bigg(\frac{|f_k|}{\lambda}\bigg)\bigg\|_{L\log L(w_k),B(y,2r)}\bigg]\\
&\leq C\cdot\frac{\nu_{\vec{w}}(B(y,2r))^{m+1/q-1/{\alpha}}}{\nu_{\vec{w}}(B(y,r))^{m+1/q-1/{\alpha}}}\\
&\times\prod_{k=1}^m\bigg[w_k(B(y,2r))^{1/{\alpha_k}-1/{q_k}}
\bigg\|\Phi\bigg(\frac{|f_k|}{\lambda}\bigg)\bigg\|_{L\log L(w_k),B(y,2r)}\bigg].
\end{split}
\end{equation*}
Moreover, since $\nu_{\vec{w}}$ is in $A_{1}$ and $m+1/q-1/{\alpha}>0$, then by inequality \eqref{weights}, we have
\begin{equation*}
\begin{split}
J^{0,\dots,0}_\ast(y,r)
&\leq C\prod_{k=1}^m\bigg[w_k(B(y,2r))^{1/{\alpha_k}-1/{q_k}}
\bigg\|\Phi\bigg(\frac{|f_k|}{\lambda}\bigg)\bigg\|_{L\log L(w_k),B(y,2r)}\bigg].
\end{split}
\end{equation*}
To estimate the other terms $J^{\beta_1,\dots,\beta_m}_\ast(y,r)$ for $(\beta_1,\dots,\beta_m)\in\mathfrak{L}$, we remark that for any $x\in B(y,r)$,
\begin{equation*}
\big[b,T_\theta\big]_1(\vec{f})(x)=[b(x)-b_B]\cdot T_\theta(f_1,f_2,\dots,f_m)(x)-T_{\theta}((b-b_B)f_1,f_2,\dots,f_m)(x).
\end{equation*}
So we can further decompose $J^{\beta_1,\dots,\beta_m}_\ast(y,r)$ as
\begin{equation*}
\begin{split}
&J^{\beta_1,\dots,\beta_m}_\ast(y,r)\\
\leq&C\nu_{\vec{w}}(B(y,r))^{1/{\alpha}-m-1/q}
\Big[\nu_{\vec{w}}\Big(\Big\{x\in B(y,r):\big|[b(x)-b_B]\cdot T_\theta(f^{\beta_1}_1,f^{\beta_2}_2,\ldots,f^{\beta_m}_m)(x)\big|>\lambda^m/{2^{m+1}}\Big\}\Big)\Big]^m\\
+&C\nu_{\vec{w}}(B(y,r))^{1/{\alpha}-m-1/q}\cdot
\Big[\nu_{\vec{w}}\Big(\Big\{x\in B(y,r):\big|T_{\theta}((b-b_B)f^{\beta_1}_1,f^{\beta_2}_2,\dots,f^{\beta_m}_m)(x)\big|>\lambda^m/{2^{m+1}}\Big\}\Big)\Big]^m\\
:=&\widetilde{J}^{\beta_1,\dots,\beta_m}_\star(y,r)+\widetilde{J}^{\beta_1,\dots,\beta_m}_{\star\star}(y,r).
\end{split}
\end{equation*}
By using the previous pointwise estimates \eqref{111} and \eqref{112} together with Chebyshev's inequality, we can deduce that
\begin{equation*}
\begin{split}
\widetilde{J}^{\beta_1,\dots,\beta_m}_\star(y,r)&\leq C\nu_{\vec{w}}(B(y,r))^{1/{\alpha}-m-1/q}\\
&\times\frac{2^{m+1}}{\lambda^m}
\bigg(\int_{B(y,r)}\big|[b(x)-b_{B(y,r)}]\cdot T_\theta(f^{\beta_1}_1,f^{\beta_2}_2,\ldots,f^{\beta_m}_m)(x)\big|^{\frac{\,1\,}{m}}\nu_{\vec{w}}(x)\,dx\bigg)^m\\
&\leq C\cdot\nu_{\vec{w}}(B(y,r))^{1/{\alpha}-m-1/q}
\sum_{j=1}^\infty\bigg(\prod_{k=1}^m\frac{1}{|B(y,2^{j+1}r)|}
\int_{B(y,2^{j+1}r)}\frac{|f_k(z_k)|}{\lambda}\,dz_k\bigg)\\
&\times\bigg(\int_{B(y,r)}\big|b(x)-b_{B(y,r)}\big|^{\frac{\,1\,}{m}}\nu_{\vec{w}}(x)\,dx\bigg)^m.
\end{split}
\end{equation*}
Now we claim that for $2\leq m\in\mathbb{N}$ and $\nu_{\vec{w}}\in A_1$,
\begin{equation}\label{assertion}
\bigg(\int_{B(y,r)}\big|b(x)-b_{B(y,r)}\big|^{\frac{\,1\,}{m}}\nu_{\vec{w}}(x)\,dx\bigg)^m\lesssim\|b\|_*\cdot\nu_{\vec{w}}(B(y,r))^{m}.
\end{equation}
Taking this claim momentarily for granted, then we have
\begin{equation*}
\widetilde{J}^{\beta_1,\dots,\beta_m}_\star(y,r)\lesssim\|b\|_*\cdot\nu_{\vec{w}}(B(y,r))^{1/{\alpha}-1/q}
\sum_{j=1}^\infty\bigg(\prod_{k=1}^m\frac{1}{|B(y,2^{j+1}r)|}
\int_{B(y,2^{j+1}r)}\frac{|f_k(z_k)|}{\lambda}\,dz_k\bigg).
\end{equation*}
Furthermore, note that $t\leq\Phi(t)=t\cdot(1+\log^+t)$ for any $t>0$. It then follows from the multiple $A_{(1,\dots,1)}$ condition \eqref{A11} and the previous estimate \eqref{main esti1} that
\begin{equation*}
\begin{split}
\widetilde{J}^{\beta_1,\dots,\beta_m}_\star(y,r)
&\lesssim\|b\|_*\cdot\nu_{\vec{w}}(B(y,r))^{1/{\alpha}-1/q}\\
&\times\sum_{j=1}^\infty\prod_{k=1}^m\bigg(\frac{1}{|B(y,2^{j+1}r)|}
\int_{B(y,2^{j+1}r)}\frac{|f_k(z_k)|}{\lambda}\cdot w_k(z_k)\,dz_k\bigg)
\left(\inf_{z_k\in B(y,2^{j+1}r)}w_k(z_k)\right)^{-1}\\
&\lesssim\|b\|_*\cdot\nu_{\vec{w}}(B(y,r))^{1/{\alpha}-1/q}\\
&\times\sum_{j=1}^\infty\frac{1}{\nu_{\vec{w}}(B(y,2^{j+1}r))^m}
\prod_{k=1}^m\int_{B(y,2^{j+1}r)}\Phi\bigg(\frac{|f_k(z_k)|}{\lambda}\bigg)\cdot w_k(z_k)\,dz_k\\
&\lesssim\|b\|_*\cdot\nu_{\vec{w}}(B(y,r))^{1/{\alpha}-1/q}\\
&\times\sum_{j=1}^\infty\frac{1}{\nu_{\vec{w}}(B(y,2^{j+1}r))^m}
\prod_{k=1}^mw_k\big(B(y,2^{j+1}r)\big)\bigg\|\Phi\bigg(\frac{|f_k|}{\lambda}\bigg)\bigg\|_{L\log L(w_k),B(y,2^{j+1}r)}\\
&=\|b\|_\ast\cdot\nu_{\vec{w}}(B(y,r))^{1/{\alpha}-1/q}
\sum_{j=1}^\infty\frac{1}{\nu_{\vec{w}}(B(y,2^{j+1}r))^m}\prod_{k=1}^m w_k\big(B(y,2^{j+1}r)\big)^{1+1/{q_k}-1/{\alpha_k}}\\
&\times\prod_{k=1}^m\bigg[w_k\big(B(y,2^{j+1}r)\big)^{1/{\alpha_k}-1/{q_k}}
\bigg\|\Phi\bigg(\frac{|f_k|}{\lambda}\bigg)\bigg\|_{L\log L(w_k),B(y,2^{j+1}r)}\bigg].
\end{split}
\end{equation*}
Therefore, in view of \eqref{C2} and \eqref{compare}, we conclude that
\begin{align}\label{WJ2yr}
\widetilde{J}^{\beta_1,\dots,\beta_m}_\star(y,r)
&\lesssim\|b\|_\ast\cdot\sum_{j=1}^\infty\bigg\{\prod_{k=1}^m\bigg[w_k\big(B(y,2^{j+1}r)\big)^{1/{\alpha_k}-1/{q_k}}
\bigg\|\Phi\bigg(\frac{|f_k|}{\lambda}\bigg)\bigg\|_{L\log L(w_k),B(y,2^{j+1}r)}\bigg]\notag\\
&\times\frac{\nu_{\vec{w}}(B(y,r))^{1/{\alpha}-1/q}}{\nu_{\vec{w}}(B(y,2^{j+1}r))^{1/{\alpha}-1/q}}\bigg\}\notag\\
&\lesssim\|b\|_\ast\cdot\sum_{j=1}^\infty\bigg\{\prod_{k=1}^m\bigg[w_k\big(B(y,2^{j+1}r)\big)^{1/{\alpha_k}-1/{q_k}}
\bigg\|\Phi\bigg(\frac{|f_k|}{\lambda}\bigg)\bigg\|_{L\log L(w_k),B(y,2^{j+1}r)}\bigg]\notag\\
&\times\left(\frac{|B(y,r)|}{|B(y,2^{j+1}r)|}\right)^{\delta(1/{\alpha}-1/q)}.
\end{align}
Let us return to the proof of \eqref{assertion}. Since $\nu_{\vec{w}}\in A_1$, we know that $\nu_{\vec{w}}$ belongs to the reverse H\"{o}lder class $RH_s$ for some $1<s<\infty$(see \cite{duoand} and \cite{grafakos3}). Here the reverse H\"{o}lder class is defined in the following way: $\omega\in RH_s$, if there is a constant $C>0$ such that
\begin{equation*}
\bigg(\frac{1}{|B|}\int_B\omega(x)^s\,dx\bigg)^{1/s}\leq C\bigg(\frac{1}{|B|}\int_B\omega(x)\,dx\bigg).
\end{equation*}
This fact together with H\"older's inequality implies that
\begin{equation*}
\begin{split}
&\int_{B(y,r)}\big|b(x)-b_{B(y,r)}\big|^{\frac{\,1\,}{m}}\nu_{\vec{w}}(x)\,dx\\
&\leq|B(y,r)|\bigg(\frac{1}{|B(y,r)|}\int_{B(y,r)}\big|b(x)-b_{B(y,r)}\big|^{\frac{s'}{m}}\,dx\bigg)^{1/s'}
\bigg(\frac{1}{|B(y,r)|}\int_{B(y,r)}\nu_{\vec{w}}(x)^s\,dx\bigg)^{1/s}\\
&\leq C\nu_{\vec{w}}(B(y,r))\bigg(\frac{1}{|B(y,r)|}\int_{B(y,r)}\big|b(x)-b_{B(y,r)}\big|^{\frac{s'}{m}}\,dx\bigg)^{1/s'}.
\end{split}
\end{equation*}
This can be done by considering the following two cases: $s'/m<1$ and $s'/m\geq1$. If $s'/m<1$, then \eqref{assertion} holds by using H\"older's inequality again. If $s'/m\geq1$, then \eqref{assertion} holds by applying $(ii)$ of Lemma \ref{BMO}. On the other hand, applying the pointwise estimates \eqref{pointwise3},\eqref{pointwise5} and Chebyshev's inequality, we have
\begin{equation*}
\begin{split}
\widetilde{J}^{\beta_1,\dots,\beta_m}_{\star\star}(y,r)
&\leq C\cdot\nu_{\vec{w}}(B(y,r))^{1/{\alpha}-m-1/q}\\
&\times\frac{2^{m+1}}{\lambda^m}
\bigg(\int_{B(y,r)}\big|T_\theta((b-b_{B(y,r)})f^{\beta_1}_1,f^{\beta_2}_2,\ldots,f^{\beta_m}_m)(x)\big|^{\frac{\,1\,}{m}}
\nu_{\vec{w}}(x)\,dx\bigg)^m\\
&\leq C\cdot\nu_{\vec{w}}(B(y,r))^{1/{\alpha}-1/q}
\sum_{j=1}^\infty\bigg(\prod_{k=2}^m\frac{1}{|B(y,2^{j+1}r)|}
\int_{B(y,2^{j+1}r)}\frac{|f_k(z_k)|}{\lambda}\,dz_k\bigg)\\
&\times\bigg(\frac{1}{|B(y,2^{j+1}r)|}\int_{B(y,2^{j+1}r)}
\big|b(z_1)-b_{B(y,r)}\big|\cdot\frac{|f_1(z_1)|}{\lambda}\,dz_1\bigg)\\
&\leq C\cdot\nu_{\vec{w}}(B(y,r))^{1/{\alpha}-1/q}
\sum_{j=1}^\infty\bigg(\prod_{k=2}^m\frac{1}{|B(y,2^{j+1}r)|}
\int_{B(y,2^{j+1}r)}\frac{|f_k(z_k)|}{\lambda}w_k(z_k)\,dz_k\bigg)\\
&\times\bigg(\frac{1}{|B(y,2^{j+1}r)|}\int_{B(y,2^{j+1}r)}
\big|b(z_1)-b_{B(y,r)}\big|\cdot\frac{|f_1(z_1)|}{\lambda}w_1(z_1)\,dz_1\bigg)\\
&\times\prod_{k=1}^m\left(\inf_{z_k\in B(y,2^{j+1}r)}w_k(z_k)\right)^{-1}\\
&\leq C\cdot\nu_{\vec{w}}(B(y,r))^{1/{\alpha}-1/q}
\times\sum_{j=1}^\infty\frac{1}{\nu_{\vec{w}}(B(y,2^{j+1}r))^m}
\bigg(\prod_{k=2}^m\int_{B(y,2^{j+1}r)}\frac{|f_k(z_k)|}{\lambda}w_k(z_k)\,dz_k\bigg)\\
&\times\bigg(\int_{B(y,2^{j+1}r)}\big|b(z_1)-b_{B(y,r)}\big|\cdot\frac{|f_1(z_1)|}{\lambda}w_1(z_1)\,dz_1\bigg),
\end{split}
\end{equation*}
where in the last inequality we have used \eqref{A11}. In addition, by \eqref{main esti1} and the fact that $t\leq \Phi(t)$, we get
\begin{equation*}
\begin{split}
&\int_{B(y,2^{j+1}r)}\frac{|f_k(z_k)|}{\lambda}w_k(z_k)\,dz_k\\
&\leq\int_{B(y,2^{j+1}r)}\Phi\bigg(\frac{|f_k(z_k)|}{\lambda}\bigg)\cdot w_k(z_k)\,dz_k\\
&\leq w_k\big(B(y,2^{j+1}r)\big)\bigg\|\Phi\bigg(\frac{|f_k|}{\lambda}\bigg)\bigg\|_{L\log L(w_k),B(y,2^{j+1}r)},
\end{split}
\end{equation*}
and by the inequality \eqref{Wholder} and the fact that $t\leq \Phi(t)$, we thus obtain
\begin{equation*}
\begin{split}
&\int_{B(y,2^{j+1}r)}\big|b(z_1)-b_{B(y,r)}\big|\cdot\frac{|f_1(z_1)|}{\lambda}w_1(z_1)\,dz_1\\
&\leq\int_{B(y,2^{j+1}r)}\big|b(z_1)-b_{B(y,r)}\big|\cdot\Phi\bigg(\frac{|f_1(z_1)|}{\lambda}\bigg)w_1(z_1)\,dz_1\\
&\leq C\cdot w_1\big(B(y,2^{j+1}r)\big)
\Big\|b-b_{B(y,r)}\Big\|_{\exp L(w_1),B(y,2^{j+1}r)}
\bigg\|\Phi\bigg(\frac{|f_1|}{\lambda}\bigg)\bigg\|_{L\log L(w_1),B(y,2^{j+1}r)}.
\end{split}
\end{equation*}
Furthermore, by the estimate \eqref{Jensen} and the assumption $w_1\in A_{\infty}$, the last expression is dominated by
\begin{equation*}
\begin{split}
&C(j+1)\|b\|_\ast\cdot w_1\big(B(y,2^{j+1}r)\big)
\bigg\|\Phi\bigg(\frac{|f_1|}{\lambda}\bigg)\bigg\|_{L\log L(w_1),B(y,2^{j+1}r)}.
\end{split}
\end{equation*}
Consequently, by combining the above two estimates, we conclude that
\begin{equation*}
\begin{split}
\widetilde{J}^{\beta_1,\dots,\beta_m}_{\star\star}(y,r)
&\leq C\|b\|_\ast\nu_{\vec{w}}(B(y,r))^{1/{\alpha}-1/q}\\
&\times\sum_{j=1}^\infty
\left\{\frac{(j+1)}{\nu_{\vec{w}}(B(y,2^{j+1}r))^m}
\prod_{k=1}^mw_k\big(B(y,2^{j+1}r)\big)\bigg\|\Phi\bigg(\frac{|f_k|}{\lambda}\bigg)\bigg\|_{L\log L(w_k),B(y,2^{j+1}r)}\right\}\\
&=C\|b\|_\ast\nu_{\vec{w}}(B(y,r))^{1/{\alpha}-1/q}
\sum_{j=1}^\infty\frac{(j+1)}{\nu_{\vec{w}}(B(y,2^{j+1}r))^m}
\prod_{k=1}^m w_k\big(B(y,2^{j+1}r)\big)^{1+1/{q_k}-1/{\alpha_k}}\\
&\times\prod_{k=1}^m\bigg[w_k\big(B(y,2^{j+1}r)\big)^{1/{\alpha_k}-1/{q_k}}
\bigg\|\Phi\bigg(\frac{|f_k|}{\lambda}\bigg)\bigg\|_{L\log L(w_k),B(y,2^{j+1}r)}\bigg].
\end{split}
\end{equation*}
By using \eqref{C2} and \eqref{compare} again, we thus obtain
\begin{align}\label{WJ3yr}
\widetilde{J}^{\beta_1,\dots,\beta_m}_{\star\star}(y,r)
&\lesssim\|b\|_\ast\cdot\sum_{j=1}^\infty
\Bigg\{\prod_{k=1}^m\bigg[w_k\big(B(y,2^{j+1}r)\big)^{1/{\alpha_k}-1/{q_k}}
\bigg\|\Phi\bigg(\frac{|f_k|}{\lambda}\bigg)\bigg\|_{L\log L(w_k),B(y,2^{j+1}r)}\bigg]\notag\\
&\times\big(j+1\big)\frac{\nu_{\vec{w}}(B(y,r))^{1/{\alpha}-1/q}}{\nu_{\vec{w}}(B(y,2^{j+1}r))^{1/{\alpha}-1/q}}\Bigg\}\notag\\
&\lesssim\|b\|_\ast\cdot\sum_{j=1}^\infty
\Bigg\{\prod_{k=1}^m\bigg[w_k\big(B(y,2^{j+1}r)\big)^{1/{\alpha_k}-1/{q_k}}
\bigg\|\Phi\bigg(\frac{|f_k|}{\lambda}\bigg)\bigg\|_{L\log L(w_k),B(y,2^{j+1}r)}\bigg]\notag\\
&\times\big(j+1\big)\left(\frac{|B(y,r)|}{|B(y,2^{j+1}r)|}\right)^{\delta(1/{\alpha}-1/q)}.
\end{align}
Therefore by taking the $L^q({\mu})$-norm of both sides of \eqref{Jprime}(with respect to the variable $y$), and then using Minkowski's inequality($q\geq1$), \eqref{WJ2yr} and \eqref{WJ3yr}, we have
\begin{equation*}
\begin{split}
&\Big\|\nu_{\vec{w}}(B(y,r))^{1/{\alpha}-m-1/q}\cdot \Big[\nu_{\vec{w}}\big(\big\{x\in B(y,r):\big|\big[{b},T_\theta\big]_1(\vec{f})(x)\big|>\lambda^m\big\}\big)\Big]^m\Big\|_{L^q({\mu})}\\
&\leq\Big\|J^{0,\dots,0}_\ast(y,r)\Big\|_{L^q({\mu})}+
\sum_{(\beta_1,\dots,\beta_m)\in\mathfrak{L}}\Big\|J^{\beta_1,\dots,\beta_m}_\ast(y,r)\Big\|_{L^q({\mu})}\\
&\leq C\bigg\|\prod_{k=1}^m
\bigg[w_k(B(y,2r))^{1/{\alpha_k}-1/{q_k}}
\bigg\|\Phi\bigg(\frac{|f_k|}{\lambda}\bigg)\bigg\|_{L\log L(w_k),B(y,2r)}\bigg]\bigg\|_{L^q({\mu})}\\
&+C\cdot2^m\sum_{j=1}^\infty\bigg\|\prod_{k=1}^m\bigg[w_k\big(B(y,2^{j+1}r)\big)^{1/{\alpha_k}-1/{q_k}}
\bigg\|\Phi\bigg(\frac{|f_k|}{\lambda}\bigg)\bigg\|_{L\log L(w_k),B(y,2^{j+1}r)}\bigg]\bigg\|_{L^q({\mu})}\\
&\times\big(j+1\big)\left(\frac{|B(y,r)|}{|B(y,2^{j+1}r)|}\right)^{\delta(1/{\alpha}-1/q)}.
\end{split}
\end{equation*}
Another application of H\"older's inequality gives us that
\begin{equation*}
\begin{split}
&\Big\|\nu_{\vec{w}}(B(y,r))^{1/{\alpha}-m-1/q}\cdot\Big[\nu_{\vec{w}}\big(\big\{x\in B(y,r):\big|\big[{b},T_\theta\big]_1(\vec{f})(x)\big|>\lambda^m\big\}\big)\Big]^m\Big\|_{L^q(\mu)}\\
&\leq C\prod_{k=1}^m\bigg\|w_k(B(y,2r))^{1/{\alpha_k}-1/{q_k}}
\bigg\|\Phi\bigg(\frac{|f_k|}{\lambda}\bigg)\bigg\|_{L\log L(w_k),B(y,2r)}\bigg\|_{L^{q_k}(\mu)}\\
&+C\sum_{j=1}^\infty\prod_{k=1}^m\bigg\|w_k\big(B(y,2^{j+1}r)\big)^{1/{\alpha_k}-1/{q_k}}
\bigg\|\Phi\bigg(\frac{|f_k|}{\lambda}\bigg)\bigg\|_{L\log L(w_k),B(y,2^{j+1}r)}\bigg\|_{L^{q_k}(\mu)}\\
&\times\big(j+1\big)\left(\frac{|B(y,r)|}{|B(y,2^{j+1}r)|}\right)^{\delta(1/{\alpha}-1/q)}\\
&\leq C\prod_{k=1}^m\bigg\|\Phi\bigg(\frac{|f_k|}{\lambda}\bigg)\bigg\|_{(L\log L,L^{q_k})^{\alpha_k}(w_k;\mu)}\\
&+C\prod_{k=1}^m\bigg\|\Phi\bigg(\frac{|f_k|}{\lambda}\bigg)\bigg\|_{(L\log L,L^{q_k})^{\alpha_k}(w_k;\mu)}
\times\sum_{j=1}^\infty
\big(j+1\big)\left(\frac{|B(y,r)|}{|B(y,2^{j+1}r)|}\right)^{\delta(1/{\alpha}-1/q)}\\
&\leq C\prod_{k=1}^m\bigg\|\Phi\bigg(\frac{|f_k|}{\lambda}\bigg)\bigg\|_{(L\log L,L^{q_k})^{\alpha_k}(w_k;\mu)},
\end{split}
\end{equation*}
where the last inequality holds since $\delta(1/{\alpha}-1/q)>0$. This completes the proof of Theorem \ref{mainthm:4}.
\end{proof}
For the iterated commutator $\big[\Pi\vec{b},T_\theta\big]$, we can also establish the following results in the same manner as in Theorems \ref{mainthm:3} and \ref{mainthm:4}. The proof then needs appropriate but minor modifications and we leave this to the reader.
\begin{theorem}\label{mainthm:5}
Let $m\geq2$ and $\big[\Pi\vec{b},T_\theta\big]$ be the iterated commutator of $\theta$-type Calder\'on--Zygmund operator $T_{\theta}$ with $\theta$ satisfying the condition \eqref{theta1} and $\vec{b}\in\mathrm{BMO}^m$. Assume that $1<p_k\leq\alpha_k<q_k<\infty$, $k=1,2,\ldots,m$ and $p\in(1/m,\infty)$ with $1/p=\sum_{k=1}^m 1/{p_k}$, $q\in[1,\infty)$ with $1/q=\sum_{k=1}^m 1/{q_k}$ and $1/{\alpha}=\sum_{k=1}^m 1/{\alpha_k};$ $\vec{w}=(w_1,\ldots,w_m)\in A_{\vec{P}}$ with $w_1,\ldots,w_m\in A_\infty$ and $\mu\in\Delta_2$. Assume further that \eqref{supp} holds. Then there exists a constant $C>0$ such that for all $\vec{f}=(f_1,\ldots,f_m)\in(L^{p_1},L^{q_1})^{\alpha_1}(w_1;\mu)\times\cdots
\times(L^{p_m},L^{q_m})^{\alpha_m}(w_m;\mu)$,
\begin{equation*}
\big\|\big[\Pi\vec{b},T_\theta\big](\vec{f})\big\|_{(L^p,L^q)^{\alpha}(\nu_{\vec{w}};\mu)}\leq C\prod_{k=1}^m\big\|f_k\big\|_{(L^{p_k},L^{q_k})^{\alpha_k}(w_k;\mu)}
\end{equation*}
with $\nu_{\vec{w}}=\prod_{k=1}^m w_k^{p/{p_k}}$.
\end{theorem}

\begin{theorem}\label{mainthm:6}
Let $m\geq2$ and $\big[\Pi\vec{b},T_\theta\big]$ be the iterated commutator of $\theta$-type Calder\'on--Zygmund operator $T_{\theta}$ with $\theta$ satisfying the condition \eqref{theta3} and $\vec{b}\in \mathrm{BMO}^m$. Assume that $p_k=1$, $1\leq\alpha_k<q_k<\infty$, $k=1,2,\ldots,m$ and $p=1/m$, $q\in[1,\infty)$ with $1/q=\sum_{k=1}^m 1/{q_k}$ and $1/{\alpha}=\sum_{k=1}^m 1/{\alpha_k};$ $\vec{w}=(w_1,\ldots,w_m)\in A_{(1,\dots,1)}$ with $w_1,\ldots,w_m\in A_\infty$ and $\mu\in\Delta_2$. Assume further that \eqref{supp} holds. Then for any given $\lambda>0$ and any ball $B(y,r)\subset\mathbb R^n$ with $(y,r)\in\mathbb R^n\times(0,+\infty)$, there exists a constant $C>0$ independent of $\vec{f}=(f_1,\ldots,f_m)$, $B(y,r)$ and $\lambda$ such that
\begin{equation*}
\begin{split}
&\Big\|\nu_{\vec{w}}(B(y,r))^{1/{\alpha}-m-1/q}\cdot\Big[\nu_{\vec{w}}\Big(\Big\{x\in B(y,r):\big|\big[\Pi\vec{b},T_\theta\big](\vec{f})(x)\big|>\lambda^m\Big\}\Big)\Big]^m\Big\|_{L^q(\mu)}\\
&\leq C\cdot\prod_{k=1}^m\bigg\|\Phi^{(m)}\bigg(\frac{|f_k|}{\lambda}\bigg)\bigg\|_{(L\log L,L^{q_k})^{\alpha_k}(w_k;\mu)},
\end{split}
\end{equation*}
where $\nu_{\vec{w}}=\prod_{k=1}^m w_k^{1/{m}}$, $\Phi(t)=t\cdot(1+\log^+t)$ and $\Phi^{(m)}=\overbrace{\Phi\circ\cdots\circ\Phi}^m$. Here the norm $\|\cdot\|_{L^q(\mu)}$ is taken with respect to the variable $y$, i.e.,
\begin{equation*}
\begin{split}
&\Big\|\nu_{\vec{w}}(B(y,r))^{1/{\alpha}-m-1/q}\cdot\Big[\nu_{\vec{w}}\Big(\Big\{x\in B(y,r):\big|\big[\Pi\vec{b},T_\theta\big](\vec{f})(x)\big|>\lambda^m\Big\}\Big)\Big]^m\Big\|_{L^q(\mu)}\\
=&\left\{\int_{\mathbb R^n}\bigg[\nu_{\vec{w}}(B(y,r))^{1/{\alpha}-m-1/q}\cdot\Big[\nu_{\vec{w}}\Big(\Big\{x\in B(y,r):\big|\big[\Pi\vec{b},T_\theta\big](\vec{f})(x)\big|>\lambda^m\Big\}\Big)\Big]^m\bigg]^q\mu(y)\,dy\right\}^{1/q}.
\end{split}
\end{equation*}
\end{theorem}

\begin{rek}
In light of Theorem \ref{mainthm:6}, we can say that the iterated commutator $\big[\Pi\vec{b},T_\theta\big]$ is bounded from $(L\log L,L^{q_1})^{\alpha_1}(w_1;\mu)\times\cdots\times(L\log L,L^{q_m})^{\alpha_m}(w_m;\mu)$ into $(WL^p,L^q)^{\alpha}(\nu_{\vec{w}};\mu)$ with $p=1/m$.
\end{rek}

Finally, in view of the relation \eqref{include}, we have the following results concerning weighted estimates for the multilinear commutators $\big[\Sigma\vec{b},T_\theta\big]$ and $\big[\Pi\vec{b},T_\theta\big]$.
\begin{corollary}
Let $m\geq2$ and $\vec{b}\in \mathrm{BMO}^m$. Assume that $1<p_k\leq\alpha_k<q_k<\infty$, $k=1,2,\ldots,m$ and $p\in(1/m,\infty)$ with $1/p=\sum_{k=1}^m 1/{p_k}$, $q\in[1,\infty)$ with $1/q=\sum_{k=1}^m 1/{q_k}$ and $1/{\alpha}=\sum_{k=1}^m 1/{\alpha_k};$ $\vec{w}=(w_1,\ldots,w_m)\in\prod_{k=1}^m A_{p_k}$ and $\mu\in\Delta_2$. In addition, assume that \eqref{supp} holds. Then both the multilinear commutator $\big[\Sigma\vec{b},T_\theta\big]$ and the iterated commutator $\big[\Pi\vec{b},T_\theta\big]$ satisfy
\begin{equation*}
\big\|\big[\Sigma\vec{b},T_\theta\big](\vec{f})\big\|_{(L^p,L^q)^{\alpha}(\nu_{\vec{w}};\mu)}\leq C\prod_{k=1}^m\big\|f_k\big\|_{(L^{p_k},L^{q_k})^{\alpha_k}(w_k;\mu)}
\end{equation*}
and
\begin{equation*}
\big\|\big[\Pi\vec{b},T_\theta\big](\vec{f})\big\|_{(L^p,L^q)^{\alpha}(\nu_{\vec{w}};\mu)}\leq C\prod_{k=1}^m\big\|f_k\big\|_{(L^{p_k},L^{q_k})^{\alpha_k}(w_k;\mu)}
\end{equation*}
with $\nu_{\vec{w}}=\prod_{k=1}^m w_k^{p/{p_k}}$, provided that $\theta$ satisfying the condition \eqref{theta1}.
\end{corollary}

\begin{corollary}
Let $m\geq2$ and $\vec{b}\in \mathrm{BMO}^m$. Assume that $p_k=1$, $1\leq\alpha_k<q_k<\infty$, $k=1,2,\ldots,m$ and $p=1/m$, $q\in[1,\infty)$ with $1/q=\sum_{k=1}^m 1/{q_k}$ and $1/{\alpha}=\sum_{k=1}^m 1/{\alpha_k};$ $\vec{w}=(w_1,\ldots,w_m)\in\prod_{k=1}^m A_1$ and $\mu\in\Delta_2$. In addition, assume that \eqref{supp} holds. Then for any given $\lambda>0$ and for any ball $B(y,r)$ with $y\in\mathbb R^n$ and $r>0$, there exists a constant $C>0$ independent of $\vec{f}=(f_1,\ldots,f_m)$, $B(y,r)$ and $\lambda$ such that $(\nu_{\vec{w}}=\prod_{k=1}^m w_k^{1/{m}})$
\begin{equation*}
\begin{split}
&\Big\|\nu_{\vec{w}}(B(y,r))^{1/{\alpha}-m-1/q}\cdot\Big[\nu_{\vec{w}}\Big(\Big\{x\in B(y,r):\big|\big[\Sigma\vec{b},T_\theta\big](\vec{f})(x)\big|>\lambda^m\Big\}\Big)\Big]^m\Big\|_{L^q(\mu)}\\
&\leq C\cdot\prod_{k=1}^m\bigg\|\Phi\bigg(\frac{|f_k|}{\lambda}\bigg)\bigg\|_{(L\log L,L^{q_k})^{\alpha_k}(w_k;\mu)},
\end{split}
\end{equation*}
provided that $\theta$ satisfies the condition \eqref{theta2}, and
\begin{equation*}
\begin{split}
&\Big\|\nu_{\vec{w}}(B(y,r))^{1/{\alpha}-m-1/q}\cdot\Big[\nu_{\vec{w}}\Big(\Big\{x\in B(y,r):\big|\big[\Pi\vec{b},T_\theta\big](\vec{f})(x)\big|>\lambda^m\Big\}\Big)\Big]^m\Big\|_{L^q(\mu)}\\
&\leq C\cdot\prod_{k=1}^m\bigg\|\Phi^{(m)}\bigg(\frac{|f_k|}{\lambda}\bigg)\bigg\|_{(L\log L,L^{q_k})^{\alpha_k}(w_k;\mu)},
\end{split}
\end{equation*}
provided that $\theta$ satisfies the condition \eqref{theta3}.
\end{corollary}

\subsection*{Acknowledgment}
This work was supported by the Natural Science Foundation of China (Grant No. XJEDU2020Y002 and 2022D01C407).

\end{document}